\documentclass[pdflatex,sn-mathphys-num]{sn-jnl}

\usepackage{hyperref}
\usepackage{mathtools}
\usepackage{setspace}
\usepackage{mathrsfs}
\usepackage{yfonts}
\usepackage[T1]{fontenc}
\usepackage{caption}
\usepackage{subcaption}
\usepackage{tikz}
\usepackage{pgfplots}
\usepackage{longtable}
\usepackage{siunitx}
\usepackage{enumitem}
\usepackage{dsfont}
\usepackage[section]{placeins}

\usetikzlibrary{shapes.misc}
\tikzset{cross/.style={cross out, draw=black, minimum size=2*(#1-\pgflinewidth), inner sep=0pt, outer sep=0pt},cross/.default={2pt}}

\usepackage{graphicx}%
\usepackage{multirow}%
\usepackage{amsmath,amssymb,amsfonts}%
\usepackage{amsthm}%
\usepackage{mathrsfs}%
\usepackage[title]{appendix}%
\usepackage{xcolor}%
\usepackage{textcomp}%
\usepackage{manyfoot}%
\usepackage{booktabs}%
\usepackage{algorithm}%
\usepackage{algorithmicx}%
\usepackage{algpseudocode}%
\usepackage{listings}%

\theoremstyle{thmstyleone}%
\newtheorem{theorem}{Theorem}
\newtheorem{proposition}{Proposition}%
\newtheorem{lemma}{Lemma}%
\newtheorem{corollary}{Corollary}

\theoremstyle{thmstyletwo}%
\newtheorem{example}{Example}%
\newtheorem{remark}{Remark}%

\theoremstyle{thmstylethree}%
\newtheorem{definition}{Definition}%

\begin{document}

\title[Article Title]{Lengthscale-informed sparse grids for kernel methods in high dimensions}


\author[1]{\fnm{Elliot J.} \sur{Addy}}\email{e.j.addy@sms.ed.ac.uk\footnote{ORCID-IDs: 0009-0003-2667-8493, 0000-0002-4600-0247, 0000-0001-5089-0476.}}

\author[2]{\fnm{Jonas} \sur{Latz}}

\author[1]{\fnm{Aretha L.} \sur{Teckentrup}}

\affil[1]{\orgdiv{School of Mathematics and Maxwell Institute for Mathematical Sciences}, \orgname{University of Edinburgh}, \orgaddress{\street{King's Buildings}, \city{Edinburgh}, \postcode{EH9 3FD}, \country{UK}}}

\affil[2]{\orgdiv{Department of Mathematics}, \orgname{University of Manchester}, \orgaddress{\street{Alan Turing Building, Oxford Road}, \city{Manchester}, \postcode{M13 9PL}, \country{UK}}}


\abstract{Kernel interpolation, especially in the context of Gaussian process emulation, is a widely used technique in surrogate modelling, where the goal is to cheaply approximate an input-output map using a limited number of function evaluations. However, in high-dimensional settings, such methods typically suffer from the \emph{curse of dimensionality}; the number of required evaluations to achieve a fixed approximation error grows exponentially with the input dimension. To overcome this, a common technique used in high-dimensional approximation methods, such as quasi-Monte Carlo and sparse grids, is to exploit functional anisotropy: the idea that some input dimensions are more `sensitive' than others. In doing so, such methods can significantly reduce the dimension dependence in the error. In this work, we propose a generalisation of sparse grid methods that incorporates a form of anisotropy encoded by the lengthscale parameter in Mat\'ern kernels. We derive error bounds and perform numerical experiments that show that our approach enables effective emulation over arbitrarily high dimensions for functions exhibiting sufficient anisotropy.
}




\maketitle

\section{Introduction}\label{sec: introduciton}

Kernel methods are widely used in surrogate modelling, where the goal is to cheaply approximate an input-output map using a limited number of function evaluations. Due to their mesh-free nature, kernel methods are especially well-suited to the low-data regime, where data acquisition poses the primary computational bottleneck, as well as scattered data approximation, where the location of function evaluations may be very irregular. Kernel methods are well known in the numerical analysis community, for example through approximation with radial basis functions \cite{wendland_2004,buhmann2000radial,nww05,nww06,adt12}, as well as in the statistics community, where the Gaussian process framework for regression is well established \cite{Rasmussen2005,gramacy2020surrogates,stein,kennedy2001bayesian,hkccr04}. 

In this work, we are interested in the accuracy of kernel methods as a function of the number of function evaluations (or \emph{training points}) used to construct the approximation. More precisely, we consider the problem of interpolation, where we wish to approximate a function $f$ from given evaluations \(\{(\mathbf{x},f(\mathbf{x}))\}_{\mathbf{x}\in\mathcal{X}}\)  at \(N\) distinct points. From a numerical analysis point of view, we then wish to bound the error $\|f - s_{\mathcal{X},\varphi}(f)\|$ between the true function $f$ and its interpolant $s_{\mathcal{X},\varphi}(f)$ built using the evaluations at $\mathcal{X}$ and kernel $\varphi$, in a suitably chosen norm. From a statistical point of view, the interpolant $s_{\mathcal{X},\varphi}(f)$ takes the role of the posterior mean in Gaussian process regression (see e.g. \cite{sss13}), and the former error analysis hence enables us to conclude on convergence of the methodology also in this framework.

The form of the error bounds, and the tools used to derive them, depend heavily on the kernel $\varphi$ employed in the approximation. Frequently used kernels in practice are the family of Mat\'ern kernels \cite{matern2013spatial,porcu2024matern}, since their parametric nature allows for flexibility to model various behaviours in terms of amplitude, length-scale and regularity. A further desirable feature of Mat\'ern kernels is that the corresponding interpolants $s_{\mathcal{X},\varphi}(f)$ lie in Sobolev spaces \cite{adams2003sobolev},  enabling the derivation of error bounds in the widely used Sobolev framework \cite{wendland_2004,nww05,van2011information}. In general, however, these error bounds unfortunately scale poorly with the number of inputs and suffer from the \emph{curse of dimensionality}. The number of function evaluations required to maintain a fixed error $\varepsilon$ grows exponentially with the input dimension, rendering the application of kernel methods to high-dimensional models notoriously challenging \cite{Santner2003,gramacy2020surrogates,GoldsteinVernon2010}.

To combat the challenge posed by a high input dimension, a central idea in high-dimensional approximation is to take advantage of anisotropy in the target function \(f\). If \(f\) is more difficult to approximate along certain axial directions than others, we can allocate computational effort preferentially to the more sensitive directions---placing fewer interpolation points in directions where \(f\) is relatively smooth and/or flat. Two well-known techniques in this direction are anisotropic sparse grids (for approximation and integration) and weighted quasi-Monte Carlo methods (for integration). To alleviate the dependence on the input dimension in the error bounds, anisotropic sparse grids require that the function $f$ is more regular in certain inputs, or equivalently that approximations in these input dimensions decay at faster rates \cite{Nobile2008, Nobile2018, Rieger2019}. Weighted quasi-Monte Carlo methods, on the other hand, assume that the partial first (and possibly higher) derivative of $f$ is much smaller in certain inputs, or equivalently that the variation of $f$ is small in these input dimensions \cite{sloan1998quasi,dick2013high}. Alternative approaches to kernel methods in high-dimensional input spaces are typically based on low-dimensional approximations, including for example ANOVA-type decompositions \cite{rieger2024approximability,rieger2024constructive}, sparsity assumptions \cite{yang2015minimax} or projection-based dimension reduction \cite{liu2017dimension}. 

In practice, the assumption that $f$ is more regular in some of its inputs, as required by anisotropic sparse grids, is often not satisfied or difficult to verify. This motivates the development of alternative formulations of anisotropy for kernel methods for approximation with sparse grids. One natural candidate is the lengthscale parameter \(\lambda\) in Mat\'ern kernels; larger values of \(\lambda\) correspond to flatter, more slowly varying interpolants, effectively assuming a reduced sensitivity to changes in the inputs. Lengthscale-anisotropic kernels, such as the anisotropic squared exponential kernel $\varphi(\mathbf{x},\mathbf{x'}) = \sigma^2 \exp^{-\sum_{j=1}^{d} (x_j-x_j')^2/2\lambda_j^2}$, are commonly used in the spatial statistics and machine learning literature, see e.g. \cite{rasmussen2004gaussian,vanderVaart2009,Bhattacharya2014} and the references therein. The lengthscales $\{\lambda_j\}_{j=1}^d$ are typically estimated using the observed values \(\{(\mathbf{x},f(\mathbf{x}))\}_{\mathbf{x}\in\mathcal{X}}\), and, in the same spirit as in this work, the estimated lengthscales are then frequently used in an automatic relevance determination (ARD) framework \cite{mackay1992bayesian,neal2012bayesian,wipf2007new} to assign varying levels of importance to the different inputs.  

From a theoretical point of view, the effect of lengthscale adaptivity and/or anisotropy in Gaussian process regression has notably been studied in the works \cite{vanderVaart2009,fasshauer2012dimension, Bhattacharya2014,Teckentrup2020} (see also the references therein). In \cite{Teckentrup2020}, the author studies the effect of adapted (or estimated) length-scales $\{\widehat \lambda_j\}_{j=1}^d$, dependent on the number of observed evaluations $N$, in Mat\'ern kernels on error bounds for the posterior mean and covariance. The authors of \cite{vanderVaart2009} show that for a suitable hyper-prior defined on $\lambda_j \equiv \lambda$, Gaussian process regression with the squared exponential kernel can adapt to the smoothness of the target $f$ and provide optimal convergence rates in a general setting. The effect of anisotropic lengthscales, again defined through a suitable hyper-prior, was studied in \cite{Bhattacharya2014}. We note, however, that the faster convergence rates proved in \cite{Bhattacharya2014} additionally rely on anisotropic regularity as was the case for anisotropic sparse grids discussed earlier. Finally, while the authors of \cite{fasshauer2012dimension} show that approximation with anisotropic squared exponential kernels can theoretically lead to dimension independent error bounds, their proof is not constructive and a choice of kernel function and observation points achieving these rates is to the best of our knowledge unknown.

In this work, our goal is to incorporate lengthscale-based anisotropy into the well-studied framework of kernel methods and sparse grid approximation.
In particular, we investigate how directional sensitivity, encoded via the lengthscales of separable Mat\'ern kernels, can be exploited to improve upon established error estimates and lead to convergence rates that exhibit dimension robust behaviour. To the best of our knowledge, this work is the first to develop a method with proven improved error bounds in the high-dimensional setting assuming only anisotropy in the lengthscales, not requiring anisotropic (or infinite) regularity. Furthermore, we would like to emphasise that our approach does not involve any form of dimension reduction or truncation; our method works in the full high-dimensional setting.
Finally, we note that the use of structured designs such as the introduced lengthscale-informed sparse grids enables fast linear algebra routines (see e.g. \cite{jagadeeswaran2019fast, Plumlee2014}), resulting in a method that scales well with the number of observed function evaluations as well as the input dimension.

\subsection*{Our contributions}
This paper makes the following key contributions:
\begin{enumerate}
    \item We introduce a novel sparse grid construction based on uniformly-spaced points, termed lengthscale-informed sparse grids (LISGs), designed to exploit lengthscale-anisotropy in target functions and resulting in sparse grids with controlled growth in the individual input dimensions. The construction is detailed in Definition \ref{def: penalised sparse grid operator}.
    \item We derive an error bound for lengthscale-anisotropic Mat\'ern kernel interpolation using LISGs
    in terms of the sparse grid level $L$ (see Theorem \ref{thm: error in L}), demonstrating its superior performance in high-dimensional settings. 
    We further include an explicit relationship between the sparse grid level $L$ and the number of points in the corresponding LISG, see Theorem \ref{thm: counting abscissae}.
    \item Using the relationship between kernel interpolation and Gaussian process emulation, we derive an analogous error bound to point 2. above for the posterior marginal variance in Proposition \ref{prop: vairance convergence}. (The convergence of the posterior mean follows directly from Theorem \ref{thm: error in L}.)
    \item We adapt the fast inference algorithm for isotropic sparse grids from \cite{Plumlee2014} to LISGs in Algorithm \ref{alg: 1}, resulting in an algorithm with significantly lower computational cost in terms of the input dimension $d$. In terms of the sparse grid level $L$, the cost of inverting the largest matrix in Algorithm \ref{alg: 1} is $\mathcal O(L 2^L)$.
    \item We demonstrate the dimension-robust performance of our method on test problems in up to 100 dimensions.
\end{enumerate}

\subsection*{Paper structure}

In Section~\ref{sec: background}, we provide the necessary background material and review relevant prior results. Sections~\ref{subsec: gaussian process emulation} and~\ref{subsec: approximating with kernels} cover Gaussian process emulation and its connections to kernel interpolation. Sections~\ref{subsec: sparse grids and dominating mixed smoothness} and~\ref{subsec: lengthscale anisotropy} introduce the regularity and anisotropy assumptions required of the target function. In Section~\ref{sec: results} we introduce lengthscale-informed sparse grids; defining their construction in Section~\ref{subsec: construction} and in Sections~\ref{subsec: results} and \ref{subsec: predictive variance} we present general error bounds for the approximation error and associated marginal variance, respectively. Full derivations of the main theorems are provided in Section~\ref{sec: error analysis}. In Section~\ref{sec:fast implementation} we derive an adapted algorithm for fast computations of interpolants, and in Section~\ref{sec: numerics} we present numerical results exploring the performance of our method on a test problem in high-dimensional settings under different anisotropy conditions.

\subsection*{Notation}
\renewcommand{\arraystretch}{1.2}
\begin{center}
\begin{longtable}{|c|p{0.8\textwidth}|}
    \hline
    Notation & Description\\
    \hline
    \(f\) & Target function, later assumed to be in the native space of interpolating kernels; \(\mathcal{N}_\varphi(\Omega)\).\\
    \(d\) & Dimension of the input space of \(f\).\\
    \(\mathcal{X}\), \(\Omega\) & Arbitrary finite point set and compact set where \(\mathcal{X}\subset\Omega\subset\mathbb{R}^d\).\\
    \(h_{\mathcal{X},\Omega}\) & The fill-distance of \(\mathcal{X}\) in \(\Omega\); \(\,\sup_{\mathbf{x}\in\Omega}\inf_{\mathbf{x}'\in\mathcal{X}}\|\mathbf{x}-\mathbf{x}'\|_{2}\).\\
    \(R_\Omega\) & Restriction operator to \(\Omega\); \(R_\Omega(g)=g|_{\Omega}\).\\
    \(\varphi\) & Arbitrary positive definite kernel function \(\varphi:\mathbb{R}\times\mathbb{R}\rightarrow\mathbb{R}\).\\
    \(m_{\mathcal{X}}^f\), \(\varphi_{\mathcal{X}}\) & Posterior mean and covariance functions with data \(D_{\mathcal{X}}=\{\mathbf{x},f(\mathbf{x})\}_{x\in\mathcal{X}}\). (See Section \ref{subsec: gaussian process emulation})\\
    \(s_{\mathcal{X},\varphi}\) & Kernel interpolation operator with interpolating kernel \(\varphi\) and data \(D_{\mathcal{X}}=\{\mathbf{x},f(\mathbf{x})\}_{x\in\mathcal{X}}\). Note \(s_{\mathcal{X},\varphi}(f)=m_{\mathcal{X}}^f\). (See Definition \ref{def:restricted kernel interpolant}) \\
    \(\mathcal{N}_{\varphi}(\Omega)\) & native space, or reproducing kernel Hilbert space, of the kernel \(\varphi\), restricted to \(\Omega\). (See Definition \ref{def: native space})\\
    \(\phi_{\nu,\lambda}\), \(\Phi_{\boldsymbol{\nu},\boldsymbol{\lambda}}\) & Mat\'ern and separable Mat\'ern kernel functions. (See Definitions \ref{def: 1D Matern} and \ref{def: separable matern})\\
    \(\nu\), \(\alpha\), \(\beta\) & Regularity parameters in Mat\'ern kernels and Sobolev spaces.\\
    \(\lambda\), \(\boldsymbol{\lambda}\) & Lengthscale parameter in \(\mathbb{R}_{\geq0}\) and lengthscale vector in \(\mathbb{R}_{\geq0}^d\) for Mat\'ern and separable Mat\'ern kernels, respectively.\\
    \(\sigma\) & Standard deviation parameter in Mat\'ern kernels.\\
    \(\Gamma\), \(\Gamma_k\) & The unit interval \(\left(-1/2,1/2\right)\) and the intervals \(\left(-2^{-k-1},2^{-k-1}\right)\), respectively. Note \(\Gamma_0=\Gamma\).\\
    \(L\) & Level parameter for sparse grids in \(\mathbb{N}_0\).\\
    \(p\), \(\mathbf{p}\) & Penalty parameter in \(\mathbb{N}_0\), and penalty vector in \(\mathbb{N}_0^d\).\\
    \(\mathcal{I}_L^d\) & Multi-index set \(\{\boldsymbol{l}\in\mathbb{N}_0^d\,:\,|\boldsymbol{l}|_1\leq L\}\).\\
    \(\mathcal{X}_l\), \(\mathcal{X}_l^p\) & One dimensional uniformly spaced point sets and \(p\)-penalised point sets in \(\Gamma\). Note \(\mathcal{X}_l^0 = \mathcal{X}_l\). (See Definitions \ref{def:chi_l} and \ref{def:chi_l^p})\\
    \(\mathcal{X}^{\otimes}_{d,L}\), \(\mathcal{X}^\otimes_{d,L,\mathbf{p}}\) & Standard and lengthscale-informed sparse grid designs in \(\Gamma^d\). (See Definition \ref{def: iso sparse grid})\\
    \(S_{L,\boldsymbol{\nu}}\), \(P_{L,\boldsymbol{\nu},\mathbf{p}}\) & Standard and lengthscale-informed sparse grid kernel interpolation operators. Note \(P_{L,\boldsymbol{\nu},\mathbf{0}}=S_{L,\boldsymbol{\nu}}\). (See Definition \ref{def: penalised sparse grid operator})\\
    \(N_{d,\mathbf{p}}(L)\) & Number of points in lengthscale-informed sparse grid; 
    \(|\mathcal{X}^\otimes_{d,L,\mathbf{p}}|\).\\
    \(\mathcal{S}_\lambda\), \(\mathcal{T}_\lambda\) & Function- and point- stretching operators, respectively. (See Definition \ref{def: stretching operators})\\
    \(\mathfrak{u}\), \(\mathfrak{v}\) & Subsets of \(\{1,\dots,d\}\) of size \(0\leq k\leq d\), denoting a \(k\)-dimensional subspace. For a given \(\mathbf{c}\in\mathbb{R}^d\), we define \(\mathbf{c}_{\mathfrak{u}}\coloneqq(c_{\mathfrak{u}_1},\dots,c_{\mathfrak{u}_k})\subset\mathbb{R}^k\), where \(\mathfrak{u}_j<\mathfrak{u}_{j+1}\).\\
    \(\mathcal{P}^d_k\) & Set of all subsets of \(\{1,\dots,d\}\) of size \(k\); \(\{\mathfrak{u}\subset\{1,\dots,d\}:|\mathfrak{u}|=k\}\).\\
    \hline
    
\end{longtable}
\end{center}

\section{Background}\label{sec: background}
In this section, we provide the relevant background material to define lengthscale-informed sparse grids, including: The wider context of Gaussian process emulation in Section \ref{subsec: gaussian process emulation}, the connection to kernel interpolation in Section \ref{subsec: approximating with kernels}, sparse grids for kernel methods in Section \ref{subsec: sparse grids and dominating mixed smoothness} and finally our formulation of lengthscale-anisotropy in Section \ref{subsec: lengthscale anisotropy}. 

\subsection{Gaussian process emulation}\label{subsec: gaussian process emulation}
The broad aim of emulation is to approximate a function \(f:\Omega\rightarrow \mathbb{R}\), where \(\Omega\subset\mathbb{R}^d\) is compact with Lipschitz boundary and \(d\in\mathbb{N}\), using only point-wise evaluations of \(f\). Gaussian process emulation approaches this problem in a Bayesian framework: We begin by placing a prior distribution on the unknown function, \(f\), modelling it as a Gaussian process,
\begin{align}
    f_0\sim\textrm{GP}(0,\varphi),\nonumber
\end{align}
where \(\varphi:\Omega\times\Omega\rightarrow\mathbb{R}\) is the prior covariance kernel in which prior knowledge about \(f\) is encoded, and we have chosen a zero prior mean. The choice of zero prior mean is for ease of presentation only, and our error bound in Theorem \ref{thm: error in L} can be extended to suitable non-zero prior means (satisfying the same assumptions as $f$), as shown in Theorem 3.5, \cite{Teckentrup2020}. Given a set of function evaluations \(D_{\mathcal{X}}=\{(\mathbf{x},f(\mathbf{x}))\}_{\mathbf{x}\in\mathcal{X}}\)  at \(N\) distinct points, \(\mathcal{X}=\{\mathbf{x}_i\}_{1\leq i\leq N}\subset\Omega\), we update the prior distribution with Bayes' rule. 
The posterior distribution is also a Gaussian process (see e.g. \cite{Rasmussen2005}),
\begin{align}
    f_0|D_{\mathcal{X}}\sim\textrm{GP}(m^f_{\mathcal{X}},\varphi_{\mathcal{X}}),\nonumber
\end{align}
with closed form expressions for the posterior mean function, \(m^f_{\mathcal{X}}:\Omega\rightarrow\mathbb{R}\), and posterior covariance kernel, \(\varphi_{\mathcal{X}}:\Omega \times \Omega \rightarrow\mathbb{R}\), given by
\begin{align}
    m_{\mathcal{X}}^f(\mathbf{x})&=\varphi(\mathbf{x},\mathcal{X})^T\varphi(\mathcal{X},\mathcal{X})^{-1}f(\mathcal{X}),\label{eq: mean func}\\
    \varphi_{\mathcal{X}}(\mathbf{x},\mathbf{x}')&=\varphi(\mathbf{x},\mathbf{x}')-\varphi(\mathbf{x},\mathcal{X})^T\varphi(\mathcal{X},\mathcal{X})^{-1}\varphi(\mathbf{x}',\mathcal{X}).
\end{align}
Here we define the vectors \(\varphi(\mathbf{x},\mathcal{X})\coloneqq[\varphi(\mathbf{x},\mathbf{x}_1),\dots,\varphi(\mathbf{x},\mathbf{x}_N)]\in\mathbb{R}^N\) and \(f(\mathcal{X})\coloneqq[f(\mathbf{x}_1),\dots,f(\mathbf{x}_N)]\in\mathbb{R}^N\), whereas \(\varphi(\mathcal{X},\mathcal{X})\in\mathbb{R}^{N\times N}\) is the covariance matrix such that \(\varphi(\mathcal{X},\mathcal{X})_{ik}=\varphi(\mathbf{x}_i,\mathbf{x}_k)\).
In general, this procedure is \(\mathcal{O}(N^3)\) as a consequence of having to invert (or solve linear systems with) \(\varphi(\mathcal{X},\mathcal{X})\). The posterior mean function, \(m_{\mathcal{X}}^f\), is often used in inference algorithms as an approximation for \(f\), and the posterior covariance, \(\varphi_{\mathcal{X}}\), can be considered a local measure of uncertainty in this approximation \cite{conrad2017statistical,bai2024gaussian}. Our focus is to provide error estimates for the posterior mean and covariance kernel that are insensitive to the input dimension, \(d\), by a judicious choice of the design, \(\mathcal{X}\), and covariance kernel, \(\varphi\).

\subsection{Approximating with kernels}\label{subsec: approximating with kernels}
To measure the approximation error of the mean function to \(f\), we draw from results in the scattered data approximation literature \cite{wendland_2004}. In particular, the posterior mean function of a Gaussian process, \(m_\mathcal{X}^f\), is a \textit{kernel interpolant} of \(f\), uniquely defined as the function of minimum norm in the \textit{native space} of the covariance kernel that interpolates \(f\) at \(\mathcal{X}\subseteq \Omega\).

\begin{definition}\label{def: native space}
    A Hilbert space \(H\) of functions \(f:\Omega\rightarrow\mathbb{R}\), with \(\Omega\subset\mathbb{R}^d\), is a \textit{native} or \emph{reproducing kernel Hilbert space} if point evaluation is a continuous linear functional. Each native space, \(H\), is associated to a unique function, \(\varphi:\Omega\times\Omega\rightarrow\mathbb{R}\), called the \textit{reproducing kernel} of \(H\), which satisfies the reproducing kernel property: For all \(f\in H\), \(x\in\Omega\), we have \(\langle f,\varphi(\cdot,x)\rangle_{H}=f(x)\). For a given kernel, \(\varphi\), we denote the corresponding native space by \(H=\mathcal{N}_{\varphi}(\Omega)\), with induced norm \(\|f\|_{\mathcal{N}_{\varphi}(\Omega)}=\langle f,f\rangle_H^{1/2}\).
\end{definition}
\begin{definition}\label{def:restricted kernel interpolant}
   Let \(f:\Omega\rightarrow\mathbb{R}\) be bounded and let \(\mathcal{X}=\{\mathbf{x}_1,\dots,\mathbf{x}_N\}\subset \Omega\) be a finite discrete set such that \(|\mathcal{X}|=N\). Define a linear sampling operator by \(T_{\mathcal{X}}(f)=f(\mathcal{X})\coloneqq[f(\mathbf{x}_1),\dots,f(\mathbf{x}_N)]\). The \(\varphi\)-\emph{kernel interpolant operator}, applied to \(f\), with respect to \(\mathcal{X}\), denoted \(s_{\mathcal{X},\varphi}(f)\in\mathcal{N}_{\varphi}(\Omega)\), is given by
   \begin{align}
       s_{\mathcal{X},\varphi}(f)\coloneqq\underset{\substack{g\in\mathcal{N}_{\varphi}(\Omega)\\T_{\mathcal{X}}(g)=T_{\mathcal{X}}(f)}}{\arg\min}\|g\|_{\mathcal{N}_{\varphi}(\Omega)}.\nonumber
   \end{align}
\end{definition}

\begin{proposition}[Section 6.2, \cite{Rasmussen2005}, Theorem 1, \cite{Scholkopf2001}]\label{prop: mean functions are kernel interpolants}
For all \(\mathbf{x}\in\Omega\), \(s_{\mathcal{X},\varphi}(f)(\mathbf{x}) = m_{\mathcal{X}}^f(\mathbf{x})\). 
\end{proposition}
\begin{proposition}[Chapter 10, \cite{wendland_2004}]\label{prop: native space norm}
The native space norm of a kernel interpolant is given by
\begin{align}
\|s_{\mathcal{X},\varphi}(f)\|_{\mathcal{N}_{\varphi}(\Omega)}&=\mathbf{w}^T\varphi(\mathcal{X},\mathcal{X})\mathbf{w},\label{eq: norm of interpolant}
\end{align}
where \(\mathbf{w}=\varphi(\mathcal{X},\mathcal{X})^{-1}f(\mathcal{X})\).
\end{proposition} 
We restrict ourselves to the case of Mat\'ern kernels; a popular class of positive definite kernels whose native spaces are known to be \textit{Sobolev spaces}, see e.g.\cite{wendland_2004,Novak2018}. 
\begin{definition}[\cite{Lord_Powell_Shardlow_2014}]\label{def: 1D Matern}
    Let \(d\in\mathbb{N}\). A \emph{Mat\'ern covariance kernel} is a symmetric, positive-definite kernel \(\phi_{\nu,\lambda}:\mathbb{R}^d\times\mathbb{R}^d\rightarrow\mathbb{R}_{>0}\), defined by
    \begin{align}\label{eq:1d_matern}
        \phi_{\nu,\lambda}(\mathbf{x},\mathbf{x}')\coloneqq\sigma^2\frac{2^{1-\nu}}{\Gamma(\nu)}\left(\sqrt{2\nu}\frac{\|\mathbf{x}-\mathbf{x}'\|_2}{\lambda}\right)^{\nu}K_{\nu}\left(\sqrt{2\nu}\frac{\|\mathbf{x}-\mathbf{x}'\|_2}{\lambda}\right),
    \end{align}
    for all \(\mathbf{x},\mathbf{x}'\in\mathbb{R}^d\), where \(\nu,\lambda,\sigma>0\) and \(K_\nu\) is the modified Bessel function of the second kind. 
\end{definition}
The three hyperparameters \(\nu\), \(\lambda\), and \(\sigma\) represent the smoothness, lengthscale and standard deviation of the kernels, respectively. The posterior mean function of a Gaussian process with covariance kernel \(\phi_{\nu,\lambda}\) does not depend on the standard deviation, \(\sigma\), and hence we suppress it in our notation, considering it a free parameter of concern only in the wider statistical framework. Definition \ref{def: 1D Matern} is much simplified for half-integer \(\nu\), with notable examples including the exponential kernel, \(\phi_{1/2,\lambda}(\mathbf{x}, \mathbf{x}')=\sigma^2\exp(-\|\mathbf{x}-\mathbf{x}'\|_2/\lambda)\) and, by taking the limit \(\nu\rightarrow\infty\), the Gaussian (or squared-exponential) kernel \(\phi_{\nu\rightarrow\infty,\lambda}=\sigma^2\exp(-\|\mathbf{x}-\mathbf{x}'\|_2^2/2\lambda^2)\).
\begin{proposition}[Corollary 10.13, \cite{wendland_2004}]\label{prop:mat_sobolev}
    The function spaces \(\mathcal{N}_{\phi_{\nu,\lambda}}(\mathbb{R}^d)\) and \(H^{\nu+d/2}(\mathbb{R}^d)\) are isomorphic, with equivalent norms. This further holds for restrictions to Lipschitz \(\Omega\subset\mathbb{R}^d\), specifically \(\mathcal{N}_{\phi_{\nu,\lambda}|_{\Omega\times\Omega}}(\Omega)\cong H^{\nu+d/2}(\Omega)\).
\end{proposition}
\begin{remark}
For ease of notation, we will denote \(\mathcal{N}_{\varphi}(\Omega) =\mathcal{N}_{\varphi|_{\Omega\times\Omega}}(\Omega)\) for kernels \(\varphi:\mathbb{R}^d\times\mathbb{R}^d\rightarrow\mathbb{R}\) and domains \(\Omega\subset\mathbb{R}^d\).
\end{remark}
\begin{proposition}[Theorem 10.47, \cite{wendland_2004}]
    Let \(g\in\mathcal{N}_{\phi_{\nu,\lambda}}(\mathbb{R}^d
    )\). Then, \(R_{\Omega}(g)\in\mathcal{N}_{\phi_{\nu,\lambda}}(\Omega)\), where \(R_{\Omega}(g)\) represents the usual restriction of functions \(g|_\Omega\).
\end{proposition}
We note that the Sobolev spaces \(\mathcal{N}_{\phi_{\nu,\lambda}}=H^{\nu+1/2}\) depend only on the smoothness parameter, \(\nu\). In particular, changing the lengthscale does not change the space. Upper bounds for the interpolation error in the Sobolev norm are then given in terms of the \textit{fill-distance} of the point-set in the domain;
\begin{align}
    h_{\mathcal{X},\Omega}\coloneqq\sup_{\mathbf{x}\in\Omega}\inf_{\mathbf{x}'\in\mathcal{X}}\|\mathbf{x}-\mathbf{x}'\|_{2}.\nonumber
\end{align}
    
\begin{proposition}(Theorem 3.2, \cite{adt12})\label{prop: initial wendland}
    Let \(d\in\mathbb{N}\), \( \beta>d/2\) and \(0 \leq \alpha \leq \beta\). Let \(\Omega\subset\mathbb{R}^d\) be bounded with Lipschitz boundary such that is satisfies an interior cone condition. For \(f\in H^{\beta}(\Omega)\), and \(h_{\mathcal{X},\Omega}\) sufficiently small, there exists a constant $C$, independent of $f$ and $N$, such that
    \begin{align}
        \left\|I-R_{\Omega}(s_{\mathcal{X},\phi_{\nu,\lambda}})\right\|_{H^{\beta}(\Omega)\rightarrow H^{\alpha}(\Omega)}\leq Ch^{\beta-\alpha}_{\mathcal{X},\Omega}.\nonumber
    \end{align}
\end{proposition}

We see that the error is governed by both the smoothness of $f$ and the fill-distance, which can be thought of as the largest distance any point in $\Omega$ can be from the points in \(\mathcal{X}\), including the boundary of \(\Omega\). In the best case, the fill-distance decays at a rate of \(N^{-1/d}\) in terms of the number of training points, see e.g.~\cite{Teckentrup2020} and the references therein. Here we are immediately confronted with the so called `curse of dimensionality;' if we increase the dimension, \(d\), we require exponentially many more points in \(\mathcal{X}\) in order to retain the same fill-distance, and hence the same approximation error. Approximation schemes that aim to reduce this dimension dependence assume additional structure on \(f\), with the general goals of (a), avoiding the \(d\)-dimensional 2-norm in the error bound, and (b), requiring fewer points in less `sensitive' dimensions. To address these, we introduce spaces of dominating mixed smoothness and function anisotropy, respectively.

\subsection{Sparse grids and Sobolev spaces of dominating mixed smoothness}\label{subsec: sparse grids and dominating mixed smoothness}
It is common in the high-dimensional approximation literature to work with functions with bounded mixed derivatives, see e.g.  \cite{Hackbush2012,dick2013high,Nobile2018}. Specifically, we will consider \(f\) to be in a \textit{Sobolev space of dominating mixed smoothness},
\begin{align}
    H^{\boldsymbol{\alpha}}_{\textrm{mix}}(\Omega) \coloneqq H^{\alpha_1}(\Gamma^{(1)})\otimes\cdots\otimes H^{\alpha_d}(\Gamma^{(d)}),\nonumber
\end{align}
where \(\boldsymbol{\alpha}\in\mathbb{N}_0^d\) and \(\Omega=\prod_{j=1}^d\Gamma^{(j)}\), with each \(\Gamma^{(j)}\) compact in \(\mathbb{R}\). The theory can extend to higher dimensional component domains, however we consider only one-dimensional \(\Gamma^{(j)}\) for ease of notation. Sparse grid methods are an application of the Smolyak algorithm to function approximation problems in such spaces, \cite{bungartz2004sparse,WASILKOWSKI19951, Nobile2008,rieger2017sampling}. In the Gaussian process emulation setting, this is simply equivalent to employing product covariance kernels on sparse grid designs. For Mat\'ern kernels, the corresponding product kernels are known as \textit{separable Mat\'ern kernels}.
\begin{definition}[See e.g.\ \cite{Teckentrup2020}]\label{def: separable matern}
    Define the hyperparameter vectors \(\boldsymbol{\nu}\coloneqq(\nu_1,\dots,\nu_d)\) and \(\boldsymbol{\lambda}\coloneqq(\lambda_1,\dots,\lambda_d)\), and let \(\{\phi_{\nu_j,\lambda_j}\}_{1\leq j\leq d}\) be a collection of \(d\)-many one dimensional Mat\'ern kernels, \(\phi_{\nu_j,\lambda_j}:\mathbb{R}\times\mathbb{R}\rightarrow\mathbb{R}_{>0}\). We define the \emph{separable Mat\'ern kernel}, \(\Phi_{\boldsymbol{\nu},\boldsymbol{\lambda}}:\mathbb{R}^d\times\mathbb{R}^d\rightarrow\mathbb{R}_{>0}\), by the tensor product, \(\Phi_{\boldsymbol{\nu},\boldsymbol{\lambda}}=\phi_{\nu_1,\lambda_1}\otimes\cdots\otimes\phi_{\nu_d,\lambda_d}\).
\end{definition}

\begin{proposition}\label{prop: separable native space isomorph}
\begin{enumerate}[label=(\roman*)]
    \item The native spaces of separable Mat\'ern kernels are the completion of the tensor product of the native spaces of the constituent one-dimensional Mat\'ern kernels, \(\mathcal{N}_{\Phi_{\boldsymbol{\nu},\boldsymbol{\lambda}}}(\Omega)=\mathcal{N}_{\phi_{\nu_1,\lambda_1}}(\Gamma^{(1)})\otimes\cdots\otimes\mathcal{N}_{\phi_{\nu_d,\lambda_d}}(\Gamma^{(d)})\) , under the induced norm,
\[
\left\langle f_1\otimes\cdots\otimes f_d,f_1\otimes\cdots\otimes f_d\right\rangle_{\mathcal{N}_{\Phi_{\boldsymbol{\nu},\boldsymbol{\lambda}}}(\Omega)}^{1/2} = \prod_{j=1}^d\langle f_j,f_j\rangle^{1/2}_{\mathcal{N}_{\phi_{\nu_j,\lambda_j}}\left(\Omega\right)},
\]
for \(f_j\in\mathcal{N}_{\phi_{\nu_j,\lambda_j}}\left(\Omega\right), 1\leq j\leq d\).
\item These spaces are isomorphic to Sobolev spaces with dominating mixed smoothness. In particular,
\[
\mathcal{N}_{\Phi_{\boldsymbol{\nu},\boldsymbol{\lambda}}}(\Omega)\cong H^{\boldsymbol{\nu}+1/2}_{\textrm{mix}}(\Omega^d)
,\]
with equivalent norms.
\end{enumerate}
\end{proposition}

\begin{proof}
As shown in Proposition 2, \cite{Nobile2018}, (i) follows from Theorem 1, Section 8, \cite{Aronszajn1950}, and (ii) follows from Theorem 10.12 \cite{wendland_2004}.
\end{proof}
We define a sparse grid design in terms of Cartesian products of one-dimensional point sets. Since the error is driven by the fill-distance, a natural choice are uniformly-spaced designs. Without loss of generality, we consider all component dimensions to be the unit interval \(\Gamma^{(j)}=\Gamma\coloneqq(-1/2, 1/2)\) for all \(1\leq j\leq d\), and consequently \(\Omega=\Gamma^d\).
\begin{definition}\label{def:chi_l}
Let \(l\in\mathbb{N}_0\) be given. We define \(\mathcal{X}_l\) to be the set of \(2^l-1\) uniformly spaced points in \(\Gamma\); \(\mathcal{X}_l \coloneqq\left\{n/2^{l+1}\in\Gamma\,:\,n\in\mathbb{Z}\right\}.\)
\end{definition}
\begin{remark}
    In order for the sparse grid approximation to also be an interpolant, and accordingly coincide with the posterior mean function of a Gaussian process, we require this sequence to be nested, i.e., \(\mathcal{X}_l\subset\mathcal{X}_{l+1}\) for all \(l\in\mathbb{N}_0\). 
\end{remark}
\begin{definition}[See e.g \cite{bungartz2004sparse}]\label{def: iso sparse grid}
    The \emph{sparse grid design}, \(\mathcal{X}^{\otimes}_{d,L}\subset\Gamma^d\), of \emph{construction level} \(L\in\mathbb{N}_0\) is given by
    \begin{align}
        \mathcal{X}_{d,L}^{\otimes}\coloneqq\bigcup_{\boldsymbol{l}\in\mathcal{I}_L^d}\mathcal{X}_{l_1}\times\cdots\times\mathcal{X}_{l_d},\nonumber
    \end{align}
    where \(\mathcal{I}_L^d\coloneqq\{\boldsymbol{l}\in\mathbb{N}_0^d\,:\,|\boldsymbol{l}|_1\leq L\}\).
\end{definition}
\begin{theorem}[Theorem 1, \cite{Nobile2018}, Theorem 3.11, \cite{Teckentrup2020}]\label{prop: nobile result}
    Let \(0\leq\alpha_j<\nu_j\) for all \(1\leq j\leq d\) and define the standard sparse grid kernel interpolant operator by
    \(S_{L,\boldsymbol{\nu},\boldsymbol{\lambda}}=s_{\mathcal{X}^{\otimes}_{d,L},\Phi_{\boldsymbol{\nu,\lambda}}}\). For large enough \(L\in\mathbb{N}_0\), there exists a constant \(C_{\ref{prop: nobile result}}\), independent of \(L\), such that
    \begin{align}
        \|I-S_{L,\boldsymbol{\nu},\boldsymbol{\lambda}}\|_{H^{\boldsymbol{\nu}+1/2}_{\textrm{mix}}(\Gamma^d)\rightarrow H^{\boldsymbol{\alpha}+1/2}_{\textrm{mix}}(\Gamma^d)}\leq C_{\ref{prop: nobile result}}N_{L,d}^{-b_{\min}}(\log N_{L,d})^{(1+b_{\min})(d-1)},
    \end{align}
    where \(N_{L, d}\coloneqq|\mathcal{X}^{\otimes}_{d,L}|\) and \(b_{\min}=\min_{1\leq j\leq d}\nu_j-\alpha_j\).
\end{theorem}

In contrast to the general Sobolev setting shown in Proposition \ref{prop: initial wendland}, the error in the number of training points exhibits dimension dependence only in the logarithmic factor, resulting in significantly improved scalability in high dimensions. Beyond \(d\approx10\), however, the logarithmic term becomes prohibitively large in many practical applications. As such, in order to apply these methods to arbitrarily high dimensional problems, we must assume additional structure on \(f\).

\subsection{Lengthscale anisotropy}\label{subsec: lengthscale anisotropy}
In order to enforce lengthscale assumptions on \(f\), we work directly in the native space norm of separable Mat\'ern kernels. We say that a function \(f\in H^{\boldsymbol{\nu}+1/2}_{\textrm{mix}}(\Gamma^d)\) exhibits lengthscale anisotropy with respect to \(\boldsymbol{\lambda}\in\mathbb{R}_{\geq0}^d\) if it is bounded in the corresponding native space norm,
\begin{align}
    \|f\|_{\mathcal{N}_{\Phi_{\boldsymbol{\nu},\boldsymbol{\lambda}}}(\Gamma^d)}<C.\label{eq: anisotropic condition}
\end{align}
This can be viewed as a weighted mixed Sobolev norm, where variation in \(f\) along the axes, \(j\), associated with larger lengthscales, \(\lambda_j\), has a stronger influence on the norm, so that boundedness enforces fewer effective degrees of freedom in those directions. In particular, in the high-dimensional setting, we are interested in functions whose norms are bounded independently of the dimension, \(d\). This is similar in spirit to weighted function spaces with product weights in the Quasi-Monte Carlo literature (see, e.g., \cite{sloan1998quasi,dick2013high}), in which the weights, \(\gamma_j\), directly multiply the gradient in dimension \(j\), and hence penalise large variation in these directions. Without loss of generality, for ease of presentation we will assume that the axes are ordered from most to least `sensitive,' i.e. \(\lambda_j\leq\lambda_{j+1}\) for all \(1\leq j\leq d-1\). In Figure~\ref{fig: anisotropic function}, we provide an illustration of such anisotropy with axial cross-sections of such a \(d\)-dimensional function. We note that identifying this ordering of dimensions can be a challenging task in high dimensions in the absence of prior knowledge \cite{Iooss2015}.

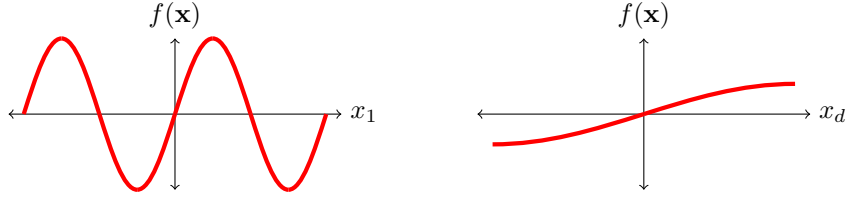
\begin{figure}[h]
    \centering
    \begin{tikzpicture}[scale = 1]
        \draw[<->] (-0.2,0) -- (4.2,0);
        \draw[<->] (2,-1) -- (2,1);
        \node at (2,1.3) {$f(\mathbf{x})$};
        \node at (4.5,0) {$x_1$};

        \draw[ultra thick, red] (0,0) sin (1/2,1);    
        \draw[ultra thick, red] (1/2,1) cos (2/2,0);    
        \draw[ultra thick, red] (2/2,0) sin (3/2,-1);    
        \draw[ultra thick, red] (3/2,-1) cos (4/2,0);    
        \draw[ultra thick, red] (4/2,0)  sin (5/2,1);    
        \draw[ultra thick, red] (5/2,1) cos (6/2,0);    
        \draw[ultra thick, red] (6/2,0) sin (7/2,-1);    
        \draw[ultra thick, red] (7/2,-1) cos (8/2,0);

        \draw[<->] (-0.2+6.2,0) -- (4.2+6.2,0);
        \draw[<->] (2+6.2,-1) -- (2+6.2,1);
        \node at (2+6.2,1.3) {$f(\mathbf{x})$};
        \node at (4.5+6.2,0) {$x_d$};

        \draw[ultra thick, red] (6.2,-0.4) cos (8.2,0);    
        \draw[ultra thick, red] (8.2,0)  sin (10.2,0.4);
    \end{tikzpicture}
    \vspace{3mm}
    \caption{Illustrated cross-sections of an anisotropic function satisfying \eqref{eq: anisotropic condition} in which dimensions are ordered such that the lengthscales grow, \(\lambda_j\leq\lambda_{j+1}\). We observe that the function exhibits much more variation in its first parameter, \(x_1\), than its last, \(x_d\).}
    \label{fig: anisotropic function}
\end{figure}

To fit within the sparse grid framework, we restrict ourselves to lengthscales of the form \(\lambda_j=2^{p_j}\), where we call \(p_j\in\mathbb{N}_0\) the \textit{penalty} associated to dimension \(j\). A \textit{penalty vector}, \(\mathbf{p}\in\mathbb{N}_0^d\), then completely specifies the lengthscale anisotropy of a given problem. While our theoretical results presented in Section~\ref{sec: results} apply to any degree of anisotropy, in order to obtain useful error bounds in high dimensions, we additionally require the penalties to be at least `eventually' increasing; i.e. if we consider the penalty as a sequence \(\{p_j\}_{j\in\mathbb{N}}\), then for all \(j\in\mathbb{N}\), there must exist some \(m\in\mathbb{N}\) such that \(p_{j+m}>p_j\). The complexity of a given problem is then directly linked to the growth of these penalties. In Section~\ref{sec: numerics}, we consider test problems with both linearly and logarithmically growing penalties, corresponding to exponentially and linearly growing lengthscales, respectively.

As discussed in Section~\ref{sec: introduciton}, although there is substantial empirical evidence supporting the benefits of employing kernels with anisotropic lengthscales, there remains a comparative lack of accompanying numerical analysis theory, especially in the absence of additional regularity assumptions. The framework proposed here offers a setting in which such theoretical advantages can be rigorously established. The results presented in Section~\ref{sec: results}, however, assume exact knowledge of the lengthscales \emph{a priori}. We note that this condition is unlikely to be satisfied in many practical applications, for example in our particular interest of Gaussian process emulation, where lengthscales are routinely estimated from data. Reassuringly, numerical results in  Section~\ref{sec: numerics} suggest that the method is robust to minor misspecifications of the penalties, however, its practical success in this setting remains dependent on the accuracy and reliability of the chosen estimation procedure. While we view an analysis into such techniques as being beyond the scope of this work, we point the reader to the substantial literature on the topic, for example \cite{Rasmussen2005, mardia1984, Sundararajan2001, hyperparamEstimationInverse}.
Estimation of lengthscales for separable Mat\'ern kernels in particular is considered in \cite{Briol2019}, and such kernels are widely used in surrogate modelling applications where they are often referred to as a form of automatic relevance determination (ARD) kernel \cite{geophysExampleARD,ZHANG2021100188}. Additionally, we believe that the formulation of the condition in \eqref{eq: anisotropic condition} opens the door to the development of novel estimation procedures in which, for a given target function, a natural aim is to identify the penalty vector corresponding to its minimal native space norm. Such methods may take advantage of the error bound presented in Theorem~\ref{thm: error in L} and the relative ease of computing native space norms of interpolants compared with more conventional Sobolev norms (see Proposition~\ref{prop: native space norm}).

\section{Lengthscale-informed sparse grids for kernel interpolation}\label{sec: results}
This section develops the theoretical foundations of our lengthscale-informed sparse grid interpolation framework. In Section~\ref{subsec: construction}, we formally introduce the construction of the interpolation operator. Section \ref{subsec: results} presents approximation error bounds in the native space and \(L^\infty\)-norms, including a discussion of the theoretical and practical benefits of incorporating lengthscale information in comparison to standard isotropic sparse grids. Finally, in Section~\ref{subsec: predictive variance}, we extend these results to bound the corresponding posterior marginal variance in the Gaussian process framework.

\subsection{Construction}\label{subsec: construction}
The construction of lengthscale-informed sparse grids diverges from their isotropic counterparts only in that the onset of the growth of points in each axial direction, \(1\leq j\leq d\), is delayed by the corresponding penalty, \(p_j\in\mathbb{N}_0\). In this way, we link the size of the sparse grid in each dimension to the lengthscale parameter, \(\lambda_j\).

\begin{figure}[h]
    \centering
    \begin{tikzpicture}
        \node[black] at (-2-3,-0.3) {\small-0.5};
        \node[black] at (2-3,-0.3) {\small0.5};
        \node[black] at (-2.8-3,0.75) {\(l\)};
        \node[black] at (-2.3-3,-0) {\footnotesize0};
        \node[black] at (-2.3-3,0.5) {\footnotesize1};
        \node[black] at (-2.3-3,1) {\footnotesize2};
        \node[black] at (-2.3-3,1.5) {\footnotesize3};
        \node[black] at (-2.3-3,2) {\footnotesize4};
        \node[black] at (0-3,0-0.6) {(a)};
        \node[black] at (0-3,0+2.6) {\large\(\mathcal{X}_l=\mathcal{X}_l^0\)};
        \draw[gray!40] (-2-3,0) -- (2-3,0);
        \draw[gray!40] (-2-3,0.5) -- (2-3,0.5);
        \draw[gray!40] (-2-3,1) -- (2-3,1);
        \draw[gray!40] (-2-3,1.5) -- (2-3,1.5);
        \draw[gray!40] (-2-3,2) -- (2-3,2);
        
        \filldraw[black] (0-3,0) circle (1pt);
        \filldraw[black] (0-3,0+0.5) circle (1pt);
        \filldraw[black] (0-1-3,0+0.5)  circle (1pt);
        \filldraw[black] (0+1-3,0+0.5) circle (1pt);
        \filldraw[black] (0-3,0+1) circle (1pt);
        \filldraw[black] (0-1-3,0+1) circle (1pt);
        \filldraw[black] (0+1-3,0+1) circle (1pt);
        \filldraw[black] (0+0.5-3,0+1) circle (1pt);
        \filldraw[black] (0-1+0.5-3,0+1) circle (1pt);
        \filldraw[black] (0+1+0.5-3,0+1) circle (1pt);
        \filldraw[black] (0-0.5-3,0+1) circle (1pt);
        \filldraw[black] (0-1-0.5-3,0+1) circle (1pt);
        \filldraw[black] (0+1-0.5-3,0+1) circle (1pt);
        \filldraw[black] (0-3,0+1.5) circle (1pt);
        \filldraw[black] (0-1-3,0+1.5) circle (1pt);
        \filldraw[black] (0+1-3,0+1.5) circle (1pt);
        \filldraw[black] (0+0.5-3,0+1.5) circle (1pt);
        \filldraw[black] (0-1+0.5-3,0+1.5) circle (1pt);
        \filldraw[black] (0+1+0.5-3,0+1.5) circle (1pt);
        \filldraw[black] (0-0.5-3,0+1.5) circle (1pt);
        \filldraw[black] (0-1-0.5-3,0+1.5) circle (1pt);
        \filldraw[black] (0+1-0.5-3,0+1.5) circle (1pt);
        \filldraw[black] (0-0.25-3,0+1.5) circle (1pt);
        \filldraw[black] (0-1-0.25-3,0+1.5) circle (1pt);
        \filldraw[black] (0+1-0.25-3,0+1.5) circle (1pt);
        \filldraw[black] (0+0.5-0.25-3,0+1.5) circle (1pt);
        \filldraw[black] (0-1+0.5-0.25-3,0+1.5) circle (1pt);
        \filldraw[black] (0+1+0.5-0.25-3,0+1.5) circle (1pt);
        \filldraw[black] (0-0.5-0.25-3,0+1.5) circle (1pt);
        \filldraw[black] (0-1-0.5-0.25-3,0+1.5) circle (1pt);
        \filldraw[black] (0+1-0.5-0.25-3,0+1.5) circle (1pt);
        \filldraw[black] (0+0.25-3,0+1.5) circle (1pt);
        \filldraw[black] (0-1+0.25-3,0+1.5) circle (1pt);
        \filldraw[black] (0+1+0.25-3,0+1.5) circle (1pt);
        \filldraw[black] (0+0.5+0.25-3,0+1.5) circle (1pt);
        \filldraw[black] (0-1+0.5+0.25-3,0+1.5) circle (1pt);
        \filldraw[black] (0+1+0.5+0.25-3,0+1.5) circle (1pt);
        \filldraw[black] (0-0.5+0.25-3,0+1.5) circle (1pt);
        \filldraw[black] (0-1-0.5+0.25-3,0+1.5) circle (1pt);
        \filldraw[black] (0+1-0.5+0.25-3,0+1.5) circle (1pt);
        \filldraw[black] (0-3,0+2) circle (1pt);
        \filldraw[black] (0-1-3,0+2) circle (1pt);
        \filldraw[black] (0+1-3,0+2) circle (1pt);
        \filldraw[black] (0+0.5-3,0+2) circle (1pt);
        \filldraw[black] (0-1+0.5-3,0+2) circle (1pt);
        \filldraw[black] (0+1+0.5-3,0+2) circle (1pt);
        \filldraw[black] (0-0.5-3,0+2) circle (1pt);
        \filldraw[black] (0-1-0.5-3,0+2) circle (1pt);
        \filldraw[black] (0+1-0.5-3,0+2) circle (1pt);
        \filldraw[black] (0-0.25-3,0+2) circle (1pt);
        \filldraw[black] (0-1-0.25-3,0+2) circle (1pt);
        \filldraw[black] (0+1-0.25-3,0+2) circle (1pt);
        \filldraw[black] (0+0.5-0.25-3,0+2) circle (1pt);
        \filldraw[black] (0-1+0.5-0.25-3,0+2) circle (1pt);
        \filldraw[black] (0+1+0.5-0.25-3,0+2) circle (1pt);
        \filldraw[black] (0-0.5-0.25-3,0+2) circle (1pt);
        \filldraw[black] (0-1-0.5-0.25-3,0+2) circle (1pt);
        \filldraw[black] (0+1-0.5-0.25-3,0+2) circle (1pt);
        \filldraw[black] (0+0.25-3,2) circle (1pt);
        \filldraw[black] (0-1+0.25-3,2) circle (1pt);
        \filldraw[black] (0+1+0.25-3,2) circle (1pt);
        \filldraw[black] (0+0.5+0.25-3,2) circle (1pt);
        \filldraw[black] (0-1+0.5+0.25-3,0+2) circle (1pt);
        \filldraw[black] (0+1+0.5+0.25-3,0+2) circle (1pt);
        \filldraw[black] (0-0.5+0.25-3,0+2) circle (1pt);
        \filldraw[black] (0-1-0.5+0.25-3,0+2) circle (1pt);
        \filldraw[black] (0+1-0.5+0.25-3,0+2) circle (1pt);
        \filldraw[black] (0-0.25-0.125-3,0+2) circle (1pt);
        \filldraw[black] (0-1-0.25-0.125-3,0+2) circle (1pt);
        \filldraw[black] (0+1-0.25-0.125-3,0+2) circle (1pt);
        \filldraw[black] (0+0.5-0.25-0.125-3,0+2) circle (1pt);
        \filldraw[black] (0-1+0.5-0.25-0.125-3,0+2) circle (1pt);
        \filldraw[black] (0+1+0.5-0.25-0.125-3,0+2) circle (1pt);
        \filldraw[black] (0-0.5-0.25-0.125-3,0+2) circle (1pt);
        \filldraw[black] (0-1-0.5-0.25-0.125-3,0+2) circle (1pt);
        \filldraw[black] (0+1-0.5-0.25-0.125-3,0+2) circle (1pt);
        \filldraw[black] (0+0.25-0.125-3,0+2) circle (1pt);
        \filldraw[black] (0-1+0.25-0.125-3,0+2) circle (1pt);
        \filldraw[black] (0+1+0.25-0.125-3,0+2) circle (1pt);
        \filldraw[black] (0+0.5+0.25-0.125-3,0+2) circle (1pt);
        \filldraw[black] (0-1+0.5+0.25-0.125-3,0+2) circle (1pt);
        \filldraw[black] (0+1+0.5+0.25-0.125-3,0+2) circle (1pt);
        \filldraw[black] (0-0.5+0.25-0.125-3,0+2) circle (1pt);
        \filldraw[black] (0-1-0.5+0.25-0.125-3,0+2) circle (1pt);
        \filldraw[black] (0+1-0.5+0.25-0.125-3,0+2) circle (1pt);
        \filldraw[black] (0-0.25+0.125-3,0+2) circle (1pt);
        \filldraw[black] (0-1-0.25+0.125-3,0+2) circle (1pt);
        \filldraw[black] (0+1-0.25+0.125-3,0+2) circle (1pt);
        \filldraw[black] (0+0.5-0.25+0.125-3,0+2) circle (1pt);
        \filldraw[black] (0-1+0.5-0.25+0.125-3,0+2) circle (1pt);
        \filldraw[black] (0+1+0.5-0.25+0.125-3,0+2) circle (1pt);
        \filldraw[black] (0-0.5-0.25+0.125-3,0+2) circle (1pt);
        \filldraw[black] (0-1-0.5-0.25+0.125-3,0+2) circle (1pt);
        \filldraw[black] (0+1-0.5-0.25+0.125-3,0+2) circle (1pt);
        \filldraw[black] (0+0.25+0.125-3,0+2) circle (1pt);
        \filldraw[black] (0-1+0.25+0.125-3,0+2) circle (1pt);
        \filldraw[black] (0+1+0.25+0.125-3,0+2) circle (1pt);
        \filldraw[black] (0+0.5+0.25+0.125-3,0+2) circle (1pt);
        \filldraw[black] (0-1+0.5+0.25+0.125-3,0+2) circle (1pt);
        \filldraw[black] (0+1+0.5+0.25+0.125-3,0+2) circle (1pt);
        \filldraw[black] (0-0.5+0.25+0.125-3,0+2) circle (1pt);
        \filldraw[black] (0-1-0.5+0.25+0.125-3,0+2) circle (1pt);
        \filldraw[black] (0+1-0.5+0.25+0.125-3,0+2) circle (1pt);
        
        \node[black] at (-2+3,-0.3) {\small-0.5};
        \node[black] at (2+3,-0.3) {\small0.5};
        \node[black] at (-2.8+3,0.75) {\(l\)};
        \node[black] at (-2.3+3,-0) {\footnotesize0};
        \node[black] at (-2.3+3,0.5) {\footnotesize1};
        \node[black] at (-2.3+3,1) {\footnotesize2};
        \node[black] at (-2.3+3,1.5) {\footnotesize3};
        \node[black] at (-2.3+3,2) {\footnotesize4};
        \node[black] at (0+3,0-0.6) {(b)};
        \node[black] at (0+3,0+2.6) {\large\(\mathcal{X}_l^2\)};
        \draw[gray!40] (-2+3,0) -- (2+3,0);
        \draw[gray!40] (-2+3,0.5) -- (2+3,0.5);
        \draw[gray!40] (-2+3,1) -- (2+3,1);
        \draw[gray!40] (-2+3,1.5) -- (2+3,1.5);
        \draw[gray!40] (-2+3,2) -- (2+3,2);
        
        \filldraw[black] (0+3,0) circle (1pt);
        \filldraw[black] (0+3,0.5) circle (1pt);
        \filldraw[black] (0+3,1) circle (1pt);
        \filldraw[black] (0+3,0+1.5) circle (1pt);

        \filldraw[black] (0-1+3,0+1.5) circle (1pt);
        \filldraw[black] (0+1+3,0+1.5) circle (1pt);
        \filldraw[black] (0+3,0+2) circle (1pt);
        \filldraw[black] (0-1+3,0+2) circle (1pt);
        \filldraw[black] (0+1+3,0+2) circle (1pt);
        \filldraw[black] (0+0.5+3,0+2) circle (1pt);
        \filldraw[black] (0-1+0.5+3,0+2) circle (1pt);
        \filldraw[black] (0+1+0.5+3,0+2) circle (1pt);
        \filldraw[black] (0-0.5+3,0+2) circle (1pt);
        \filldraw[black] (0-1-0.5+3,0+2) circle (1pt);
        \filldraw[black] (0+1-0.5+3,0+2) circle (1pt);

    \end{tikzpicture}
    \caption{The nested point-sets \(\mathcal{X}_l\), (a), and \(\mathcal{X}_l^p\) with penalty \(p=2\), (b), for different levels, \(l\in\mathbb{N}_0\). 
    }
    \label{fig: point sets}
\end{figure}
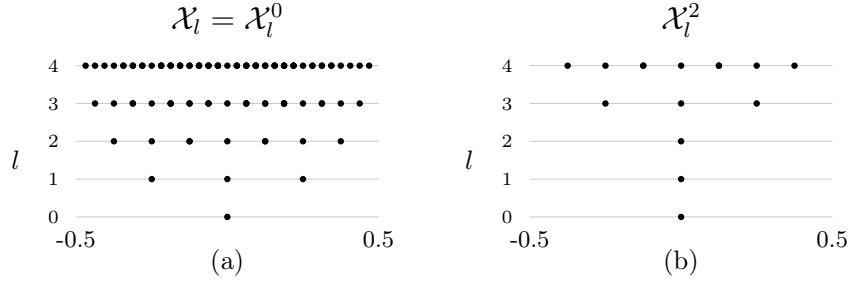

In contrast to the isotropic construction, the component one-dimensional point sets are now defined in terms of the penalty, depicted in Figure \ref{fig: point sets}.
\begin{definition}\label{def:chi_l^p}
    Let \(l,p\in\mathbb{N}_0\) be given. We define the \(p\)-\emph{penalised point-set}, \(\mathcal{X}_{l}^p\subset\Gamma\), by 
    \begin{align}
        \mathcal{X}_{l}^p\coloneqq\begin{cases}
            \mathcal{X}_{l-p}\quad&\textrm{ for }l\geq p+1, \textrm{and}\\
            \mathcal{X}_{0}=\{0\}&\textrm{ for }0\leq l\leq p.\end{cases}\nonumber
    \end{align}
\end{definition}
We note that this sequence retains a nested structure \(\mathcal{X}_l^p\subset\mathcal{X}_{l+1}^p\), and that, for the case \(p=0\), we recover our original point set \(\mathcal{X}_l^0= \mathcal{X}_l\). The lengthscale-informed sparse grid operator is defined in terms of tensor products of one dimensional interpolants defined over these designs, where the penalties are now described by a vector. 

\begin{figure}
\centering
\begin{subfigure}{.5\textwidth}
  \centering
  \begin{tikzpicture}[scale=0.75]
      \draw[->, gray, very thick] (0,0) -- (5.5,0);
      \draw[->, gray, very thick] (0,0) -- (0,5.5);

      \node at (3,4.8) {\Large\(\mathcal{X}_{l_1}^{1}\times\mathcal{X}_{l_2}^{2}\)};

      \node at (2.5,-0.9) {\Large\(l_1\)};
      \node at (-0.9,2.5) {\Large\(l_2\)};

      \node at (0.5,-0.3) {0};
      \node at (1.5,-0.3) {1};
      \node at (2.5,-0.3) {2};
      \node at (3.5,-0.3) {3};
      \node at (4.5,-0.3) {4};
      \node at (-0.3,0.5) {0};
      \node at (-0.3,1.5) {1};
      \node at (-0.3,2.5) {2};
      \node at (-0.3,3.5) {3};
      \node at (-0.3,4.5) {4};

      \draw[gray] (0,4) -- (1,4);
      \draw[gray] (0,3) -- (2,3);
      \draw[gray] (0,2) -- (3,2);
      \draw[gray] (0,1) -- (4,1);
      \draw[gray] (4,0) -- (4,1);
      \draw[gray] (3,0) -- (3,2);
      \draw[gray] (2,0) -- (2,3);
      \draw[gray] (1,0) -- (1,4);

      \draw[very thick] (0,5) -- (1,5) -- (1,4) -- (2,4) -- (2,3) -- (3,3) -- (3,2) -- (4,2) -- (4,1) -- (5,1) -- (5,0);

      \filldraw[black] (0+0.5,0+0.5) circle (0.5pt);
      \filldraw[black] (1+0.5,0+0.5) circle (0.5pt);
      \filldraw[black] (0+0.5,1+0.5) circle (0.5pt);
      \filldraw[black] (1+0.5,1+0.5) circle (0.5pt);
      \filldraw[black] (1+0.5,2+0.5) circle (0.5pt);
      \filldraw[black] (0+0.5,2+0.5) circle (0.5pt);

      %
      \filldraw[black] (2+0.5,0+0.5) circle (0.5pt);
      \filldraw[black] (2+0.5+0.25,0+0.5) circle (0.5pt);
      \filldraw[black] (2+0.5-0.25,0+0.5) circle (0.5pt);
      \filldraw[black] (2+0.5,1+0.5) circle (0.5pt);
      \filldraw[black] (2+0.5+0.25,1+0.5) circle (0.5pt);
      \filldraw[black] (2+0.5-0.25,1+0.5) circle (0.5pt);
      \filldraw[black] (2+0.5,2+0.5) circle (0.5pt);
      \filldraw[black] (2+0.5+0.25,2+0.5) circle (0.5pt);
      \filldraw[black] (2+0.5-0.25,2+0.5) circle (0.5pt);
      %
      \filldraw[black] (0+0.5,3+0.5) circle (0.5pt);
      \filldraw[black] (0+0.5,3+0.5+0.25) circle (0.5pt);
      \filldraw[black] (0+0.5,3+0.5-0.25) circle (0.5pt);
      \filldraw[black] (1+0.5,3+0.5) circle (0.5pt);
      \filldraw[black] (1+0.5,3+0.5+0.25) circle (0.5pt);
      \filldraw[black] (1+0.5,3+0.5-0.25) circle (0.5pt);

      %
      \filldraw[black] (3+0.5,0+0.5) circle (0.5pt);
      \filldraw[black] (3+0.5+0.125,0+0.5) circle (0.5pt);
      \filldraw[black] (3+0.5-0.125,0+0.5) circle (0.5pt);
      \filldraw[black] (3+0.5+0.25,0+0.5) circle (0.5pt);
      \filldraw[black] (3+0.5-0.25,0+0.5) circle (0.5pt);
      \filldraw[black] (3+0.5+0.25+0.125,0+0.5) circle (0.5pt);
      \filldraw[black] (3+0.5-0.25-0.125,0+0.5) circle (0.5pt);
      \filldraw[black] (3+0.5,1+0.5) circle (0.5pt);
      \filldraw[black] (3+0.5+0.125,1+0.5) circle (0.5pt);
      \filldraw[black] (3+0.5-0.125,1+0.5) circle (0.5pt);
      \filldraw[black] (3+0.5+0.25,1+0.5) circle (0.5pt);
      \filldraw[black] (3+0.5-0.25,1+0.5) circle (0.5pt);
      \filldraw[black] (3+0.5+0.25+0.125,1+0.5) circle (0.5pt);
      \filldraw[black] (3+0.5-0.25-0.125,1+0.5) circle (0.5pt);
      %
      %
      \filldraw[black] (0+0.5,4+0.5) circle (0.5pt);
      \filldraw[black] (0+0.5,4+0.5+0.125) circle (0.5pt);
      \filldraw[black] (0+0.5,4+0.5-0.125) circle (0.5pt);
      \filldraw[black] (0+0.5,4+0.5+0.25) circle (0.5pt);
      \filldraw[black] (0+0.5,4+0.5-0.25) circle (0.5pt);
      \filldraw[black] (0+0.5,4+0.5+0.25+0.125) circle (0.5pt);
      \filldraw[black] (0+0.5,4+0.5-0.25-0.125) circle (0.5pt);
      %
      \filldraw[black] (4+0.5,0+0.5) circle (0.5pt);
      \filldraw[black] (4+0.5+0.125,0+0.5) circle (0.5pt);
      \filldraw[black] (4+0.5+0.125+0.0625,0+0.5) circle (0.5pt);
      \filldraw[black] (4+0.5+0.125-0.0625,0+0.5) circle (0.5pt);
      \filldraw[black] (4+0.5-0.125,0+0.5) circle (0.5pt);
      \filldraw[black] (4+0.5-0.125+0.0625,0+0.5) circle (0.5pt);
      \filldraw[black] (4+0.5-0.125-0.0625,0+0.5) circle (0.5pt);
      \filldraw[black] (4+0.5+0.25,0+0.5) circle (0.5pt);
      \filldraw[black] (4+0.5-0.25,0+0.5) circle (0.5pt);
      \filldraw[black] (4+0.5+0.25+0.125,0+0.5) circle (0.5pt);
      \filldraw[black] (4+0.5+0.25+0.125+0.0625,0+0.5) circle (0.5pt);
      \filldraw[black] (4+0.5+0.25+0.125-0.0625,0+0.5) circle (0.5pt);
      \filldraw[black] (4+0.5-0.25-0.125,0+0.5) circle (0.5pt);
      \filldraw[black] (4+0.5-0.25-0.125+0.0625,0+0.5) circle (0.5pt);
      \filldraw[black] (4+0.5-0.25-0.125-0.0625,0+0.5) circle (0.5pt);
    \end{tikzpicture}
    \subcaption[]{\centering Lengthscale-informed sparse grid, \(L=4\), \(\mathbf{p}=(1,2)\).}
    \label{subfig: LISG component diagram}
\end{subfigure}%
\begin{subfigure}{.5\textwidth}
  \centering
  \begin{tikzpicture}[scale=0.75]
      \draw[->, gray, very thick] (0,0) -- (5.5,0);
      \draw[->, gray, very thick] (0,0) -- (0,5.5);

      \node at (3,4.8) {\Large\(\mathcal{X}_{l_1}^0\times\mathcal{X}_{l_2}^0\)};

      \node at (2.5,-0.9) {\Large\(l_1\)};
      \node at (-0.9,2.5) {\Large\(l_2\)};

      \node at (0.5,-0.3) {0};
      \node at (1.5,-0.3) {1};
      \node at (2.5,-0.3) {2};
      \node at (3.5,-0.3) {3};
      \node at (4.5,-0.3) {4};
      \node at (-0.3,0.5) {0};
      \node at (-0.3,1.5) {1};
      \node at (-0.3,2.5) {2};
      \node at (-0.3,3.5) {3};
      \node at (-0.3,4.5) {4};


      \draw[gray] (0,4) -- (1,4);
      \draw[gray] (0,3) -- (2,3);
      \draw[gray] (0,2) -- (3,2);
      \draw[gray] (0,1) -- (4,1);
      \draw[gray] (4,0) -- (4,1);
      \draw[gray] (3,0) -- (3,2);
      \draw[gray] (2,0) -- (2,3);
      \draw[gray] (1,0) -- (1,4);

      \draw[very thick, gray] (0,5) -- (1,5) -- (1,4) -- (2,4) -- (2,3) -- (3,3) -- (3,2) -- (4,2) -- (4,1) -- (5,1) -- (5,0);
      
      \filldraw[gray, fill opacity=0.2]  (0,5) -- (1,5) -- (1,4) -- (2,4) -- (2,3) -- (3,3) -- (3,2) -- (4,2) -- (4,1) -- (5,1) -- (5,0) -- (4,0) -- (4,1) -- (1,1) -- (1,3) -- (0,3) -- (0,5);

      \filldraw[black] (0+0.5,0+0.5) circle (0.5pt);

      \filldraw[black] (1+0.5,0+0.5) circle (0.5pt);
      \filldraw[black] (1+0.5+0.25,0+0.5) circle (0.5pt);
      \filldraw[black] (1+0.5-0.25,0+0.5) circle (0.5pt);
      \filldraw[black] (0+0.5,1+0.5) circle (0.5pt);
      \filldraw[black] (0+0.5,1+0.5+0.25) circle (0.5pt);
      \filldraw[black] (0+0.5,1+0.5-0.25) circle (0.5pt);

      \filldraw[black] (2+0.5,0+0.5) circle (0.5pt);
      \filldraw[black] (2+0.5+0.125,0+0.5) circle (0.5pt);
      \filldraw[black] (2+0.5-0.125,0+0.5) circle (0.5pt);
      \filldraw[black] (2+0.5+0.25,0+0.5) circle (0.5pt);
      \filldraw[black] (2+0.5-0.25,0+0.5) circle (0.5pt);
      \filldraw[black] (2+0.5+0.25+0.125,0+0.5) circle (0.5pt);
      \filldraw[black] (2+0.5-0.25-0.125,0+0.5) circle (0.5pt);
      \filldraw[black] (0+0.5,2+0.5) circle (0.5pt);
      \filldraw[black] (0+0.5,2+0.5+0.125) circle (0.5pt);
      \filldraw[black] (0+0.5,2+0.5-0.125) circle (0.5pt);
      \filldraw[black] (0+0.5,2+0.5+0.25) circle (0.5pt);
      \filldraw[black] (0+0.5,2+0.5-0.25) circle (0.5pt);
      \filldraw[black] (0+0.5,2+0.5+0.25+0.125) circle (0.5pt);
      \filldraw[black] (0+0.5,2+0.5-0.25-0.125) circle (0.5pt);

      \filldraw[black] (3+0.5,0+0.5) circle (0.5pt);
      \filldraw[black] (3+0.5+0.125,0+0.5) circle (0.5pt);
      \filldraw[black] (3+0.5+0.125+0.0625,0+0.5) circle (0.5pt);
      \filldraw[black] (3+0.5+0.125-0.0625,0+0.5) circle (0.5pt);
      \filldraw[black] (3+0.5-0.125,0+0.5) circle (0.5pt);
      \filldraw[black] (3+0.5-0.125+0.0625,0+0.5) circle (0.5pt);
      \filldraw[black] (3+0.5-0.125-0.0625,0+0.5) circle (0.5pt);
      \filldraw[black] (3+0.5+0.25,0+0.5) circle (0.5pt);
      \filldraw[black] (3+0.5-0.25,0+0.5) circle (0.5pt);
      \filldraw[black] (3+0.5+0.25+0.125,0+0.5) circle (0.5pt);
      \filldraw[black] (3+0.5+0.25+0.125+0.0625,0+0.5) circle (0.5pt);
      \filldraw[black] (3+0.5+0.25+0.125-0.0625,0+0.5) circle (0.5pt);
      \filldraw[black] (3+0.5-0.25-0.125,0+0.5) circle (0.5pt);
      \filldraw[black] (3+0.5-0.25-0.125+0.0625,0+0.5) circle (0.5pt);
      \filldraw[black] (3+0.5-0.25-0.125-0.0625,0+0.5) circle (0.5pt);
      \filldraw[gray] (0.5,3+0.5) circle (0.5pt);
      \filldraw[gray] (0.5,3+0.5+0.125) circle (0.5pt);
      \filldraw[gray] (0.5,3+0.5+0.125+0.0625) circle (0.5pt);
      \filldraw[gray] (0.5,3+0.5+0.125-0.0625) circle (0.5pt);
      \filldraw[gray] (0.5,3+0.5-0.125) circle (0.5pt);
      \filldraw[gray] (0.5,3+0.5-0.125+0.0625) circle (0.5pt);
      \filldraw[gray] (0.5,3+0.5-0.125-0.0625) circle (0.5pt);
      \filldraw[gray] (0.5,3+0.5+0.25) circle (0.5pt);
      \filldraw[gray] (0.5,3+0.5-0.25) circle (0.5pt);
      \filldraw[gray] (0.5,3+0.5+0.25+0.125) circle (0.5pt);
      \filldraw[gray] (0.5,3+0.5+0.25+0.125+0.0625) circle (0.5pt);
      \filldraw[gray] (0.5,3+0.5+0.25+0.125-0.0625) circle (0.5pt);
      \filldraw[gray] (0.5,3+0.5-0.25-0.125) circle (0.5pt);
      \filldraw[gray] (0.5,3+0.5-0.25-0.125+0.0625) circle (0.5pt);
      \filldraw[gray] (0.5,3+0.5-0.25-0.125-0.0625) circle (0.5pt);

      \filldraw[gray] (4+0.5,0+0.5) circle (0.5pt);
      \filldraw[gray] (4+0.5+0.125,0+0.5) circle (0.5pt);
      \filldraw[gray] (4+0.5+0.125+0.0625,0+0.5) circle (0.5pt);
      \filldraw[gray] (4+0.5+0.125+0.0625+0.03125,0+0.5) circle (0.5pt);
      \filldraw[gray] (4+0.5+0.125+0.0625-0.03125,0+0.5) circle (0.5pt);
      \filldraw[gray] (4+0.5+0.125-0.0625,0+0.5) circle (0.5pt);
      \filldraw[gray] (4+0.5+0.125-0.0625+0.03125,0+0.5) circle (0.5pt);
      \filldraw[gray] (4+0.5+0.125-0.0625-0.03125,0+0.5) circle (0.5pt);
      \filldraw[gray] (4+0.5-0.125,0+0.5) circle (0.5pt);
      \filldraw[gray] (4+0.5-0.125+0.0625,0+0.5) circle (0.5pt);
      \filldraw[gray] (4+0.5-0.125+0.0625+0.03125,0+0.5) circle (0.5pt);
      \filldraw[gray] (4+0.5-0.125+0.0625-0.03125,0+0.5) circle (0.5pt);
      \filldraw[gray] (4+0.5-0.125-0.0625,0+0.5) circle (0.5pt);
      \filldraw[gray] (4+0.5-0.125-0.0625+0.03125,0+0.5) circle (0.5pt);
      \filldraw[gray] (4+0.5-0.125-0.0625-0.03125,0+0.5) circle (0.5pt);
      \filldraw[gray] (4+0.5+0.25,0+0.5) circle (0.5pt);
      \filldraw[gray] (4+0.5-0.25,0+0.5) circle (0.5pt);
      \filldraw[gray] (4+0.5+0.25+0.125,0+0.5) circle (0.5pt);
      \filldraw[gray] (4+0.5+0.25+0.125+0.0625,0+0.5) circle (0.5pt);
      \filldraw[gray] (4+0.5+0.25+0.125+0.0625+0.03125,0+0.5) circle (0.5pt);
      \filldraw[gray] (4+0.5+0.25+0.125+0.0625-0.03125,0+0.5) circle (0.5pt);
      \filldraw[gray] (4+0.5+0.25+0.125-0.0625,0+0.5) circle (0.5pt);
      \filldraw[gray] (4+0.5+0.25+0.125-0.0625+0.03125,0+0.5) circle (0.5pt);
      \filldraw[gray] (4+0.5+0.25+0.125-0.0625-0.03125,0+0.5) circle (0.5pt);
      \filldraw[gray] (4+0.5-0.25-0.125,0+0.5) circle (0.5pt);
      \filldraw[gray] (4+0.5-0.25-0.125+0.0625,0+0.5) circle (0.5pt);
      \filldraw[gray] (4+0.5-0.25-0.125+0.0625+0.03125,0+0.5) circle (0.5pt);
      \filldraw[gray] (4+0.5-0.25-0.125+0.0625-0.03125,0+0.5) circle (0.5pt);
      \filldraw[gray] (4+0.5-0.25-0.125-0.0625,0+0.5) circle (0.5pt);
      \filldraw[gray] (4+0.5-0.25-0.125-0.0625+0.03125,0+0.5) circle (0.5pt);
      \filldraw[gray] (4+0.5-0.25-0.125-0.0625-0.03125,0+0.5) circle (0.5pt);
      \filldraw[gray] (0+0.5,4+0.5) circle (0.5pt);
      \filldraw[gray] (0+0.5,4+0.5+0.125) circle (0.5pt);
      \filldraw[gray] (0+0.5,4+0.5+0.125+0.0625) circle (0.5pt);
      \filldraw[gray] (0+0.5,4+0.5+0.125+0.0625+0.03125) circle (0.5pt);
      \filldraw[gray] (0+0.5,4+0.5+0.125+0.0625-0.03125) circle (0.5pt);
      \filldraw[gray] (0+0.5,4+0.5+0.125-0.0625) circle (0.5pt);
      \filldraw[gray] (0+0.5,4+0.5+0.125-0.0625+0.03125) circle (0.5pt);
      \filldraw[gray] (0+0.5,4+0.5+0.125-0.0625-0.03125) circle (0.5pt);
      \filldraw[gray] (0+0.5,4+0.5-0.125) circle (0.5pt);
      \filldraw[gray] (0+0.5,4+0.5-0.125+0.0625) circle (0.5pt);
      \filldraw[gray] (0+0.5,4+0.5-0.125+0.0625+0.03125) circle (0.5pt);
      \filldraw[gray] (0+0.5,4+0.5-0.125+0.0625-0.03125) circle (0.5pt);
      \filldraw[gray] (0+0.5,4+0.5-0.125-0.0625) circle (0.5pt);
      \filldraw[gray] (0+0.5,4+0.5-0.125-0.0625+0.03125) circle (0.5pt);
      \filldraw[gray] (0+0.5,4+0.5-0.125-0.0625-0.03125) circle (0.5pt);
      \filldraw[gray] (0+0.5,4+0.5+0.25) circle (0.5pt);
      \filldraw[gray] (0+0.5,4+0.5-0.25) circle (0.5pt);
      \filldraw[gray] (0+0.5,4+0.5+0.25+0.125) circle (0.5pt);
      \filldraw[gray] (0+0.5,4+0.5+0.25+0.125+0.0625) circle (0.5pt);
      \filldraw[gray] (0+0.5,4+0.5+0.25+0.125+0.0625+0.03125) circle (0.5pt);
      \filldraw[gray] (0+0.5,4+0.5+0.25+0.125+0.0625-0.03125) circle (0.5pt);
      \filldraw[gray] (0+0.5,4+0.5+0.25+0.125-0.0625) circle (0.5pt);
      \filldraw[gray] (0+0.5,4+0.5+0.25+0.125-0.0625+0.03125) circle (0.5pt);
      \filldraw[gray] (0+0.5,4+0.5+0.25+0.125-0.0625-0.03125) circle (0.5pt);
      \filldraw[gray] (0+0.5,4+0.5-0.25-0.125) circle (0.5pt);
      \filldraw[gray] (0+0.5,4+0.5-0.25-0.125+0.0625) circle (0.5pt);
      \filldraw[gray] (0+0.5,4+0.5-0.25-0.125+0.0625+0.03125) circle (0.5pt);
      \filldraw[gray] (0+0.5,4+0.5-0.25-0.125+0.0625-0.03125) circle (0.5pt);
      \filldraw[gray] (0+0.5,4+0.5-0.25-0.125-0.0625) circle (0.5pt);
      \filldraw[gray] (0+0.5,4+0.5-0.25-0.125-0.0625+0.03125) circle (0.5pt);
      \filldraw[gray] (0+0.5,4+0.5-0.25-0.125-0.0625-0.03125) circle (0.5pt);
      
      \filldraw[gray] (1+0.5,1+0.5) circle (0.5pt);
      \filldraw[gray] (1+0.5+0.25,1+0.5) circle (0.5pt);
      \filldraw[gray] (1+0.5-0.25,1+0.5) circle (0.5pt);
      \filldraw[gray] (1+0.5,1+0.5+0.25) circle (0.5pt);
      \filldraw[gray] (1+0.5+0.25,1+0.5+0.25) circle (0.5pt);
      \filldraw[gray] (1+0.5-0.25,1+0.5+0.25) circle (0.5pt);
      \filldraw[gray] (1+0.5,1+0.5-0.25) circle (0.5pt);
      \filldraw[gray] (1+0.5+0.25,1+0.5-0.25) circle (0.5pt);
      \filldraw[gray] (1+0.5-0.25,1+0.5-0.25) circle (0.5pt);
      
      \filldraw[gray] (2+0.5,1+0.5) circle (0.5pt);
      \filldraw[gray] (2+0.5+0.25,1+0.5) circle (0.5pt);
      \filldraw[gray] (2+0.5+0.25+0.125,1+0.5) circle (0.5pt);
      \filldraw[gray] (2+0.5+0.25-0.125,1+0.5) circle (0.5pt);
      \filldraw[gray] (2+0.5-0.25,1+0.5) circle (0.5pt);
      \filldraw[gray] (2+0.5-0.25+0.125,1+0.5) circle (0.5pt);
      \filldraw[gray] (2+0.5-0.25-0.125,1+0.5) circle (0.5pt);
      \filldraw[gray] (2+0.5,1+0.5+0.25) circle (0.5pt);
      \filldraw[gray] (2+0.5+0.25,1+0.5+0.25) circle (0.5pt);
      \filldraw[gray] (2+0.5+0.25+0.125,1+0.5+0.25) circle (0.5pt);
      \filldraw[gray] (2+0.5+0.25-0.125,1+0.5+0.25) circle (0.5pt);
      \filldraw[gray] (2+0.5-0.25,1+0.5+0.25) circle (0.5pt);
      \filldraw[gray] (2+0.5-0.25+0.125,1+0.5+0.25) circle (0.5pt);
      \filldraw[gray] (2+0.5-0.25-0.125,1+0.5+0.25) circle (0.5pt);
      \filldraw[gray] (2+0.5,1+0.5-0.25) circle (0.5pt);
      \filldraw[gray] (2+0.5+0.25,1+0.5-0.25) circle (0.5pt);
      \filldraw[gray] (2+0.5+0.25+0.125,1+0.5-0.25) circle (0.5pt);
      \filldraw[gray] (2+0.5+0.25-0.125,1+0.5-0.25) circle (0.5pt);
      \filldraw[gray] (2+0.5-0.25,1+0.5-0.25) circle (0.5pt);
      \filldraw[gray] (2+0.5-0.25+0.125,1+0.5-0.25) circle (0.5pt);
      \filldraw[gray] (2+0.5-0.25-0.125,1+0.5-0.25) circle (0.5pt);
      \filldraw[gray] (1+0.5,2+0.5) circle (0.5pt);
      \filldraw[gray] (1+0.5+0.25,2+0.5) circle (0.5pt);
      \filldraw[gray] (1+0.5-0.25,2+0.5) circle (0.5pt);
      \filldraw[gray] (1+0.5,2+0.5+0.25) circle (0.5pt);
      \filldraw[gray] (1+0.5+0.25,2+0.5+0.25) circle (0.5pt);
      \filldraw[gray] (1+0.5-0.25,2+0.5+0.25) circle (0.5pt);
      \filldraw[gray] (1+0.5,2+0.5+0.25+0.125) circle (0.5pt);
      \filldraw[gray] (1+0.5+0.25,2+0.5+0.25+0.125) circle (0.5pt);
      \filldraw[gray] (1+0.5-0.25,2+0.5+0.25+0.125) circle (0.5pt);
      \filldraw[gray] (1+0.5,2+0.5+0.25-0.125) circle (0.5pt);
      \filldraw[gray] (1+0.5+0.25,2+0.5+0.25-0.125) circle (0.5pt);
      \filldraw[gray] (1+0.5-0.25,2+0.5+0.25-0.125) circle (0.5pt);
      \filldraw[gray] (1+0.5,2+0.5-0.25) circle (0.5pt);
      \filldraw[gray] (1+0.5+0.25,2+0.5-0.25) circle (0.5pt);
      \filldraw[gray] (1+0.5-0.25,2+0.5-0.25) circle (0.5pt);
      \filldraw[gray] (1+0.5,2+0.5-0.25+0.125) circle (0.5pt);
      \filldraw[gray] (1+0.5+0.25,2+0.5-0.25+0.125) circle (0.5pt);
      \filldraw[gray] (1+0.5-0.25,2+0.5-0.25+0.125) circle (0.5pt);
      \filldraw[gray] (1+0.5,2+0.5-0.25-0.125) circle (0.5pt);
      \filldraw[gray] (1+0.5+0.25,2+0.5-0.25-0.125) circle (0.5pt);
      \filldraw[gray] (1+0.5-0.25,2+0.5-0.25-0.125) circle (0.5pt);
      
      \filldraw[gray] (2+0.5,2+0.5) circle (0.5pt);
      \filldraw[gray] (2+0.5+0.25,2+0.5) circle (0.5pt);
      \filldraw[gray] (2+0.5+0.25+0.125,2+0.5) circle (0.5pt);
      \filldraw[gray] (2+0.5+0.25-0.125,2+0.5) circle (0.5pt);
      \filldraw[gray] (2+0.5-0.25,2+0.5) circle (0.5pt);
      \filldraw[gray] (2+0.5-0.25+0.125,2+0.5) circle (0.5pt);
      \filldraw[gray] (2+0.5-0.25-0.125,2+0.5) circle (0.5pt);
      \filldraw[gray] (2+0.5,2+0.5+0.25) circle (0.5pt);
      \filldraw[gray] (2+0.5+0.25,2+0.5+0.25) circle (0.5pt);
      \filldraw[gray] (2+0.5+0.25+0.125,2+0.5+0.25) circle (0.5pt);
      \filldraw[gray] (2+0.5+0.25-0.125,2+0.5+0.25) circle (0.5pt);
      \filldraw[gray] (2+0.5-0.25,2+0.5+0.25) circle (0.5pt);
      \filldraw[gray] (2+0.5-0.25+0.125,2+0.5+0.25) circle (0.5pt);
      \filldraw[gray] (2+0.5-0.25-0.125,2+0.5+0.25) circle (0.5pt);
      \filldraw[gray] (2+0.5,2+0.5+0.25+0.125) circle (0.5pt);
      \filldraw[gray] (2+0.5+0.25,2+0.5+0.25+0.125) circle (0.5pt);
      \filldraw[gray] (2+0.5+0.25+0.125,2+0.5+0.25+0.125) circle (0.5pt);
      \filldraw[gray] (2+0.5+0.25-0.125,2+0.5+0.25+0.125) circle (0.5pt);
      \filldraw[gray] (2+0.5-0.25,2+0.5+0.25+0.125) circle (0.5pt);
      \filldraw[gray] (2+0.5-0.25+0.125,2+0.5+0.25+0.125) circle (0.5pt);
      \filldraw[gray] (2+0.5-0.25-0.125,2+0.5+0.25+0.125) circle (0.5pt);
      \filldraw[gray] (2+0.5,2+0.5+0.25-0.125) circle (0.5pt);
      \filldraw[gray] (2+0.5+0.25,2+0.5+0.25-0.125) circle (0.5pt);
      \filldraw[gray] (2+0.5+0.25+0.125,2+0.5+0.25-0.125) circle (0.5pt);
      \filldraw[gray] (2+0.5+0.25-0.125,2+0.5+0.25-0.125) circle (0.5pt);
      \filldraw[gray] (2+0.5-0.25,2+0.5+0.25-0.125) circle (0.5pt);
      \filldraw[gray] (2+0.5-0.25+0.125,2+0.5+0.25-0.125) circle (0.5pt);
      \filldraw[gray] (2+0.5-0.25-0.125,2+0.5+0.25-0.125) circle (0.5pt);
      \filldraw[gray] (2+0.5,2+0.5-0.25) circle (0.5pt);
      \filldraw[gray] (2+0.5+0.25,2+0.5-0.25) circle (0.5pt);
      \filldraw[gray] (2+0.5+0.25+0.125,2+0.5-0.25) circle (0.5pt);
      \filldraw[gray] (2+0.5+0.25-0.125,2+0.5-0.25) circle (0.5pt);
      \filldraw[gray] (2+0.5-0.25,2+0.5-0.25) circle (0.5pt);
      \filldraw[gray] (2+0.5-0.25+0.125,2+0.5-0.25) circle (0.5pt);
      \filldraw[gray] (2+0.5-0.25-0.125,2+0.5-0.25) circle (0.5pt);
      \filldraw[gray] (2+0.5,2+0.5-0.25+0.125) circle (0.5pt);
      \filldraw[gray] (2+0.5+0.25,2+0.5-0.25+0.125) circle (0.5pt);
      \filldraw[gray] (2+0.5+0.25+0.125,2+0.5-0.25+0.125) circle (0.5pt);
      \filldraw[gray] (2+0.5+0.25-0.125,2+0.5-0.25+0.125) circle (0.5pt);
      \filldraw[gray] (2+0.5-0.25,2+0.5-0.25+0.125) circle (0.5pt);
      \filldraw[gray] (2+0.5-0.25+0.125,2+0.5-0.25+0.125) circle (0.5pt);
      \filldraw[gray] (2+0.5-0.25-0.125,2+0.5-0.25+0.125) circle (0.5pt);
      \filldraw[gray] (2+0.5,2+0.5-0.25-0.125) circle (0.5pt);
      \filldraw[gray] (2+0.5+0.25,2+0.5-0.25-0.125) circle (0.5pt);
      \filldraw[gray] (2+0.5+0.25+0.125,2+0.5-0.25-0.125) circle (0.5pt);
      \filldraw[gray] (2+0.5+0.25-0.125,2+0.5-0.25-0.125) circle (0.5pt);
      \filldraw[gray] (2+0.5-0.25,2+0.5-0.25-0.125) circle (0.5pt);
      \filldraw[gray] (2+0.5-0.25+0.125,2+0.5-0.25-0.125) circle (0.5pt);
      \filldraw[gray] (2+0.5-0.25-0.125,2+0.5-0.25-0.125) circle (0.5pt);
      
      \filldraw[gray] (3+0.5,1+0.5) circle (0.5pt);
      \filldraw[gray] (3+0.5+0.25,1+0.5) circle (0.5pt);
      \filldraw[gray] (3+0.5+0.25+0.125,1+0.5) circle (0.5pt);
      \filldraw[gray] (3+0.5+0.25+0.125+0.0625,1+0.5) circle (0.5pt);
      \filldraw[gray] (3+0.5+0.25+0.125-0.0625,1+0.5) circle (0.5pt);
      \filldraw[gray] (3+0.5+0.25-0.125,1+0.5) circle (0.5pt);
      \filldraw[gray] (3+0.5+0.25-0.125+0.0625,1+0.5) circle (0.5pt);
      \filldraw[gray] (3+0.5+0.25-0.125-0.0625,1+0.5) circle (0.5pt);
      \filldraw[gray] (3+0.5-0.25,1+0.5) circle (0.5pt);
      \filldraw[gray] (3+0.5-0.25+0.125,1+0.5) circle (0.5pt);
      \filldraw[gray] (3+0.5-0.25+0.125+0.0625,1+0.5) circle (0.5pt);
      \filldraw[gray] (3+0.5-0.25+0.125-0.0625,1+0.5) circle (0.5pt);
      \filldraw[gray] (3+0.5-0.25-0.125,1+0.5) circle (0.5pt);
      \filldraw[gray] (3+0.5-0.25-0.125+0.0625,1+0.5) circle (0.5pt);
      \filldraw[gray] (3+0.5-0.25-0.125-0.0625,1+0.5) circle (0.5pt);
      \filldraw[gray] (3+0.5,1+0.5+0.25) circle (0.5pt);
      \filldraw[gray] (3+0.5+0.25,1+0.5+0.25) circle (0.5pt);
      \filldraw[gray] (3+0.5+0.25+0.125,1+0.5+0.25) circle (0.5pt);
      \filldraw[gray] (3+0.5+0.25+0.125+0.0625,1+0.5+0.25) circle (0.5pt);
      \filldraw[gray] (3+0.5+0.25+0.125-0.0625,1+0.5+0.25) circle (0.5pt);
      \filldraw[gray] (3+0.5+0.25-0.125,1+0.5+0.25) circle (0.5pt);
      \filldraw[gray] (3+0.5+0.25-0.125+0.0625,1+0.5+0.25) circle (0.5pt);
      \filldraw[gray] (3+0.5+0.25-0.125-0.0625,1+0.5+0.25) circle (0.5pt);
      \filldraw[gray] (3+0.5-0.25,1+0.5+0.25) circle (0.5pt);
      \filldraw[gray] (3+0.5-0.25+0.125,1+0.5+0.25) circle (0.5pt);
      \filldraw[gray] (3+0.5-0.25+0.125+0.0625,1+0.5+0.25) circle (0.5pt);
      \filldraw[gray] (3+0.5-0.25+0.125-0.0625,1+0.5+0.25) circle (0.5pt);
      \filldraw[gray] (3+0.5-0.25-0.125,1+0.5+0.25) circle (0.5pt);
      \filldraw[gray] (3+0.5-0.25-0.125+0.0625,1+0.5+0.25) circle (0.5pt);
      \filldraw[gray] (3+0.5-0.25-0.125-0.0625,1+0.5+0.25) circle (0.5pt);
      \filldraw[gray] (3+0.5,1+0.5-0.25) circle (0.5pt);
      \filldraw[gray] (3+0.5+0.25,1+0.5-0.25) circle (0.5pt);
      \filldraw[gray] (3+0.5+0.25+0.125,1+0.5-0.25) circle (0.5pt);
      \filldraw[gray] (3+0.5+0.25+0.125+0.0625,1+0.5-0.25) circle (0.5pt);
      \filldraw[gray] (3+0.5+0.25+0.125-0.0625,1+0.5-0.25) circle (0.5pt);
      \filldraw[gray] (3+0.5+0.25-0.125,1+0.5-0.25) circle (0.5pt);
      \filldraw[gray] (3+0.5+0.25-0.125+0.0625,1+0.5-0.25) circle (0.5pt);
      \filldraw[gray] (3+0.5+0.25-0.125-0.0625,1+0.5-0.25) circle (0.5pt);
      \filldraw[gray] (3+0.5-0.25,1+0.5-0.25) circle (0.5pt);
      \filldraw[gray] (3+0.5-0.25+0.125,1+0.5-0.25) circle (0.5pt);
      \filldraw[gray] (3+0.5-0.25+0.125+0.0625,1+0.5-0.25) circle (0.5pt);
      \filldraw[gray] (3+0.5-0.25+0.125-0.0625,1+0.5-0.25) circle (0.5pt);
      \filldraw[gray] (3+0.5-0.25-0.125,1+0.5-0.25) circle (0.5pt);
      \filldraw[gray] (3+0.5-0.25-0.125+0.0625,1+0.5-0.25) circle (0.5pt);
      \filldraw[gray] (3+0.5-0.25-0.125-0.0625,1+0.5-0.25) circle (0.5pt);
      \filldraw[gray] (1+0.5,3+0.5) circle (0.5pt);
      \filldraw[gray] (1+0.5+0.25,3+0.5) circle (0.5pt);
      \filldraw[gray] (1+0.5-0.25,3+0.5) circle (0.5pt);
      \filldraw[gray] (1+0.5,3+0.5+0.25) circle (0.5pt);
      \filldraw[gray] (1+0.5+0.25,3+0.5+0.25) circle (0.5pt);
      \filldraw[gray] (1+0.5-0.25,3+0.5+0.25) circle (0.5pt);
      \filldraw[gray] (1+0.5,3+0.5+0.25+0.125) circle (0.5pt);
      \filldraw[gray] (1+0.5+0.25,3+0.5+0.25+0.125) circle (0.5pt);
      \filldraw[gray] (1+0.5-0.25,3+0.5+0.25+0.125) circle (0.5pt);
      \filldraw[gray] (1+0.5,3+0.5+0.25+0.125+0.0625) circle (0.5pt);
      \filldraw[gray] (1+0.5+0.25,3+0.5+0.25+0.125+0.0625) circle (0.5pt);
      \filldraw[gray] (1+0.5-0.25,3+0.5+0.25+0.125+0.0625) circle (0.5pt);
      \filldraw[gray] (1+0.5,3+0.5+0.25+0.125-0.0625) circle (0.5pt);
      \filldraw[gray] (1+0.5+0.25,3+0.5+0.25+0.125-0.0625) circle (0.5pt);
      \filldraw[gray] (1+0.5-0.25,3+0.5+0.25+0.125-0.0625) circle (0.5pt);
      \filldraw[gray] (1+0.5,3+0.5+0.25-0.125) circle (0.5pt);
      \filldraw[gray] (1+0.5+0.25,3+0.5+0.25-0.125) circle (0.5pt);
      \filldraw[gray] (1+0.5-0.25,3+0.5+0.25-0.125) circle (0.5pt);
      \filldraw[gray] (1+0.5,3+0.5+0.25-0.125+0.0625) circle (0.5pt);
      \filldraw[gray] (1+0.5+0.25,3+0.5+0.25-0.125+0.0625) circle (0.5pt);
      \filldraw[gray] (1+0.5-0.25,3+0.5+0.25-0.125+0.0625) circle (0.5pt);
      \filldraw[gray] (1+0.5,3+0.5+0.25-0.125-0.0625) circle (0.5pt);
      \filldraw[gray] (1+0.5+0.25,3+0.5+0.25-0.125-0.0625) circle (0.5pt);
      \filldraw[gray] (1+0.5-0.25,3+0.5+0.25-0.125-0.0625) circle (0.5pt);
      \filldraw[gray] (1+0.5,3+0.5-0.25) circle (0.5pt);
      \filldraw[gray] (1+0.5+0.25,3+0.5-0.25) circle (0.5pt);
      \filldraw[gray] (1+0.5-0.25,3+0.5-0.25) circle (0.5pt);
      \filldraw[gray] (1+0.5,3+0.5-0.25+0.125) circle (0.5pt);
      \filldraw[gray] (1+0.5+0.25,3+0.5-0.25+0.125) circle (0.5pt);
      \filldraw[gray] (1+0.5-0.25,3+0.5-0.25+0.125) circle (0.5pt);
      \filldraw[gray] (1+0.5,3+0.5-0.25+0.125+0.0625) circle (0.5pt);
      \filldraw[gray] (1+0.5+0.25,3+0.5-0.25+0.125+0.0625) circle (0.5pt);
      \filldraw[gray] (1+0.5-0.25,3+0.5-0.25+0.125+0.0625) circle (0.5pt);
      \filldraw[gray] (1+0.5,3+0.5-0.25+0.125-0.0625) circle (0.5pt);
      \filldraw[gray] (1+0.5+0.25,3+0.5-0.25+0.125-0.0625) circle (0.5pt);
      \filldraw[gray] (1+0.5-0.25,3+0.5-0.25+0.125-0.0625) circle (0.5pt);
      \filldraw[gray] (1+0.5,3+0.5-0.25-0.125) circle (0.5pt);
      \filldraw[gray] (1+0.5+0.25,3+0.5-0.25-0.125) circle (0.5pt);
      \filldraw[gray] (1+0.5-0.25,3+0.5-0.25-0.125) circle (0.5pt);
      \filldraw[gray] (1+0.5,3+0.5-0.25-0.125+0.0625) circle (0.5pt);
      \filldraw[gray] (1+0.5+0.25,3+0.5-0.25-0.125+0.0625) circle (0.5pt);
      \filldraw[gray] (1+0.5-0.25,3+0.5-0.25-0.125+0.0625) circle (0.5pt);
      \filldraw[gray] (1+0.5,3+0.5-0.25-0.125-0.0625) circle (0.5pt);
      \filldraw[gray] (1+0.5+0.25,3+0.5-0.25-0.125-0.0625) circle (0.5pt);
      \filldraw[gray] (1+0.5-0.25,3+0.5-0.25-0.125-0.0625) circle (0.5pt);

      \draw[very thick] (0,3) -- (1,3) -- (1,1) -- (4,1) -- (4,0);
  \end{tikzpicture}
  \subcaption[]{\centering Isotropic sparse grid (gray), \(L=4\), \(\mathbf{p}=(0,0)\).}
  \label{subfig: isotropic component diagram}
\end{subfigure}
\caption{Component diagrams for constructing two-dimensional sparse grids. A sparse grid is the union of all cartesian-product grids, \(\mathcal{X}_{l_1}^{p_1}\times\cdots\times\mathcal{X}_{l_d}^{p_d}\), each corresponding to a multi-index \(\boldsymbol{l}\in\mathcal{I}_L^{d}\). The highlighted components in (b) show the components of (a) mapped onto the isotropic component diagram.}
\label{fig: component diagrams}
\end{figure}
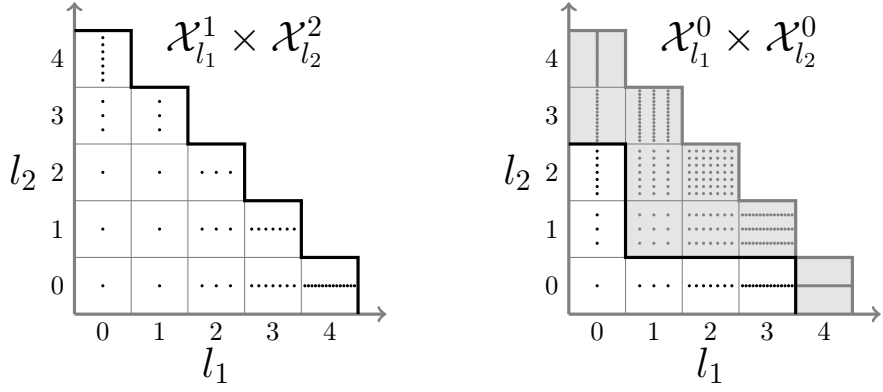

\begin{definition}\label{def: penalised sparse grid operator}
    Let \(L\in\mathbb{N}_0\) and
    \(\mathbf{p}\in\mathbb{N}_0^d\). We define the \emph{lengthscale-informed sparse grid interpolant operator}, \(P_{L,
    \boldsymbol{\nu},\mathbf{p}}:\mathcal{N}_{\Phi_{\boldsymbol{\nu},2^{\mathbf{p}}}}(\Gamma^d)\rightarrow\mathcal{N}_{\Phi_{\boldsymbol{\nu},2^{\mathbf{p}}}}(\Gamma^d)\), by 
\begin{align}
    P_{L,\boldsymbol{\nu}
    ,\mathbf{p}}\coloneqq \sum_{\boldsymbol{l}\in\mathcal{I}^d_L}\bigotimes_{j=1}^dR_{\Gamma}\circ\left(s_{\mathcal{X}_{l_j}^{p_j},\phi_{\nu_j,2^{p_j}}}-s_{\mathcal{X}_{l_j-1}^{p_j},\phi_{\nu_j,2^{p_j}}}\right).
\end{align}
\end{definition}
This operator maps a function \(f\in\mathcal{N}_{\Phi_{\boldsymbol{\nu},2^{\mathbf{p}}}}(\Gamma^d)\) to its \(\Phi_{\boldsymbol{\nu},2^{\mathbf{p}}}\)-kernel interpolant, with interpolation points arranged in a \textit{lengthscale-informed sparse grid}; for all \(\mathbf{x}\in\Gamma^d\),
\begin{align}
    P_{L,\boldsymbol{\nu},\mathbf{p}}(f)(\mathbf{x}) = s_{\mathcal{X}^{\otimes}_{L,\mathbf{p}},\Phi_{\nu,2^{\mathbf{p}}}}(f)(\mathbf{x}),\nonumber
\end{align}
where
\begin{align}
    \mathcal{X}^{\otimes}_{\mathbf{p},L}=\bigcup_{\boldsymbol{l}\in\mathcal{I}_{L}^d}\mathcal{X}^{p_1}_{l_1}\times\cdots\times\mathcal{X}^{p_d}_{l_d}.\nonumber
\end{align}
By Proposition \ref{prop: mean functions are kernel interpolants}, this also means that \(P_{L,\boldsymbol{\nu},\mathbf{p}}(f)\) is indeed the posterior mean function of a Gaussian process with zero prior mean, prior covariance kernel \(\Phi_{\boldsymbol{\nu},2^{\mathbf{p}}}\), and data in the form of evaluations of \(f\) at points in \(\mathcal{X}^{\otimes}_{\mathbf{p},L}\). In the component diagrams shown in Figure \ref{fig: component diagrams}, we see a two dimensional example of how the construction of lengthscale-informed sparse grids relate to their isotropic counterparts. Choosing a trivial penalty \(\mathbf{p}=\mathbf{0}\) recovers exactly the isotropic sparse grid construction in Definition \ref{def: iso sparse grid}. In this way, lengthscale-informed sparse grids can be viewed as a generalisation of sparse grids to settings with lengthscale anisotropy. The resultant isotropic and lengthscale-informed sparse grid designs are illustrated in Figure \(\ref{fig: stretched sparse grid}\). 

\begin{figure}[h]
    \centering
    \begin{tikzpicture}

    \draw[gray, thick] (-5,-2) -- (-5,2);
    \draw[gray, thick] (-1,-2) -- (-1,2);
    \draw[gray, thick] (-5,2) -- (-1,2);
    \draw[gray, thick] (-5,-2) -- (-1,-2);
    \filldraw[gray, fill opacity = 0.2] (-5,2) -- (-1,2) -- (-1,-2) -- (-5,-2) -- (-5,2);
    \filldraw[white] (-4,0.5) -- (-2,0.5) -- (-2,-0.5) -- (-4,-0.5) -- (-4,0.5);

    \draw[gray, thick] (5,-2) -- (5,2);
    \draw[gray, thick] (1,-2) -- (1,2);
    \draw[gray, thick] (5,2) -- (1,2);
    \draw[gray, thick] (5,-2) -- (1,-2);

    \draw[gray, thick] (-4,-0.5) -- (-4,0.5);
    \draw[gray, thick] (-2,-0.5) -- (-2,0.5);
    \draw[gray, thick] (-4,0.5) -- (-2,0.5);
    \draw[gray, thick] (-4,-0.5) -- (-2,-0.5);

    \draw[gray, thick] (-4,-0.5) -- (1,-2);
    \draw[gray, thick] (-2,-0.5) -- (1,-0.5-0.64285714285);
    \draw[gray, thick, dotted] (1,-0.5-0.64285714285) -- (5,-2);
    \draw[gray, thick] (-4,0.5) -- (1,2);
    \draw[gray, thick] (-2,0.5) -- (1,0.5+0.64285714285);
    \draw[gray, thick, dotted] (1,0.5+0.64285714285) -- (5,2);

    
    \filldraw[black] (-3,0) circle (1pt);
    
    \filldraw[gray] (-2,0) circle (1pt);
    \filldraw[gray] (-4,0) circle (1pt);
    \filldraw[black] (-2.5,0) circle (1pt);
    \filldraw[black] (-3.5,0) circle (1pt);
    \filldraw[gray] (-1.5,0) circle (1pt);
    \filldraw[gray] (-4.5,0) circle (1pt);
    \filldraw[black] (-2.25,0) circle (1pt);
    \filldraw[black] (-3.25,0) circle (1pt);
    \filldraw[gray] (-1.25,0) circle (1pt);
    \filldraw[gray] (-4.25,0) circle (1pt);
    \filldraw[black] (-2.75,0) circle (1pt);
    \filldraw[black] (-3.75,0) circle (1pt);
    \filldraw[gray] (-1.75,0) circle (1pt);
    \filldraw[gray] (-4.75,0) circle (1pt);
    \filldraw[black] (-2.25-0.125,0) circle (1pt);
    \filldraw[black] (-3.25-0.125,0) circle (1pt);
    \filldraw[gray] (-1.25-0.125,0) circle (1pt);
    \filldraw[gray] (-4.25-0.125,0) circle (1pt);
    \filldraw[black] (-2.75-0.125,0) circle (1pt);
    \filldraw[black] (-3.75-0.125,0) circle (1pt);
    \filldraw[gray] (-1.75-0.125,0) circle (1pt);
    \filldraw[gray] (-4.75-0.125,0) circle (1pt);
    
    \filldraw[black] (-2.25+0.125,0) circle (1pt);
    \filldraw[black] (-3.25+0.125,0) circle (1pt);
    \filldraw[gray] (-1.25+0.125,0) circle (1pt);
    \filldraw[gray] (-4.25+0.125,0) circle (1pt);
    \filldraw[black] (-2.75+0.125,0) circle (1pt);
    \filldraw[black] (-3.75+0.125,0) circle (1pt);
    \filldraw[gray] (-1.75+0.125,0) circle (1pt);
    \filldraw[gray] (-4.75+0.125,0) circle (1pt);

    \filldraw[gray] (-3,1) circle (1pt);
    \filldraw[gray] (-3,-1) circle (1pt);
    \filldraw[gray] (-3,-0.5) circle (1pt);
    \filldraw[gray] (-3,-1.5) circle (1pt);
    \filldraw[gray] (-3,0.5) circle (1pt);
    \filldraw[gray] (-3,1.5) circle (1pt);
    \filldraw[black] (-3,-0.25) circle (1pt);
    \filldraw[black] (-3,0.25) circle (1pt);
    \filldraw[gray] (-3,1.25) circle (1pt);
    \filldraw[gray] (-3,-1.25) circle (1pt);
    \filldraw[gray] (-3,0.75) circle (1pt);
    \filldraw[gray] (-3,-0.75) circle (1pt);
    \filldraw[gray] (-3,-1.75) circle (1pt);
    \filldraw[gray] (-3,1.75) circle (1pt);
    \filldraw[black] (-3,-0.25-0.125) circle (1pt);
    \filldraw[black] (-3,0.25-0.125) circle (1pt);
    \filldraw[gray] (-3,1.25-0.125) circle (1pt);
    \filldraw[gray] (-3,-1.25-0.125) circle (1pt);
    \filldraw[gray] (-3,0.75-0.125) circle (1pt);
    \filldraw[gray] (-3,-0.75-0.125) circle (1pt);
    \filldraw[gray] (-3,-1.75-0.125) circle (1pt);
    \filldraw[gray] (-3,1.75-0.125) circle (1pt);
    
    \filldraw[black] (-3,-0.25+0.125) circle (1pt);
    \filldraw[black] (-3,0.25+0.125) circle (1pt);
    \filldraw[gray] (-3,1.25+0.125) circle (1pt);
    \filldraw[gray] (-3,-1.25+0.125) circle (1pt);
    \filldraw[gray] (-3,0.75+0.125) circle (1pt);
    \filldraw[gray] (-3,-0.75+0.125) circle (1pt);
    \filldraw[gray] (-3,-1.75+0.125) circle (1pt);
    \filldraw[gray] (-3,1.75+0.125) circle (1pt);

    \filldraw[gray] (-4,1) circle (1pt);
    \filldraw[gray] (-2,1) circle (1pt);
    \filldraw[gray] (-2,-1) circle (1pt);
    \filldraw[gray] (-4,-1) circle (1pt);

    \filldraw[gray] (1/2 + -4,0 + 1) circle (1pt);
    \filldraw[gray] (-1/2 + -4,0 + 1) circle (1pt);
    \filldraw[gray] (-0.5/2 + -4,0 + 1) circle (1pt);
    \filldraw[gray] (-1.5/2 + -4,0 + 1) circle (1pt);
    \filldraw[gray] (1.5/2 + -4,0 + 1) circle (1pt);
    \filldraw[gray] (0.5/2 + -4,0 + 1) circle (1pt);
    \filldraw[gray] (0 + -4,1/2 + 1) circle (1pt);
    \filldraw[gray] (0 + -4,-1/2 + 1) circle (1pt);
    \filldraw[gray] (0 + -4,-0.5/2 + 1) circle (1pt);
    \filldraw[gray] (0 + -4,-1.5/2 + 1) circle (1pt);
    \filldraw[gray] (0 + -4,1.5/2 + 1) circle (1pt);
    \filldraw[gray] (0 + -4,0.5/2 + 1) circle (1pt);

    \filldraw[gray] (1/2 + -2,0 + 1) circle (1pt);
    \filldraw[gray] (-1/2 + -2,0 + 1) circle (1pt);
    \filldraw[gray] (-0.5/2 + -2,0 + 1) circle (1pt);
    \filldraw[gray] (-1.5/2 + -2,0 + 1) circle (1pt);
    \filldraw[gray] (1.5/2 + -2,0 + 1) circle (1pt);
    \filldraw[gray] (0.5/2 + -2,0 + 1) circle (1pt);
    \filldraw[gray] (0 + -2,1/2 + 1) circle (1pt);
    \filldraw[gray] (0 + -2,-1/2 + 1) circle (1pt);
    \filldraw[gray] (0 + -2,-0.5/2 + 1) circle (1pt);
    \filldraw[gray] (0 + -2,-1.5/2 + 1) circle (1pt);
    \filldraw[gray] (0 + -2,1.5/2 + 1) circle (1pt);
    \filldraw[gray] (0 + -2,0.5/2 + 1) circle (1pt);

    \filldraw[gray] (1/2 + -2,0 - 1) circle (1pt);
    \filldraw[gray] (-1/2 + -2,0 - 1) circle (1pt);
    \filldraw[gray] (-0.5/2 + -2,0 - 1) circle (1pt);
    \filldraw[gray] (-1.5/2 + -2,0 - 1) circle (1pt);
    \filldraw[gray] (1.5/2 + -2,0 - 1) circle (1pt);
    \filldraw[gray] (0.5/2 + -2,0 - 1) circle (1pt);
    \filldraw[gray] (0 + -2,1/2 - 1) circle (1pt);
    \filldraw[gray] (0 + -2,-1/2 - 1) circle (1pt);
    \filldraw[gray] (0 + -2,-0.5/2 - 1) circle (1pt);
    \filldraw[gray] (0 + -2,-1.5/2 - 1) circle (1pt);
    \filldraw[gray] (0 + -2,1.5/2 - 1) circle (1pt);
    \filldraw[gray] (0 + -2,0.5/2 - 1) circle (1pt);

    \filldraw[gray] (1/2 + -4,0 - 1) circle (1pt);
    \filldraw[gray] (-1/2 + -4,0 - 1) circle (1pt);
    \filldraw[gray] (-0.5/2 + -4,0 - 1) circle (1pt);
    \filldraw[gray] (-1.5/2 + -4,0 - 1) circle (1pt);
    \filldraw[gray] (1.5/2 + -4,0 - 1) circle (1pt);
    \filldraw[gray] (0.5/2 + -4,0 - 1) circle (1pt);
    \filldraw[gray] (0 + -4,1/2 - 1) circle (1pt);
    \filldraw[gray] (0 + -4,-1/2 - 1) circle (1pt);
    \filldraw[gray] (0 + -4,-0.5/2 - 1) circle (1pt);
    \filldraw[gray] (0 + -4,-1.5/2 - 1) circle (1pt);
    \filldraw[gray] (0 + -4,1.5/2 - 1) circle (1pt);
    \filldraw[gray] (0 + -4,0.5/2 - 1) circle (1pt);

    \filldraw[gray] (-4 + 0.5,1 + 0.5) circle (1pt);
    \filldraw[gray] (-2 + 0.5,1 + 0.5) circle (1pt);
    \filldraw[gray] (-2 + 0.5,-1 + 0.5) circle (1pt);
    \filldraw[gray] (-4 + 0.5,-1 + 0.5) circle (1pt);

    \filldraw[gray] (-4 + 0.5,1 - 0.5) circle (1pt);
    \filldraw[gray] (-2 + 0.5,1 - 0.5) circle (1pt);
    \filldraw[gray] (-2 + 0.5,-1 - 0.5) circle (1pt);
    \filldraw[gray] (-4 + 0.5,-1 - 0.5) circle (1pt);

    \filldraw[gray] (-4 - 0.5,1 + 0.5) circle (1pt);
    \filldraw[gray] (-2 - 0.5,1 + 0.5) circle (1pt);
    \filldraw[gray] (-2 - 0.5,-1 + 0.5) circle (1pt);
    \filldraw[gray] (-4 - 0.5,-1 + 0.5) circle (1pt);
    
    \filldraw[gray] (-4 - 0.5,1 - 0.5) circle (1pt);
    \filldraw[gray] (-2 - 0.5,1 - 0.5) circle (1pt);
    \filldraw[gray] (-2 - 0.5,-1 - 0.5) circle (1pt);
    \filldraw[gray] (-4 - 0.5,-1 - 0.5) circle (1pt);

    \filldraw[black] (3,0) circle (1pt);
    
    \filldraw[black] (2,0) circle (1pt);
    \filldraw[black] (4,0) circle (1pt);
    \filldraw[black] (2.5,0) circle (1pt);
    \filldraw[black] (3.5,0) circle (1pt);
    \filldraw[black] (1.5,0) circle (1pt);
    \filldraw[black] (4.5,0) circle (1pt);
    \filldraw[black] (2.25,0) circle (1pt);
    \filldraw[black] (3.25,0) circle (1pt);
    \filldraw[black] (1.25,0) circle (1pt);
    \filldraw[black] (4.25,0) circle (1pt);
    \filldraw[black] (2.75,0) circle (1pt);
    \filldraw[black] (3.75,0) circle (1pt);
    \filldraw[black] (1.75,0) circle (1pt);
    \filldraw[black] (4.75,0) circle (1pt);

    \filldraw[black] (3,1) circle (1pt);
    \filldraw[black] (3,-1) circle (1pt);
    \filldraw[black] (3,-0.5) circle (1pt);
    \filldraw[black] (3,-1.5) circle (1pt);
    \filldraw[black] (3,0.5) circle (1pt);
    \filldraw[black] (3,1.5) circle (1pt);

    \node at (-3,-3.1) {\small(a) Isotropic sparse grid, \(L=4\).};
    \node at (3,-3.1) {\small(b) Lengthscale-informed sparse grid,};
    \node at (3.2,-3.5) {\small\(L=4\), with penalty \(\mathbf{p}=(1,2)\).};

    \node[black] at (-5,-2.3) {\small-0.5};
    \node[black] at (-1,-2.3) {\small0.5};
    \node[black] at (5,-2.3) {\small0.5};
    \node[black] at (1,-2.3) {\small-0.5};

    \node[black] at (-5.5,-2) {\small-0.5};
    \node[black] at (-5.5,2) {\small0.5};
    \node[black] at (5.5,-2) {\small-0.5};
    \node[black] at (5.5,2) {\small0.5};

    \node[black] at (-3,-2.5) {\Large\(x_1\)};
    \node[black] at (-5.5,0) {\Large\(x_2\)};
    \node[black] at (3,-2.5) {\Large\(x_1\)};
    \node[black] at (5.5,0) {\Large\(x_2\)};

\end{tikzpicture}

    \caption{Here we can see how a lengthscale-informed sparse grid, (b), is simply an isotropic sparse grid, (a), of the same level, \(L\), stretched by the penalty vector, \(\mathbf{p}\), where points outside the domain are then excluded. Since \(p_2>p_1\), \(f\) is assumed to be more sensitive in \(x_1\), and so more points are placed in the horizontal direction.}
    \label{fig: stretched sparse grid}
\end{figure}
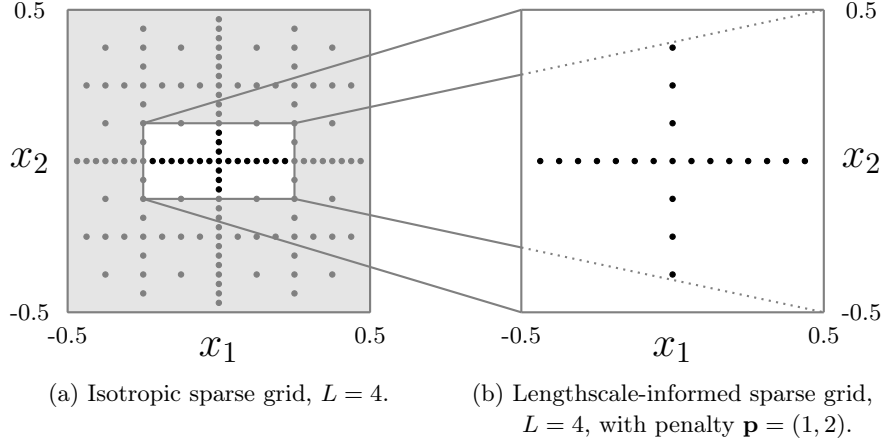
\begin{remark}
    The theoretical results presented in Section \ref{subsec: results} rely upon recursive properties of uniformly spaced points. Extending the analysis to other point sets, for example Clenshaw-Curtis abscissae, would require significant alterations to the theory. The numerical experiments presented in Section \ref{sec: numerics} nevertheless suggest that similar convergence behaviour may hold in practice.
\end{remark}
\subsection{Bounding the error}\label{subsec: results}

We now state our main theoretical results. Theorem \ref{thm: error in L} establishes an upper bound for the approximation error for Mat\'ern kernel interpolation on lengthscale-informed sparse grids in terms of \(L\), the level of construction of the sparse grid. We measure the error in the corresponding native space norm of the separable Mat\'ern kernel, wherein any anisotropy in the lengthscale is encoded via the penalty vector \(\mathbf{p}\in\mathbb{N}_0^d\). By Proposition \ref{prop: sup norm}, this also gives us an upper bound in \(L^\infty\). Theorem \ref{thm: counting abscissae} then links the number of abscissae in a lengthscale-informed sparse grid to \(L\) via an explicit expression. We compare both results directly with their isotropic counterparts and consequently exhibit how, for a fixed \(L\), increased penalties correspond, in a sense, to lower dimensional problems, with respect to their computational complexity. Derivations of the two theorems are presented in Section \ref{sec: error analysis}.
\begin{definition} We define \(\mathcal{P}_k^d\) to be the set of all \(k\)-length subsets of \(\{1,\dots,d\}\), ordered such that, for a given \(\mathfrak{u}\in\mathcal{P}^d_k\), we have \(\mathfrak{u}_i<\mathfrak{u}_{i+1}\), for all \(1\leq i\leq k-1\). For a given \(\mathfrak{u}\in\mathcal{P}^d_k\), and a given vector \(\mathbf{a}=[a_1,\dots,a_d]\), we define the subvector \(\mathbf{a}_{\mathfrak{u}}\coloneqq[a_{\mathfrak{u}_1},\dots,a_{\mathfrak{u}_k}]\).
    
\end{definition}
\begin{theorem}[Native space approximation error]\label{thm: error in L}
    Let \(\mathbf{p}\in\mathbb{N}_0^d\), \(L\in\mathbb{N}_0\), be given, and let \(\boldsymbol{\nu},\boldsymbol{\alpha}\in\mathbb{R}^d_{\geq1/2}\) be such that \(\nu_j-\alpha_j=c\in\mathbb{R}_{\geq0}\) for all \(1\leq j \leq d\). For constants \(C^{(\boldsymbol{\nu},\boldsymbol{\alpha})}_{\ref{lem: tensor hilbert}}\) and \(C_{\ref{lem: supersets}}^{(\boldsymbol{\nu}_{\mathfrak{u}},\boldsymbol{\alpha}_{\mathfrak{u}})}\) independent of \(\mathbf{p}\), the approximation error of the lengthscale-informed sparse grid interpolation operator is bounded by,
    \begin{align}
    \begin{split}
    &\qquad\quad\|I-P_{L,\boldsymbol{\nu},\mathbf{p}}\|_{\mathcal{N}_{\Phi_{\boldsymbol{\nu},2^{\mathbf{p}}}}(\Gamma^d)\rightarrow\mathcal{N}_{\Phi_{\boldsymbol{\alpha},2^{\mathbf{p}}}}(\Gamma^d)}\\
    &\leq C^{(\boldsymbol{\nu},\boldsymbol{\alpha})}_{\ref{lem: tensor hilbert}}\sum_{k=1}^d2^{-ck}\sum_{\mathfrak{u}\in\mathcal{P}^d_k}2^{-c|\mathbf{p}_{\mathfrak{u}}|_1}\epsilon^{(k)}_{\boldsymbol{\nu}_{\mathfrak{u}},\boldsymbol{\alpha}_{\mathfrak
    {u}}}(L-|\mathbf{p}_{\mathfrak{u}}|_1-k),
    \end{split}\nonumber
    \end{align}
    where \(\epsilon^{(k)}_{\boldsymbol{\nu}_{\mathfrak{u}},\boldsymbol{\alpha}_{\mathfrak{u}}}(L)\) is the corresponding approximation error bound for a \(k\)-dimensional isotropic sparse grid of level \(L\), in the mixed Sobolev norm. That is,
    \begin{align}
    \|I-S_{L,\boldsymbol{\nu}_{\mathfrak{u}}}\|_{H_{\textrm{mix}}^{\boldsymbol{\nu}_{\mathfrak{u}}+1/2}(\Gamma^k)\rightarrow H_{\textrm{mix}}^{\boldsymbol{\alpha}_{\mathfrak{u}}+1/2}(\Gamma^k)}&\leq C_{\ref{lem: supersets}}^{(\boldsymbol{\nu}_{\mathfrak{u}},\boldsymbol{\alpha}_{\mathfrak{u}})}\sum_{\boldsymbol{l}\in\mathbb{N}_0^k\setminus\mathcal{I}_{L}^k}2^{-c|\boldsymbol{l}|_1}\nonumber\\
    &\eqqcolon\epsilon^{(k)}_{\boldsymbol{\nu}_{\mathfrak{u}},\boldsymbol{\alpha}_{\mathfrak{u}}}(L)\nonumber.
    \end{align}
    For \(L<0\), we define \(\epsilon^{(k)}_{\boldsymbol{\nu}_{\mathfrak{u}},\boldsymbol{\alpha}_{\mathfrak{u}}}(L)\coloneqq\epsilon^{(k)}_{\boldsymbol{\nu}_{\mathfrak{u}},\boldsymbol{\alpha}_{\mathfrak{u}}}(0)\).
\end{theorem}
\begin{proof}
    See Section \ref{subsec: error in L}.
\end{proof}

\begin{remark}\label{rem: embeddings}
    This result contrasts with those of Nobile, Tempone \& Wolfers (2018)~\cite{Nobile2018}, where error bounds are established in more standard \(H_{\textrm{mix}}^{\otimes}\) norms. While the equivalence constants between these and the corresponding native space norms generally grow with the kernel lengthscales (see Proposition \ref{prop: equivalence}), the following results show that the \(L^\infty\)-error is nonetheless controlled by the native space norm, independently of the lengthscale anisotropy.
\end{remark}
\begin{remark}
    The constants \(C^{(\boldsymbol{\nu},\boldsymbol{\alpha})}_{\ref{lem: tensor hilbert}}\) and \(C_{\ref{lem: supersets}}^{(\boldsymbol{\nu}_{\mathfrak{u}},\boldsymbol{\alpha}_{\mathfrak{u}})}\) are given explicitly in Lemmas~\ref{lem: tensor hilbert} and \ref{lem: supersets}, respectively. While \(C^{(\boldsymbol{\nu},\boldsymbol{\alpha})}_{\ref{lem: tensor hilbert}}\) depends on \(d\), we show in Proposition~\ref{prop: bounded constant} that it is bounded from above by \(1\). In contrast, \(C_{\ref{lem: supersets}}^{(\boldsymbol{\nu}_{\mathfrak{u}},\boldsymbol{\alpha}_{\mathfrak{u}})}\) is the product of the \(k\)-many constants appearing in Proposition~\ref{prop: initial wendland} and Lemma~\ref{thm: Wendland norm}, each of which is independent of both \(k\) and \(L\). To the authors' knowledge, an explicit analysis of these constants has not been carried out. This same product appears in the asymptotic formulation in \cite{Nobile2018} based on classical Sobolev norms, where vertical shifts were observed numerically in error convergence rates for increasing dimensions. Comparable small vertical shifts can also be seen for isotropic experiments (\(\mathbf{p}=\mathbf{0}\)) shown in Figures~\ref{fig: iso vs LISG} and \ref{fig: LISG time}, however appear to be absent in the setting of growing lengthscales, seen most clearly in Figure~\ref{fig: LISG high dimensions}. These observations suggests that any growth of these constants with the dimension \(k\) is quickly dominated by the exponential decay of the factors \(2^{-c|\mathbf{p}_{\mathfrak{u}}|_1}\) when sufficient anisotropy is present and properly accounted for.
\end{remark}
\begin{proposition}\label{prop: sup norm}
    For the native space \(\mathcal{N}_{\varphi}(\Omega)\),  with reproducing kernel \(\varphi\), we have, for any \(g\in\mathcal{N}_{\varphi}(\Omega)\)
    \begin{align}
       \|g\|_{L^{\infty}(\Omega)}\leq\|\sqrt{\varphi(\cdot,\cdot)}\|_{L^{\infty}(\Omega)}\|g\|_{\mathcal{N}_{\varphi}(\Omega)}.\nonumber
    \end{align}
\end{proposition}
\begin{proof}
    Using the reproducing property of the kernel \(\varphi\) and the Cauchy-Schwarz inequality, we have, for any \(\mathbf{x}\in\Omega\),
    \begin{align}
        g(\mathbf{x})=\langle g(\cdot),\varphi(\cdot,\mathbf{x})\rangle_{\mathcal{N}_{\varphi}(\Omega)}\leq\|g\|_{\mathcal{N}_{\varphi}(\Omega)}\|\varphi(\cdot,\mathbf{x})\|_{\mathcal{N}_\varphi(\Omega)}=\|g\|_{\mathcal{N}_{\varphi}(\Omega)}\sqrt{\varphi(\mathbf{x},\mathbf{x})}.\nonumber
    \end{align}
    Taking the supremum of both sides completes the proof.
\end{proof}

\begin{corollary}\label{corr: sup norm} Let \(\boldsymbol{\nu},\boldsymbol{\alpha}\in\mathbb{R}^d_{\geq1/2}\) such that \(\alpha_j\leq\nu_j\) for all \(1\leq j \leq d\). For \(f\in\mathcal{N}_{\Phi_{\boldsymbol{\nu},2^{\mathbf{p}}}}(\Gamma^d)\), we have 
    \begin{align}
        \|f-P_{L,\boldsymbol{\nu},\mathbf{p}}(f)\|_{L^\infty(\Gamma^d)}\leq\sigma\|f-P_{L,\boldsymbol{\nu},\mathbf{p}}(f)\|_{\mathcal{N}_{\Phi_{\boldsymbol{\alpha},2^{\mathbf{p}}}}(\Gamma^d)},\label{eq: var prop 2}
    \end{align}
    where \(\sigma=\sigma_1\sigma_2\cdots\sigma_d\), with \(\sigma_j\) the standard deviation of the one-dimensional Mat\'ern kernel \(\phi_{\alpha_j,2^{p_j}}\).
\end{corollary}
\begin{proof}
    Directly follows from Proposition \ref{prop: sup norm} by noticing that \(\|\Phi_{\boldsymbol{\alpha},2^{\mathbf{p}}}(\cdot,\cdot)^{1/2}\|_{L^\infty(\Gamma^d)}=\sigma\).
\end{proof}

The bound in Theorem \ref{thm: error in L} is written as the sum of contributions from each \(k\)-dimensional subspace \(\mathfrak{u}\in\mathcal{P}^d_k\), for \(1\leq k\leq d\). Each error contribution is scaled by the compounded penalty, \(2^{-c|\mathbf{p}_{\mathfrak{u}}|}\), and hence, in general, higher dimensional subspaces are penalised more. For problems exhibiting anisotropy such that the lengthscale grows with dimension, these constants can quickly become extremely small.
\begin{example}
    Consider an arbitrarily high dimensional target function, \(d\geq11\), with a lengthscale-anisotropy such that the penalty grows linearly with dimension, \(p_j=j-1\). For \(c=1\), the \textit{largest} constant corresponding to an \(11\)-dimensional subspace is given by \(2^{-(1+2+\cdots+10)}=2^{-55}\approx 3\times10^{-17}\). As such, the relative error contribution from any subspace greater than \(k=10\) is less than machine precision, in a sense `capping' the dimensionality of the problem at \(d=10\). For a slower growth penalty, \(p_j = \lceil\log j\rceil\), the equivalent effective dimensionality with respect to machine precision is \(d=20\).
\end{example}
The construction of lengthscale-informed sparse grids exploits this imbalance in error contributions by delaying the onset of additional points in the latter subspaces, instead favouring those with larger compounded penalties and, hence, with more associated error. The result is an approximation scheme with, for a given level \(L\in\mathbb{N}_0\), improved error over the isotropic case, that also employs many fewer function evaluations. We can explicitly count the number of function evaluations required with the following result.
\begin{theorem}[Counting abscissae in lengthscale-informed sparse grids]\label{thm: counting abscissae}
For a given function \(f:\Gamma^d\rightarrow\mathbb{R}\), evaluating \(P_{L,\boldsymbol{\nu},\mathbf{p}}(f)\) employs exactly
\begin{align}
    N_{d,\mathbf{p}}(L)=1+\sum_{k=1}^d2^k\sum_{\mathfrak{u}\in\mathcal{P}^d_k}N_{k,\mathbf{0}}(L-|\mathbf{p}_{\mathfrak{u}}|-k)\nonumber
\end{align}
many point evaluations of \(f\), where \(N_{k,\mathbf{0}}(L)\) is the number of points in an isotropic sparse grid of dimension \(k\) and level \(L\), given by
\begin{align}
    N_{k,\mathbf{0}}(L)=\begin{cases}
        \sum_{l=0}^{L}\binom{l+k-1}{k-1}2^l&\textrm{if }L\geq0, \textrm{and,}\\
        0&\textrm{otherwise.}
    \end{cases}\nonumber
\end{align}
\end{theorem}
\begin{proof}
    See Section \ref{subsec: counting}.
\end{proof}
Theorems \ref{thm: error in L} and \ref{thm: counting abscissae} imply the existence of two distinct regimes. Firstly, the pre-asymptotic behaviour of the error driven by the lengthscale anisotropy, and then, once \(L\) is large enough such that most \(k\)-dimensional subspaces, \(1\leq k\leq d\), contain unique points, the asymptotic regime begins. Here the rate obeys the same asymptotic rates as in Proposition \ref{prop: nobile result}, governed by the smoothness and the dimensionality. Consequently, we view the two main benefits of employing lengthscale-informed sparse grids in kernel interpolation to be:
\begin{enumerate}
    \item For very high dimensional functions, \(d\gg10\), where it is practically infeasible to reach \(L\) large enough such that the asymptotic rate begins to dominate due to the size of grid required, lengthscale-informed sparse grids offer a tractable approximation scheme for functions exhibiting sufficient lengthscale anisotropy. Furthermore, these benefits are not subject to additional regularity assumptions on \(f\), and as such are applicable to functions with limited or isotropic smoothness.
    \item Even in lower dimensional settings, \(2\leq d\leq10\), for any problem exhibiting anisotropy in the lengthscale, lengthscale-informed sparse grids outperform their isotropic counterpart in terms of the required number of evaluations to reach a given error threshold.
\end{enumerate}

We have chosen to present Theorems~\ref{thm: error in L} and \ref{thm: counting abscissae} for the case where \(\boldsymbol{\nu}-\boldsymbol{{\alpha}}=\mathbf{1}c\), for some \(c\in\mathbb{R}_{\geq0}\). In practice, we are typically only interested in isotropic \(\boldsymbol{\alpha}\), and hence this case most likely coincides with the assumption of isotropic regularity, i.e. \(\nu_j=\nu\) for all \(1\leq j\leq d\). This formulation therefore highlights that the advantages of the method are achieved without additional regularity assumptions; however, we also note that most of the error analysis given in Section~\ref{sec: error analysis} can be extended more generally to any \(\boldsymbol{\nu}\in\mathbb{R}_{\geq1/2}^d\). These results are substantially different from those of established anisotropic sparse grid constructions targeting anisotropy in the regularity, in which the rate of growth of points in each dimension \(j\) is weighted by \(\nu_j\) with the goal of improving asymptotic convergence rates \cite{Nobile2008, Rieger2019}. These methods require faster convergence along some axial directions than others, an assumption not made in this work. Indeed, the differing approaches suggest that complementary constructions simultaneously treating both forms of anisotropy are possible, see \cite{addy2026thesis} for more details.

\subsection{Predictive variance in Gaussian process emulation}\label{subsec: predictive variance}
In the context of Gaussian process emulation, we can adapt the proof strategy from Section 3,~\cite{Teckentrup2020}, to bound the convergence of the predictive marginal variance by the approximation error bounded in Theorem \ref{thm: error in L}.

\begin{proposition}\label{prop: vairance convergence}
Let \(\tilde\Phi_{L,\boldsymbol{\nu},\mathbf{p}}:\Gamma^d\times\Gamma^d\rightarrow\mathbb{R}_{>0}\) denote the posterior covariance associated with the lengthscale-informed sparse grid design \(\mathcal{X}^{\otimes}_{L,\mathbf{p}}\) and prior covariance \(\Phi_{\boldsymbol{\nu},2^{\mathbf{p}}}\). The posterior marginal variance is bounded by
\begin{align}
    \|\tilde\Phi_{L,\boldsymbol{\nu},\mathbf{p}}(\cdot,\cdot)^{1/2}\|_{L^{\infty}(\Gamma^d)} \leq \sigma\|I-P_{L,\boldsymbol{\nu},\mathbf{p}}\|_{\mathcal{N}_{\Phi_{\boldsymbol{\nu},2^{\mathbf{p}}}}(\Gamma^d)\rightarrow\mathcal{N}_{\Phi_{\boldsymbol{\alpha},2^{\mathbf{p}}}}(\Gamma^d)},\nonumber
\end{align}
for all \(\boldsymbol{\alpha}\in\mathbb{R}^d\) such that \(1/2\leq\alpha_j\leq\nu_j\), \(1\leq j\leq d\). Here, \(\sigma\coloneqq\sigma_1\sigma_2\cdots\sigma_d\), where \(\sigma_j\) is the standard deviation of the one-dimensional Mat\'ern kernel \(\phi_{\alpha_j,2^{p_j}}\).
\end{proposition}

\begin{proof}
    Using Proposition 3.5,~\cite{stuart_teckentrup_2018}, we have
    \begin{align}
        \tilde\Phi_{L,\boldsymbol{\nu},\mathbf{p}}(\mathbf{x},\mathbf{x})^{1/2} = \sup_{\|g\|_{\mathcal{N}_{\Phi_{\boldsymbol{\nu},2^\mathbf{p}}}}=1}|g(x)-P_{L,\boldsymbol{\nu},\mathbf{p}}(g)(x)|,\nonumber
    \end{align}
    and, as shown in the proof of Theorem 3.8,~\cite{Teckentrup2020}, it follows that
    \begin{align}
        \|\tilde\Phi_{L,\boldsymbol{\nu},\mathbf{p}}(\cdot,\cdot)^{1/2}\|_{L^\infty(\Gamma^d)}&=\sup_{\mathbf{x}\in\Gamma^d}\sup_{\|g\|_{\mathcal{N}_{\Phi_{\boldsymbol{\nu},2^\mathbf{p}}}}=1}|g(\mathbf{x})-P_{L,\boldsymbol{\nu},\mathbf{p}}(g)(\mathbf{x})|,\nonumber\\
        &=\sup_{\|g\|_{\mathcal{N}_{\Phi_{\boldsymbol{\nu},2^\mathbf{p}}}}=1}\sup_{\mathbf{x}\in\Gamma^d}|g(\mathbf{x})-P_{L,\boldsymbol{\nu},\mathbf{p}}(g)(\mathbf{x})|,\nonumber\\
        &=\sup_{\|g\|_{\mathcal{N}_{\Phi_{\boldsymbol{\nu},2^\mathbf{p}}}}=1}\|g-P_{L,\boldsymbol{\nu},\mathbf{p}}(g)\|_{L^\infty(\Gamma^d)}.\label{eq: var prop 1}
    \end{align}
    Applying Corollary \ref{corr: sup norm} to \eqref{eq: var prop 1} completes the proof.
\end{proof}

\section{Error analysis}\label{sec: error analysis}
In this section we present the proofs and necessary results for Theorems \ref{thm: error in L} and \ref{thm: counting abscissae} in Sections \ref{subsec: error in L} and \ref{subsec: counting}, respectively.

\subsection{Error in \(L\)}\label{subsec: error in L}
To derive the error of the lengthscale-informed sparse grid operator in terms of \(L\in\mathbb{N}_0\), the construction level of the sparse grid, in the native space norm of the separable Mat\'ern covariance kernel, we first introduce a key property of tensor product operators in tensor product spaces.

\begin{proposition}[Proposition 4.127, \cite{Hackbush2012}]\label{prop: tensor product norm}
    Let \(\boldsymbol{\alpha}\in\mathbb{R}^d_{\geq0}\) be such that \(\nu_j\geq\alpha_j\) for all \(1\leq j\leq d\), and let \(S:\mathcal{N}_{\Phi_{\boldsymbol{\nu},\boldsymbol{\lambda}}}(\Gamma^d)\rightarrow\mathcal{N}_{\Phi_{\boldsymbol{\alpha},\boldsymbol{\lambda}}}(\Gamma^d)\) be an operator of the form \(S=s_1\otimes\cdots\otimes s_d\), where each \(s_j:\mathcal{N}_{\phi_{\nu_j,\lambda_j}}\left(\Gamma\right)\rightarrow\mathcal{N}_{\phi_{\alpha_j,\lambda_j}}\left(\Gamma\right)\) is a one-dimensional linear operator. Then,
    \begin{align}
        \left\|S\right\|_{\mathcal{N}_{\Phi_{\boldsymbol{\nu},\boldsymbol{\lambda}}}(\Gamma^d)\rightarrow\mathcal{N}_{\Phi_{\boldsymbol{\alpha},\boldsymbol{\lambda}}}(\Gamma^d)}=\prod_{j=1}^d\left\|s_j\right\|_{\mathcal{N}_{\phi_{\nu_j,\lambda_j}}\left(\Gamma\right)\rightarrow\mathcal{N}_{\phi_{\alpha_j,\lambda_j}}\left(\Gamma\right)}.\label{eq: tensor product structure}
    \end{align}
\end{proposition}
This property allows us to rewrite the error of our sparse grid interpolation operator, Definition \ref{def: penalised sparse grid operator}, entirely in terms of one dimensional errors. As such, we need only establish a connection between the interpolation error and the lengthscale parameter in the case of \(d=1\). We start by establishing a key result for kernel interpolant extensions in native spaces over general intervals \(\tilde{\Gamma}\subseteq \mathbb{R}\).

\begin{proposition}\label{prop: matern extensions}
    Let \(\tilde{\Gamma}\subset\mathbb{R}\) be an interval. For \(f\in\mathcal{N}_{\phi_{\nu,\lambda}}(\tilde{\Gamma})\), we have
    \begin{align}
        \begin{split}
           \|R_{\tilde{\Gamma}}( s_{\mathcal{X},\phi_{\nu,\lambda}}(f))\|_{\mathcal{N}_{\phi_{\nu,\lambda}}(\tilde{\Gamma})}=\|s_{\mathcal{X},\phi_{\nu,\lambda}}(f)\|_{\mathcal{N}_{\phi_{\nu,\lambda}}(\mathbb{R})}.
        \end{split}\nonumber
    \end{align}
\end{proposition}
\begin{proof}
    Define \(\mathcal{E}_f\) as the set of all possible extensions
    \begin{align}
        \mathcal{E}_f\coloneqq\left\{u\in\mathcal{N}_{\phi_{\nu,\lambda}}(\mathbb{R}):u|_{\tilde{\Gamma}}=R_{\tilde{\Gamma}}(s_{\mathcal{X},\phi_{\nu,\lambda}}(f))\right\},\nonumber
    \end{align}
    then, by definition of extensions in Hilbert spaces, we have
    \begin{align}
        \|R_{\tilde{\Gamma}}(s_{\mathcal{X},\phi_{\nu,\lambda}}(f))\|_{\mathcal{N}_{\phi_{\nu,\lambda}}(\tilde{\Gamma})}\coloneqq\inf_{u\in \mathcal{E}_f} \|u\|_{\mathcal{N}_{\phi_{\nu,\lambda}}(\mathbb{R})}.\nonumber
    \end{align}
    In particular, we note that the unrestricted interpolant, \(s_{\mathcal{X},\phi_{\nu,\lambda}}(f)\), is in \(\mathcal{E}_f\), and hence
    \begin{align}\label{eq: restrictions geq}
        \|s_{\mathcal{X},\phi_{\nu,\lambda}}\|_{\mathcal{N}_\phi(\mathbb{R})}\geq\inf_{u\in \mathcal{E}_f} \|u\|_{\mathcal{N}_{\phi_{\nu,\lambda}}(\mathbb{R})}=\|R_{\tilde{\Gamma}}(s_{\mathcal{X},\phi_{\nu,\lambda}}(f))\|_{\mathcal{N}_{\phi_{\nu,\lambda}}(\tilde{\Gamma})}.
    \end{align}
    Reintroducing the sampling operators \(\mathcal{T}_{\mathcal{X}}\) used in the definition of \(s_{\mathcal{X},\phi}(f)\) in Definition \ref{def:restricted kernel interpolant}, clearly we have, for all \(u\in \mathcal{E}_f\), \(\mathcal{T}_{\mathcal{X}}(u)=\mathcal{T}_{\mathcal{X}}(f)\). In other words, each \(u\in \mathcal{E}_f\) interpolates \(f\) at the points in \(\mathcal{X}\). Hence,
    \begin{align}
        \|R_{\tilde{\Gamma}}(s_{\mathcal{X},\phi_{\nu,\lambda}}(f))\|_{\mathcal{N}_{\phi_{\nu,\lambda}}(\tilde{\Gamma})} &=\inf_{u\in \mathcal{E}_f} \|u\|_{\mathcal{N}_{\phi_{\nu,\lambda}}(\mathbb{R})},\nonumber\\
        &\geq\inf_{\substack{u\in\mathcal{N}_{\phi_{\nu,\lambda}}(\mathbb{R})\nonumber\\
        \mathcal{T}_{\mathcal{X}}(u)=\mathcal{T}_{\mathcal{X}}(f)}} \|u\|_{\mathcal{N}_{\phi_{\nu,\lambda}}(\mathbb{R})},\nonumber\\
        &\eqqcolon\|s_{\mathcal{X},\phi_{\nu,\lambda}}(f)\|_{\mathcal{N}_{\phi_{\nu,\lambda}}(\mathbb{R})},\label{eq: restrictions leq}
    \end{align}
    simply by inclusion and then by Definition \ref{def:restricted kernel interpolant}. Combining \eqref{eq: restrictions geq} and \eqref{eq: restrictions leq} gives us our result.
\end{proof}

We wish to extend these results to restrictions of the difference between successive interpolants, as employed in the construction of sparse grid operators. 
This is possible since we are able rewrite the differences themselves as kernel interpolants.

\begin{proposition}\label{prop: rewrite differences}Let \(\tilde{\mathcal{X}}\subset\mathcal{X}\) and define \(\textup{id}:\mathcal{N}_{\phi_{\nu,\lambda}}(\mathbb{R})\rightarrow\mathcal{N}_{\phi_{\nu,\lambda}}(\mathbb{R})\) to be the one-dimensional identity operator. Let \(f\in\mathcal{N}_{\phi_{\nu,\lambda}}(\Gamma)\), then
\begin{enumerate}
    \item \((s_{\mathcal{X},\phi_{\nu,\lambda}}-s_{\tilde{\mathcal{X}},\phi_{\nu,\lambda}})(f)=(\textup{id}-s_{\tilde{\mathcal{X}},\phi_{\nu,\lambda}})(s_{\mathcal{X},\phi_{\nu,\lambda}}(f))\), and
    \item \(
        (s_{\mathcal{X},\phi_{\nu,\lambda}}-s_{\tilde{\mathcal{X}},\phi_{\nu,\lambda}})(f)=s_{\mathcal{X},\phi_{\nu,\lambda}}((\textup{id}-s_{\tilde{\mathcal{X}},\phi_{\nu,\lambda}})(f))
    \).
\end{enumerate}
\end{proposition}
\begin{proof}
    Statement 1. follows simply from noticing that \(s_{\tilde{\mathcal{X}},\phi_{\nu,\lambda}}=s_{\tilde{\mathcal{X}},\phi_{\nu,\lambda}}(s_{\mathcal{X},\phi_{\nu,\lambda}}(f))\). For Statement 2., we first note that the difference is a linear combination of \(N\)-many kernels, \(N=|\mathcal{X}|\), and accordingly, there exist unique real-valued \(\{\beta_i\}_{i=1}^N\) such that
    \begin{align}
        (s_{\mathcal{X},\phi_{\nu,\lambda}}-s_{\tilde{\mathcal{X}},\phi_{\nu,\lambda}})(f)=\sum_{i=1}^{N}\beta_{i}\phi_{\nu,\lambda}(\cdot,x_i),\label{eq: a_f}
    \end{align}
    for a given ordering \(\mathcal{X}=\{x_1,\dots,x_N\}\). By definition of kernel interpolant in Definition \ref{def:restricted kernel interpolant}, there also exist unique real-valued \(\{\gamma_i\}_{i=1}^N\) such that
    \begin{align}
        s_{\mathcal{X},\phi_{\nu,\lambda}}((\textup{id}-s_{\tilde{\mathcal{X}},\phi_{\nu,\lambda}})(f))=\sum_{i=1}^N\gamma_i\phi_{\nu,\lambda}(\cdot,x_i).\label{eq: b_f}
    \end{align}
    Let \(1\leq i\leq N\), then,
    \begin{align}
        (s_{\mathcal{X},\phi_{\nu,\lambda}}-s_{\tilde{\mathcal{X}},\phi_{\nu,\lambda}})(f)(x_i) &= s_{\mathcal{X},\phi_{\nu,\lambda}}(f)(x_i) -s_{\tilde{\mathcal{X}},\phi_{\nu,\lambda}}(f)(x_i),\nonumber\\
        &= f(x_i) - s_{\tilde{\mathcal{X}},\phi_{\nu,\lambda}}(f)(x_i),\nonumber
    \end{align}
    and,
    \begin{align}
        s_{\mathcal{X},\phi_{\nu,\lambda}}((\textup{id}-s_{\tilde{\mathcal{X}},\phi_{\nu,\lambda}})(f))(x_i) &= s_{\mathcal{X},\phi_{\nu,\lambda}}(x_i)-s_{\mathcal{X},\phi_{\nu,\lambda}}(s_{\tilde{\mathcal{X}},\phi_{\nu,\lambda}}(f))(x_i),\nonumber\\
        &= f(x_i) - s_{\tilde{\mathcal{X}},\phi_{\nu,\lambda}}(f)(x_i).
     \end{align}
     Therefore, for all \(1\leq i\leq N\), by \eqref{eq: a_f} and \eqref{eq: b_f}, we have
     \begin{align}
         \sum_{i=1}^N\beta_i\phi_{\nu,\lambda}(x,x_i)=\sum_{i=1}^N\gamma_i\phi_{\nu,\lambda}(x,x_i)=f(x_i)-s_{\tilde{\mathcal{X}},\phi_{\nu,\lambda}}(f)(x_i),\nonumber
     \end{align}
     a uniquely determined system, and hence \(\gamma_i=\beta_i\) for all \(1 \leq i \leq N\).
\end{proof}

Our next goal is to outline the relationship between the lengthscale parameter and the native space norm of interpolants in the one-dimensional setting. To do this, we first define stretching operators, and associate them to changing the lengthscales of Mat\'ern kernels and their interpolants. 
\begin{definition}\label{def: stretching operators}
    Let \(a\in\mathbb{R}_{>0}\) and let \(\widehat{\mathcal{X}}=\{x_1,\dots,x_m\}\subset(-a,a)\) be a finite discrete set. For a given \(\lambda\in\mathbb{R}_{>0}\), we define the \emph{point-stretching operator} \(\mathcal{T}_{\lambda}(\{x_1,\dots,x_m\})\coloneqq\{\lambda x_1,\dots,\lambda x_m\}\subset(-\lambda a,\lambda a)\). 
    Similarly, we define the \emph{stretching operator} \(\mathcal{S}_{\lambda}(g):(-\lambda a, \lambda a)\rightarrow\mathbb{R}\) by \(\mathcal{S}_{\lambda}(g)(\lambda x) \coloneqq g(x)\) for all \(x\in(-a,a)\).
\end{definition}
\begin{remark}
    Note that clearly, for finite \(\lambda,a\in\mathbb{R}_{>0}\) and \(\alpha\in\mathbb{R}_{\geq0}\), if \(g\in H^{\alpha}((-\lambda a,\lambda a))\), then \(\mathcal{S}_\lambda(g)\in H^{\alpha}((-a,a))\).
\end{remark}
\begin{proposition}\label{prop: stretching kernels} Let \(\nu\in\mathbb{R}_{\geq1/2}\) and \(\lambda_1,\lambda_2\in \mathbb{R}_{>0}\). Then, for all \(x,x'\in\mathbb{R}\),
\begin{align}
    \phi_{\nu,\lambda_1}(x,x')=\phi_{\nu,\lambda_2}\left(\frac{\lambda_2}{\lambda_1}x,\frac{\lambda_2}{\lambda_1}x'\right).\nonumber
\end{align}
\begin{proof}
    Since Mat\'ern kernels are stationary, we can write a Mat\'ern kernel as a function of the scaled Euclidean distance between the arguments,
    \begin{align}
        \phi_{\nu,\lambda} = \tilde{\phi}_\nu\left(\frac{|x-x'|}{\lambda}\right).\nonumber
    \end{align}
    This is clear from Definition \ref{def: 1D Matern}. Consequently, it follows
    \begin{align}
        \phi_{\nu,\lambda_2}\left(\frac{\lambda_2}{\lambda_1}x,\frac{\lambda_2}{\lambda_1}x'\right) &= \tilde{\phi}_{\nu}\left(\frac{|\frac{\lambda_2}{\lambda_1}x-\frac{\lambda_2}{\lambda_1}x'|}{\lambda_2}\right),\nonumber\\
        &= \tilde{\phi}_{\nu}\left(\frac{|x-x'|}{\lambda_1}\right),\nonumber\\
        &=\phi_{\nu,\lambda_1}(x,x').\nonumber
    \end{align}
\end{proof}
    
\end{proposition}
\begin{lemma}[Stretching kernel interpolants]\label{lem: stretch}
    Let \(a\in\mathbb{R}_{>0}\) and let \(\widehat{\mathcal{X}}=\{x_1,\dots,x_m\}\subset(-a,a)\) be a finite discrete set. For a given bounded function \(f:(-a,a)\rightarrow\mathbb{R}\), regularity \(\nu\in\mathbb{R}_{\geq1/2}\), and lengthscales \(\lambda_1,\lambda_2\in\mathbb{R}_{>0}\), we have
    \begin{align}
        s_{\mathcal{T}_{\lambda_2/\lambda_1}(\widehat{\mathcal{X}}),\phi_{\nu,\lambda_2}}(\mathcal{S}_{\lambda_2/\lambda_1}(f))\left(\frac{\lambda_2}{\lambda_1}x\right)=s_{\widehat{\mathcal{X}},\phi_{\nu,\lambda_1}}(f)(x),\nonumber
    \end{align}
    for all \(x\in(-a,a)\).
\end{lemma}
\begin{proof}
    By Proposition \ref{prop: mean functions are kernel interpolants}, there exists a unique expansion 
    \begin{align}
        s_{\widehat{\mathcal{X}},\phi_{\nu,\lambda_1}}(f)(x)=\sum_{i=1}^m\beta_i\phi_{\nu,\lambda_1}(x,x_i),\nonumber
    \end{align}
    for all \(x\in(-a,a)\), where each \(\beta_i\in\mathbb{R}\) is determined as the solution to the linear system
    \begin{align}
        \sum_{i=1}^m\beta_i\phi_{\nu,\lambda_1}(x_k,x_i) = f(x_k)\label{eq: linear system}
    \end{align}
    for all \(1\leq k \leq m\). Equally, there exists a unique expansion 
    \begin{align}
        s_{\mathcal{T}_{\lambda_2/\lambda_1}(\widehat{\mathcal{X}}),\phi_{\nu,\lambda_2}}(\mathcal{S}_{\lambda_2/\lambda_1}(f))\left(\frac{\lambda_2}{\lambda_1}x\right)=\sum_{i=1}^{m}\gamma_i\phi_{\nu,\lambda_2}\left(\frac{\lambda_2}{\lambda_1}x,\frac{\lambda_2}{\lambda_1}x_i\right),\nonumber
    \end{align}
    where each \(\gamma_i\in\mathbb{R}\) is again determined by the system
    \begin{align}
        \sum_{i=1}^m\gamma_i\phi_{\nu,\lambda_2}\left(\frac{\lambda_2}{\lambda_1}x_k,\frac{\lambda_2}{\lambda_1}x_i\right) = \mathcal{S}_{\lambda_2/\lambda_1}(f)\left(\frac{\lambda_2}{\lambda_1}x_k\right),\nonumber
    \end{align}
    for all \(1\leq k \leq m\). By Proposition \ref{prop: stretching kernels} and the definitions of \(\mathcal{S}_{\lambda}\) and \(\mathcal{T}_{\lambda}\) in Definition \ref{def: stretching operators}, we see that this system reduces to the first. Thus we deduce \(\beta_i=\gamma_i\) for all \(1\leq i \leq m\), and hence, for all \(x\in(-a,a)\)

    \begin{align}
        s_{\widehat{\mathcal{X}},\phi_{\nu,\lambda_1}}(f)(x)&=\sum_{i=1}^m\beta_i\phi_{\nu,\lambda_1}(x,x_i),\nonumber\\
        &= \sum_{i=1}^m\beta_i\phi_{\nu,\lambda_2}\left(\frac{\lambda_2}{\lambda_1}x,\frac{\lambda_2}{\lambda_1}x_i\right),\nonumber\\
        &= s_{\mathcal{T}_{\lambda_2/\lambda_1}(\widehat{\mathcal{X}}),\phi_{\nu,\lambda_2}}(\mathcal{S}_{\lambda_2/\lambda_1}(f))\left(\frac{\lambda_2}{\lambda_1}x\right).\nonumber
    \end{align}
\end{proof}
\begin{lemma}(Native space norm of stretched kernel interpolants)\label{prop: RKHS stretched kernels}
Let \(\Gamma_{p}\coloneqq(-2^{-p-1},2^{-p-1})\) 
 for \(p\in\mathbb{N}_0\) and, for \(l\in\mathbb{N}_0\), let \(\mathcal{X}_l\subset\Gamma_0=(-1/2,1/2)\) be the uniformly spaced point set defined as in Definition \ref{def:chi_l}. Then
\begin{align}
    \|R_{\Gamma_0}(s_{\mathcal{X}_l,\phi_{\nu,2^p}}(f))\|_{\mathcal{N}_{\phi_{\nu,2^p}}(\Gamma_0)}=\|R_{\Gamma_p}(s_{\mathcal{X}_{l+p}\cap\Gamma_p,\phi_{\nu,1}}(\mathcal{S}_{2^{-p}}(f)))\|_{\mathcal{N}_{\phi_{\nu,1}}(\Gamma_p)}.\nonumber
\end{align}
\end{lemma}
\begin{proof}
    Using Proposition \ref{prop: matern extensions} we have that the restricted norm is equal to the unrestricted,
    \begin{align}
        \|R_{\Gamma_0}(s_{\mathcal{X}_l,\phi_{\nu,2^p},\Gamma_0}(f))\|_{\mathcal{N}_{\phi_{\nu,2^p}}}(\Gamma_0)=\|s_{\mathcal{X}_l,\phi_{\nu,2^p}}(f)\|_{\mathcal{N}_{\phi_{\nu,2^p}}}(\mathbb{R}).\nonumber
    \end{align}
    Let \(\{\beta_i\}_{i=1}^{|\mathcal{X}_l|}\) be the unique real-valued coefficients such that, for all \(x\in\mathbb{R}\),
    \begin{align}
        s_{\mathcal{X}_l,\phi_{\nu,2^p}}(f)(x)&=\sum_{i=1}^{|\mathcal{X}_l|}\beta_{i}\phi_{\nu,2^p}(x,x_i),\nonumber
    \end{align}
    whose existence is assured by Proposition \ref{prop: mean functions are kernel interpolants}. By Lemma \ref{lem: stretch}, we have that this is also equal to \(s_{\mathcal{T}_{2^p}(\mathcal{X}_l),\phi_{1}}(\mathcal{S}_{2^{-p}}(f))\). The native space norm of the unrestricted interpolant is then given by
    \begin{align}
        \|s_{\mathcal{X}_l,\phi_{\nu,2^p}}(f)\|_{\mathcal{N}_{\phi_{\nu,2^p}}}(\mathbb{R})&\coloneqq\sum_{i=1}^{|\mathcal{X}_l|}\sum_{k=1}^{|\mathcal{X}_l|}\beta_{i}\beta_{k}\phi_{\nu,2^p}(x_i,x_k),\nonumber\\
        &=\sum_{i=1}^{|\mathcal{X}_l|}\sum_{k=1}^{|\mathcal{X}_l|}\beta_{i}\beta_{k}\phi_{\nu,1}\left(\frac{x_i}{2^p},\frac{x_k}{2^p}\right),\nonumber\\
        &\eqqcolon \|s_{\mathcal{T}_{2^p}(\mathcal{X}_l),\phi_{1}}(\mathcal{S}_{2^{-p}}(f))\|_{\mathcal{N}_{\phi_{\nu,1}}(\mathbb{R})},\nonumber
    \end{align}
    by Proposition \ref{prop: stretching kernels}, where \(\mathcal{T}_{\lambda}\) is defined in Definition \ref{def: stretching operators}.
    By Proposition \ref{prop: matern extensions}, since \(\mathcal{T}_{2^p}(\mathcal{X}_l)\subset\Gamma_p\), we have
    \begin{align}
        \|s_{\mathcal{T}_{2^p}(\mathcal{X}_l),\phi_{1}}(\mathcal{S}_{2^{-p}}(f))\|_{\mathcal{N}_{\phi_{\nu,1}}(\mathbb{R})}=\|R_{\Gamma_p}(s_{\mathcal{T}_{2^p}(\mathcal{X}_l),\phi_{1}}(\mathcal{S}_{2^{-p}}(f)))\|_{\mathcal{N}_{\phi_{\nu,1}}(\Gamma_p)}.\nonumber
    \end{align}
    Finally, we note that
    \begin{align}
        \mathcal{T}_{2^p}(\mathcal{X}_l)&=\mathcal{T}_{2^p}\left(\left\{\frac{n}{2^{l+1}}\in\left(-\frac{1}{2},\frac{1}{2}\right):n\in\mathbb{Z}\right\}\right),\nonumber\\
        &=\left\{\frac{n}{2^{l+p+1}}\in\left(-\frac{1}{2^{p+1}},\frac{1}{2^{p+1}}\right):n\in\mathbb{Z}\right\},\nonumber\\
        &=\mathcal{X}_{l+p}\cap\Gamma_p.\nonumber
    \end{align}
\end{proof}

\begin{corollary}\label{corr: rkhs norm differences}
    Let \(l,p\in\mathbb{N}_0\) and let the \(p\)-penalised point set \(\mathcal{X}_l^p\) be defined as in Definition \ref{def:chi_l^p}. Then, for a given \(f\in\mathcal{N}_{\phi_{\nu,2^p}}(\Gamma_0)\),
    \begin{align}
    \begin{split}
        &\quad\qquad\qquad\|R_{\Gamma_0}((s_{\mathcal{X}^{p}_{l},\phi_{\nu,2^{p}}}-s_{\mathcal{X}^{p}_{l-1},\phi_{\nu,2^{p}}})(f))\|_{\mathcal{N}_{\phi_{\nu,2^p}}(\Gamma_0)} \\
        &=\begin{cases}
            \|R_{\Gamma_p}((s_{\mathcal{X}_{l}\cap\Gamma_p,\phi_{\nu,1}}-s_{\mathcal{X}_{l-1}\cap\Gamma_p,\phi_{\nu,1}})(\mathcal{S}_{2^{-p}}(f)))\|_{\mathcal{N}_{\phi_{\nu,1}}(\Gamma_p)} & \textrm{ if }l\geq p+1,\\
            0 & \textrm{ if }1\leq l \leq p,\\
            \|R_{\Gamma_p}(s_{\{0\},\phi_{\nu,1}}(\mathcal{S}_{2^{-p}}(f)))\|_{\mathcal{N}_{\phi_{\nu,1}}(\Gamma_p)} & \textrm{ if } l=0.
        \end{cases} 
    \end{split}\nonumber
    \end{align}
\end{corollary}
\begin{proof}
    For \(l=0\), we have \((s_{\mathcal{X}^{p}_{l},\phi_{\nu,2^{p}}}-s_{\mathcal{X}^{p}_{l-1},\phi_{\nu,2^{p}}})(f) = s_{\mathcal{X}^{p}_{l},\phi_{\nu,2^{p}}}(f)\), since interpolants of negative level are set to 0. Then apply Lemma \ref{prop: RKHS stretched kernels}, noting that \(\mathcal{X}_{p}\cap\Gamma_p =\{0\}\). For \(1\leq l\leq p\), trivially we have \(\mathcal{X}_l^p=\mathcal{X}^p_{l-1}=\{0\}\), and so the difference is 0. Finally, for \(l\geq p+1\), we can rewrite \((s_{\mathcal{X}^{p}_{l},\phi_{\nu,2^{p}}}-s_{\mathcal{X}^{p}_{l-1},\phi_{\nu,2^{p}}})(f)\) as \(s_{\mathcal{X}^{p}_{l},\phi_{\nu,2^{p}}} ((\textup{id}-s_{\mathcal{X}^{p}_{l-1},\phi_{\nu,2^{p}}})(f))\) by Statement 2, Proposition \ref{prop: rewrite differences}, since \(\mathcal{X}^p_l\subset\mathcal{X}^p_{l-1}\),
    and hence we can again apply Lemma \ref{prop: RKHS stretched kernels}, giving
    \begin{align}  
        \begin{split}
        &\quad\|R_{\Gamma_0}((s_{\mathcal{X}^{p}_{l},\phi_{\nu,2^{p}}}-s_{\mathcal{X}^{p}_{l-1},\phi_{\nu,2^{p}}})(f))\|_{\mathcal{N}_{\phi_{\nu,2^p}}(\Gamma_0)}\\
        &=\|R_{\Gamma_0}(s_{\mathcal{X}^{p}_{l},\phi_{\nu,2^{p}}} ((\textup{id}-s_{\mathcal{X}^{p}_{l-1},\phi_{\nu,2^{p}}})(f))\|_{\mathcal{N}_{\phi_{\nu,2^p}}(\Gamma_0)},
        \end{split}\nonumber\\
        &=\|R_{\Gamma_p}(s_{\mathcal{X}_{l+p}\cap\Gamma_p,\phi_{\nu,1}} (\mathcal{S}_{2^{-p}}((\textup{id}-s_{\mathcal{X}^{p}_{l-1},\phi_{\nu,2^{p}}})(f)))\|_{\mathcal{N}_{\phi_{\nu,1}}(\Gamma_p)}.\nonumber
    \end{align}
    It follows that
    \[
    \mathcal{S}_{2^{-p}}((\textup{id}-s_{\mathcal{X}^{p}_{l-1},\phi_{\nu,2^{p}}})(f))=(\textup{id}-s_{\mathcal{X}_{l+p-1}\cap\Gamma_p,\phi_{\nu,1}})(\mathcal{S}_{2^{-p}}(f)),
    \]
    and by Proposition \ref{prop: rewrite differences} we recover the result.
\end{proof}
Through Proposition \ref{prop: tensor product norm}, these one-dimensional results allow us to equate the error of our operator with an anisotropic kernel \(\Phi_{\boldsymbol{\nu},2^{\mathbf{p}}}\) over the isotropic domain \(\Gamma^d\) to the error of a corresponding `stretched' operator with an isotropic kernel \(\Phi_{\boldsymbol{\nu}, \mathbf{1}}\) over an anisotropic domain \(\prod_{j=1}^d\Gamma_{p_j}\). Smaller intervals clearly require fewer points to maintain the same fill-distance, and it is in this way that we have found a link between increasing the lengthscale and reducing the approximation error. In order to apply the error bounds in Proposition \ref{prop: initial wendland}, we must move into the appropriate Sobolev norm. For this, we use the following equivalence relationship.

\begin{proposition}\label{prop: equivalence}
Let \(\lambda \in\mathbb{R}_{>0}\), \(\nu+1/2\in\mathbb{R}_{\geq1/2}\) and, for \(f\in H^{\nu+1/2}(\mathbb{R})\), define the one-dimensional Sobolev norm by
\begin{align}
    \|f\|_{H^{\nu+1/2}(\mathbb{R})}\coloneqq\int_{\mathbb{R}}|\widehat{f}(\omega)|^2(1+|w|_2^2)^{\nu+1/2}d\omega.\nonumber
\end{align}
The following norm-equivalence relation holds,
\begin{align}
    \|f\|_{H^{\nu+1/2}(\mathbb{R})}^2=\sigma^2\frac{\Gamma(\nu+1/2)\lambda}{\Gamma(\nu)\pi^{1/2}}\frac{\int_\mathbb{R}|\widehat{f}(\omega)|^2(1+|\omega|^2)^{\nu+1/2}d\omega}{\int_\mathbb{R}|\widehat{f}(\omega)|^2(1+\lambda^2|\omega|^2)^{\nu+1/2}d\omega}\|f\|_{\mathcal{N}_{\phi_{\nu,\lambda}}(\mathbb{R})}^2.\nonumber
\end{align}
This relationship further holds for restrictions to compact sets \(\Omega\subset\mathbb{R}\).
\end{proposition}
\begin{proof}
    Follows from Theorem 10.12, \cite{wendland_2004}, using the Fourier transform of Mat\'ern kernels given in Example 7.17, \cite{Lord_Powell_Shardlow_2014}, as done in the proof of Lemma 3.4, \cite{Teckentrup2020}.
\end{proof}
\begin{remark} \label{rem: change norm}
    In particular, we note that for \(\lambda=1\) this relationship simplifies considerably to 
    \begin{align}
        \|f\|_{H^{\nu+1/2}(\mathbb{R})}^2=\sigma^2\frac{\Gamma(\nu+1/2)}{\Gamma(\nu)\pi^{1/2}}\|f\|_{\mathcal{N}_{\phi_{\nu,1}}(\mathbb{R})}^2.\nonumber
    \end{align}
\end{remark}
We give an alternative formulation of the Sobolev approximation error in Proposition \ref{prop: initial wendland}, specifically for one-dimensional interpolation operators defined on the sets \(\mathcal{X}^p_l\).
\begin{lemma}(Section 3, Proposition 4, \cite{Nobile2018})\label{thm: Wendland norm}
    Let \(\nu\geq\alpha\geq1/2\), then
    \begin{equation}\label{eq: wendland constant}
        \|R_{\Gamma_0}(\textup{id}-s_{\mathcal{X}_l,\phi_{\nu,\lambda}})\|_{H^{\nu+1/2}(\Gamma_0)\rightarrow H^{\alpha+1/2}(\Gamma_0)}\leq C^{(\nu,\alpha)}_{\ref{thm: Wendland norm}}h^{\nu-\alpha}_{\mathcal{X}_{l},\Gamma_0},
    \end{equation}
    where the fill-distance is now given by
    \(h_{\mathcal{X}_{l},\Gamma_0}=2^{-(l+1)}\), 
    and \(C^{(\nu,\alpha)}_{\ref{thm: Wendland norm}}\) is independent of \({l}\).
\end{lemma}

The following lemmas first combine the previous results to derive an error bound for the corresponding `stretched' problem, and then address the awkward \(\mathbf{Q}_{\boldsymbol{\nu},\boldsymbol{\alpha}}\) term that arises by employing Corollary \ref{cor: isomporphic} in Appendix \ref{appendix: a}. Finally, we prove Theorem \ref{thm: error in L} for the case where \(\boldsymbol{\nu}-\boldsymbol{\alpha}=c\mathbf{1}\) for some \(c\in\mathbb{R}_{\geq0}\). As discussed in Section~\ref{subsec: results}, these methods are also applicable to settings where \(\boldsymbol{\nu}\) and \(\boldsymbol{\alpha}\) do not satisfy this relationship; for conciseness, we leave that to the reader.

\begin{lemma}\label{lem: tensor hilbert}
Let \(L\in\mathbb{N}\), \(\mathbf{p}\in\mathbb{N}_0^d\), \(\Gamma\coloneqq (-1/2,1/2)\) and let \(\boldsymbol{\nu},\boldsymbol{\alpha}\in\mathbb{R}^d_{\geq1/2}\) be such that \(\alpha_j\leq\nu_j\) for all \(1\leq j \leq d\). Define the multi-index set \(\mathcal{K}^d_k\coloneqq\{\boldsymbol{l}\in\mathbb{N}_{0}^d:|\{1\leq j\leq d:l_j\neq0\}|=k\}\). Then,
    \begin{align}
    \|I-P_{L,\mathbf{p},\Phi_{\boldsymbol{\nu},2^{\mathbf{p}}}}\|_{\mathcal{N}_{\Phi_{\boldsymbol{\nu},2^{\mathbf{p}}}}(\Gamma^d)\rightarrow\mathcal{N}_{\Phi_{\boldsymbol{\alpha},2^{\mathbf{p}}}}(\Gamma^d)}\leq C^{(\boldsymbol{\nu},\boldsymbol{\alpha})}_{\ref{lem: tensor hilbert}}\sum_{k=1}^d\sum_{\boldsymbol{l}\in\mathcal{K}_k^d\setminus \mathcal{I}^d_L}\prod_{j=1}^d\mathbf{Q}_{\boldsymbol{\nu},\boldsymbol{\alpha}}(\boldsymbol{l},\mathbf{p})_j,\nonumber
    \end{align}
    where
    \begin{align}
        \mathbf{Q}_{\boldsymbol{\nu},\boldsymbol{\alpha}}(\boldsymbol{l},\mathbf{p})_j \coloneqq \begin{cases}
            C^{(\nu_j,\alpha_j)}_{\ref{thm: Wendland norm}}2^{-(\nu_j-\alpha_j)l_j} & \textrm{ if }l_j\geq p_j+1,\\
            0 & \textrm{ if }1\leq l_j\leq p_j, \textrm{ and}\\
            1 & \textrm{ if }l_j = 0,
        \end{cases}\nonumber
    \end{align}
    with constant \(C^{(\boldsymbol{\nu},\boldsymbol{\alpha})}_{\ref{lem: tensor hilbert}}\) given by
    \begin{align}
        C^{(\boldsymbol{\nu},\boldsymbol{\alpha})}_{\ref{lem: tensor hilbert}} = \prod_{j=1}^d\sqrt{\frac{\Gamma(\alpha_j+1/2)\Gamma(\nu_j)}{\Gamma(\alpha_j)\Gamma(\nu_j+1/2)}}.\nonumber
    \end{align}

\end{lemma}

\begin{proof}
    As before, by Lemma 2.2, \cite{Nobile2018}, we can write the error operator as the sum over all multi-indices \textit{not} included in Definition \ref{def: penalised sparse grid operator}, that is,
    \begin{align}
        I-P_{L,\mathbf{p},\Phi_{\boldsymbol{\nu},2^{\mathbf{p}}}} = \sum_{\boldsymbol{l}\in\mathbb{N}_0^d\setminus \mathcal{I}^d_L}\bigotimes_{j=1}^d\left(R_{\Gamma}\circ\left(s_{\mathcal{X}^{p_j}_{l_j},\phi_{\nu_j,2^{p_j}}}-s_{\mathcal{X}^{p_j}_{l_j-1},\phi_{\nu_j,2^{p_j}}}\right)\right).\label{eq: split error operator}
    \end{align}
    Clearly we have \(\mathcal{K}_0^d\cup\mathcal{K}_1^d\cup\cdots\cup\mathcal{K}_d^d=\mathbb{N}_0^d\) and \(\mathcal{K}_j^d\cap\mathcal{K}_{j'}^d=\emptyset\) for all \(j\neq j'\), and so we can rewrite \eqref{eq: split error operator} as
    \begin{align}
         I-P_{L,\mathbf{p},\Phi_{\boldsymbol{\nu},2^{\mathbf{p}}}}= \sum_{k=0}^d\sum_{\boldsymbol{l}\in\mathcal{K}_k^d\setminus \mathcal{I}^d_L}\bigotimes_{j=1}^d\left(R_{\Gamma}\circ\left(s_{\mathcal{X}^{p_j}_{l_j},\phi_{\nu_j,2^{p_j}}}-s_{\mathcal{X}^{p_j}_{l_j-1},\phi_{\nu_j,2^{p_j}}}\right)\right).\nonumber
    \end{align}
    Since \(\mathcal{K}_0^d=\{\mathbf{0}\}\subset\mathcal{I}^d_L\), we may ignore the case \(k=0\). By the triangle inequality we have
    \begin{align}
        \begin{split}
         &\qquad\qquad\|I-P_{L,\mathbf{p},\Phi_{\boldsymbol{\nu},2^{\mathbf{p}}}}\|_{\mathcal{N}_{\Phi_{\boldsymbol{\nu},2^{\mathbf{p}}}}(\Gamma)\rightarrow\mathcal{N}_{\Phi_{\boldsymbol{\alpha},2^{\mathbf{p}}}}(\Gamma)}\\
         &\leq\sum_{k=1}^d\sum_{\boldsymbol{l}\in\mathcal{K}^d_k\setminus \mathcal{I}^d_L}\left\|\bigotimes_{j=1}^dR_{\Gamma}\circ\left(s_{\mathcal{X}^{p_j}_{l_j},\phi_{\nu_j,2^{p_j}}}-s_{\mathcal{X}^{p_j}_{l_j-1},\phi_{\nu_j,2^{p_j}}}\right)\right\|_{\mathcal{N}_{\Phi_{\boldsymbol{\nu},2^{\mathbf{p}}}}(\Gamma)\rightarrow\mathcal{N}_{\Phi_{\boldsymbol{\alpha},2^{\mathbf{p}}}}(\Gamma)},
         \end{split}\nonumber\\
         &=\sum_{k=1}^d\sum_{\boldsymbol{l}\in\mathcal{K}^d_k\setminus \mathcal{I}^d_L}\prod_{j=1}^d\left\| R_{\Gamma}\circ\left(s_{\mathcal{X}^{p_j}_{l_j},\phi_{\nu_j,2^{p_j}}}-s_{\mathcal{X}^{p_j}_{l_j-1},\phi_{\nu_j,2^{p_j}}}\right)\right\|_{\mathcal{N}_{\phi_{{\nu_j},2^{{p_j}}}}(\Gamma)\rightarrow\mathcal{N}_{\phi_{{\alpha_j},2^{{p_j}}}}(\Gamma)},\nonumber\\
         &\eqqcolon\sum_{k=1}^d\sum_{\boldsymbol{l}\in\mathcal{K}^d_k\setminus \mathcal{I}^d_L}\prod_{j=1}^d\sup_{\|f\|_{\mathcal{N}_{\phi_{\nu_j,2^{p_j}}}(\Gamma)}=1}\left\|R_{\Gamma}\left(\left(s_{\mathcal{X}^{p_j}_{l_j},\phi_{\nu_j,2^{p_j}}}-s_{\mathcal{X}^{p_j}_{l_j-1},\phi_{\nu_j,2^{p_j}}}\right)(f)\right)\right\|_{\mathcal{N}_{\phi_{{\alpha_j},2^{{p_j}}}}(\Gamma)},\label{eq: 1st case pause 1}
    \end{align}
    by Proposition \ref{prop: tensor product norm} since \(\mathcal{N}_{\Phi_{\boldsymbol{\nu},2^{\mathbf{p}}}}(\Gamma^d)\coloneqq\bigotimes_{j=1}^d\mathcal{N}_{\phi_{{\nu},2^{{p_j}}}}(\Gamma)\), where \(\Gamma\coloneqq\Gamma_0=(-1/2,1/2)\), is indeed a tensor product of Hilbert spaces. We now consider the three cases presented in Corollary \ref{corr: rkhs norm differences}. Firstly, for \(l\geq p+1\), we have
    \begin{align}
        &\quad\,\left\|R_{\Gamma_0}\left(\left(s_{\mathcal{X}^{p_j}_{l_j},\phi_{\nu_j,2^{p_j}}}-s_{\mathcal{X}^{p_j}_{l_j-1},\phi_{\nu_j,2^{p_j}}}\right)(f)\right)\right\|_{\mathcal{N}_{\phi_{{\alpha_j},2^{{p_j}}}}(\Gamma_0)}\nonumber\\
        &=\left\|R_{\Gamma_{p_j}}\left(\left(s_{\mathcal{X}_{l_j}\cap\Gamma_{p_j},\phi_{\nu_j,1}}-s_{\mathcal{X}_{l_j-1}\cap\Gamma_{p_j},\phi_{\nu_j,1}}\right)\left(\mathcal{S}_{2^{-p}}(f)\right)\right)\right\|_{\mathcal{N}_{\phi_{{\alpha_j},1}}(\Gamma_{p_j})}.\label{eq: 1st case pause 2}   
    \end{align}
     Using Statement 2 in Proposition \ref{prop: rewrite differences}, since \(\mathcal{X}_{l_j-1}\subset\mathcal{X}_{l_j}\), we rewrite the difference between successive interpolants in terms of the identity operator, change norm using the equality in Remark \ref{rem: change norm}, and then apply the Sobolev error bound in Lemma \ref{thm: Wendland norm}:
    \begin{align}
       \begin{split} &\quad\,\left\|R_{\Gamma_{p_j}}\left(\left(s_{\mathcal{X}_{l_j}\cap\Gamma_{p_j},\phi_{\nu_j,1}}-s_{\mathcal{X}_{l_j-1}\cap\Gamma_{p_j},\phi_{\nu_j,1}}\right)(\mathcal{S}_{2^{-p}}(f))\right)\right\|_{\mathcal{N}_{\phi_{{\alpha_j},1}}(\Gamma_{p_j})}\nonumber\\
        &=\left\|R_{\Gamma_{p_j}}\left(\left(\textup{id}-s_{\mathcal{X}_{l_j-1}\cap\Gamma_{p_j},\phi_{\nu_j,1}}\right)\left(s_{\mathcal{X}_{l_j}\cap\Gamma_{p_j},\phi_{\nu_j,1}}(\mathcal{S}_{2^{-p}}(f))\right)\right)\right\|_{\mathcal{N}_{\phi_{{\alpha_j},1}}(\Gamma_{p_j})},
        \end{split}\nonumber\\
        &=\sigma_j\sqrt{\frac{\Gamma(\alpha_j+1/2)}{\Gamma(\alpha_j)\pi^{1/2}}}\left\|R_{\Gamma_{p_j}}\left(\left(\textup{id}-s_{\mathcal{X}_{l_j-1}\cap\Gamma_{p_j},\phi_{\nu_j,1}}\right)\left(s_{\mathcal{X}_{l_j}\cap\Gamma_{p_j},\phi_{\nu_j,1}}(\mathcal{S}_{2^{-p}}(f))\right)\right)\right\|_{H^{\alpha_j+1/2}(\Gamma_{p_j})},\nonumber\\
        &\leq\sigma_j\sqrt{\frac{\Gamma(\alpha_j+1/2)}{\Gamma(\alpha_j)\pi^{1/2}}}C^{(\nu_j,\alpha_j)}_{\ref{thm: Wendland norm}}h^{\nu_j-\alpha_j}_{\mathcal{X}_{l_{j-1}\cap\Gamma_{p_j}},\Gamma_{p_j}}\left\|R_{\Gamma_{p_j}}\left(s_{\mathcal{X}_{l_j}\cap\Gamma_{p_j},\phi_{\nu_j,1}}(\mathcal{S}_{2^{-p}}(f))\right)\right\|_{H^{\nu_j+1/2}(\Gamma_{p_j})}.\label{eq: 1stcase pause 3}
    \end{align}
    The fill-distance here is given by \(h_{\mathcal{X}_{l_j}\cap\Gamma_{p_j},\Gamma_{p_j}}=h_{\mathcal{X}_{l_j-1},\Gamma_0}=2^{-l_j}\). We then argue away the supremum of the remaining norm by reversing the previous steps; first moving back to the native space norm using Remark~\ref{rem: change norm}, and then stretching back to the unit domain with Lemma~\ref{prop: RKHS stretched kernels}.
    \begin{align}
        \begin{split}
        &\quad\,\left\|R_{\Gamma_{p_j}}\left(s_{\mathcal{X}_{l_j}\cap\Gamma_{p_j},\phi_{\nu_j,1}}(\mathcal{S}_{2^{-p}}(f))\right)\right\|_{H^{\nu_j+1/2}(\Gamma_{p_j})}\\
        &=\left(\sigma_j\sqrt{\frac{\Gamma(\nu_j+1/2)}{\Gamma(\nu_j)\pi^{1/2}}}\right)^{-1}\left\|R_{\Gamma_{p_j}}\left(s_{\mathcal{X}_{l_j}\cap\Gamma_{p_j},\phi_{\nu_j,1}}(\mathcal{S}_{2^{-p}}(f))\right)\right\|_{\mathcal{N}_{\phi_{\nu_j,1}}(\Gamma_{p_j})},
        \end{split}\nonumber\\
        &=\left(\sigma_j\sqrt{\frac{\Gamma(\nu_j+1/2)}{\Gamma(\nu_j)\pi^{1/2}}}\right)^{-1}\left\|R_{\Gamma_{0}}\left(s_{\mathcal{X}_{l_j-p_j},\phi_{\nu_j,2^{p_j}}}(f)\right)\right\|_{\mathcal{N}_{\phi_{\nu_j,2^{p_j}}}(\Gamma_{0})}.\nonumber
    \end{align}
    Then, by the definition of kernel interpolant in Definition \ref{def:restricted kernel interpolant} as the interpolant of minimum native space norm, we have
    \begin{align}
        \sup_{\|f\|_{\mathcal{N}_{\phi_{\nu_j,2^{p_j}}}(\Gamma)}=1}\left\|R_{\Gamma_{0}}\left(s_{\mathcal{X}_{l_j-p_j},\phi_{\nu_j,2^{p_j}}}(f)\right)\right\|_{\mathcal{N}_{\phi_{\nu_j,2^{p_j}}}(\Gamma_{0})}\leq 1,\label{eq: 1st case end}
    \end{align}
    since \(\Gamma = \Gamma_0\). Combining \eqref{eq: 1st case pause 1}, \eqref{eq: 1st case pause 2}, \eqref{eq: 1stcase pause 3} and \eqref{eq: 1st case end} gives us the first case. By Corollary \ref{corr: rkhs norm differences}, the case \(0\leq l_j\leq p_j\) is trivial and, for the case \(l_j=0\), we have
    \begin{align}
        \begin{split}
        &\quad\,\left\|R_{\Gamma_0}\left(\left(s_{\mathcal{X}^{p_j}_{l_j},\phi_{\nu_j,2^{p_j}}}-s_{\mathcal{X}^{p_j}_{l_j-1},\phi_{\nu_j,2^{p_j}}}\right)(f)\right)\right\|_{\mathcal{N}_{\phi_{{\alpha_j},2^{{p_j}}}}(\Gamma_0)}\\
        &=\|R_{\Gamma_{p_j}}(s_{\{0\},\phi_{\nu_j,1}}(\mathcal{S}_{2^{p_j}}(f)))\|_{\mathcal{N}_{\phi_{\alpha_j,1}}(\Gamma_{p_j})},
        \end{split}\nonumber\\
        &=\left(\sigma_j\sqrt{\frac{\Gamma(\alpha_j+1/2)}{\Gamma(\alpha_j)\pi^{1/2}}}\right)\|R_{\Gamma_{p_j}}(s_{\{0\},\phi_{\nu_j,1}}(\mathcal{S}_{2^{-p}}(f)))\|_{H^{\alpha_j+1/2}(\Gamma_{p_j})},\nonumber\\
        &\leq\left(\sigma_j\sqrt{\frac{\Gamma(\alpha_j+1/2)}{\Gamma(\alpha_j)\pi^{1/2}}}\right)\|R_{\Gamma_{p_j}}(s_{\{0\},\phi_{\nu_j,1}}(\mathcal{S}_{2^{-p}}(f)))\|_{H^{\nu_j+1/2}(\Gamma_{p_j})},\nonumber\\
        &=\left(\sqrt{\frac{\Gamma(\alpha_j+1/2)\Gamma(\nu_j)}{\Gamma(\alpha_j)\Gamma(\nu_j+1/2)}}\right)\|R_{\Gamma_{p_j}}(s_{\{0\},\phi_{\nu_j,1}}(\mathcal{S}_{2^{-p}}(f)))\|_{\mathcal{N}_{\phi_{\nu_j,1}}(\Gamma_{p_j})},\nonumber\\
        &=\left(\sqrt{\frac{\Gamma(\alpha_j+1/2)\Gamma(\nu_j)}{\Gamma(\alpha_j)\Gamma(\nu_j+1/2)}}\right)\|R_{\Gamma_0}(s_{\{0\},\phi_{\nu_j,2^{p_j}}}(f))\|_{\mathcal{N}_{\phi_{{\nu_j},2^{{p_j}}}}(\Gamma_0)},\nonumber
    \end{align}
    where we have used both Lemma \ref{prop: RKHS stretched kernels} and the equivalence relation in Remark \ref{rem: change norm}, twice. We then apply \eqref{eq: 1st case end} again to arrive at the final result. 
\end{proof}

\begin{proposition}\label{prop: bounded constant}
    For \(\boldsymbol{\nu},\boldsymbol{\alpha}\in\mathbb{R}^d_{\geq1/2}\) with \(\nu_j\geq\alpha_j\) for all \(1\leq j\leq d\), we have \(C^{(\boldsymbol{\nu},\boldsymbol{\alpha})}_{\ref{lem: tensor hilbert}}\leq1\)
\end{proposition}
\begin{proof}
    We have that the ratio \(\Gamma(\beta+1/2)/\Gamma(\beta)\) is strictly increasing for \(\beta\geq1/2\) if and only if
    \begin{align}
        \frac{\mathrm{d}}{\mathrm{d}\beta}\ln\Gamma(\beta+1/2)-\frac{\mathrm{d}}{\mathrm{d}\beta}\ln\Gamma(\beta)>0\nonumber
    \end{align}
    for all \(\beta\geq1/2\). This is true since the digamma function \(\frac{\mathrm{d}}{\mathrm{d}x}\ln\Gamma(x)\) is strictly increasing for real \(x>0\), \cite{Gautschi}. Therefore, for all \(\nu\geq\alpha\geq1/2\), we have
    \begin{align}
        \frac{\Gamma(\alpha+1/2)}{\Gamma(\alpha)} \leq\frac{\Gamma(\nu+1/2)}{\Gamma(\nu)}\nonumber
    \end{align}
    and hence,
    \begin{align}
        \frac{\Gamma(\alpha+1/2)\Gamma(\nu)}{\Gamma(\alpha)\Gamma(\nu+1/2)} \leq 1.\nonumber
    \end{align}
    For \(\boldsymbol{\nu}\) and \(\boldsymbol{\alpha}\) as stated above, it then immediately follows that
    \begin{align}
        C^{(\boldsymbol{\nu},\boldsymbol{\alpha})}_{\ref{lem: tensor hilbert}}\coloneqq \prod_{j=1}^d\sqrt{\frac{\Gamma(\alpha_j+1/2)\Gamma(\nu_j)}{\Gamma(\alpha_j)\Gamma(\nu_j+1/2)}}\leq1.\nonumber
    \end{align}
\end{proof}

\begin{lemma}\label{lem: supersets}
    Let \(L\in\mathbb{N}_{0}\), \(\mathbf{p}\in\mathbb{N}_0^d\), and let \(\boldsymbol{\nu},\boldsymbol{\alpha}\in\mathbb{R}^d_{\geq1/2}\) be such that \(\alpha_j\leq\nu_j\) for all \(1\leq j \leq d\). Then,
    \begin{align}
        \sum_{\boldsymbol{l}\in\mathcal{K}^d_k\setminus \mathcal{I}^d_L}\prod_{j=1}^d{\mathbf{Q}_{\boldsymbol{\nu},\boldsymbol{\alpha}}(\boldsymbol{l},\mathbf{p})_j} = \sum_{\mathfrak{u}\in\mathcal{P}^d_k}C_{\ref{lem: supersets}}^{(\boldsymbol{\nu}_{\mathfrak{u}},\boldsymbol{\alpha}_{\mathfrak{u}})}\sum_{\boldsymbol{l}\in\mathbb{N}_0^k\setminus\mathcal{I}_{\max\{0,L-|\mathbf{p}_{\mathfrak{u}}|-k\}}^k}2^{-(\boldsymbol{\nu}_{\mathfrak{u}}-\boldsymbol{\alpha}_{\mathfrak{u}})\cdot(\boldsymbol{l}+\mathbf{p}_{\mathfrak{u}}+\mathbf{1})},\nonumber
    \end{align}
    where \(\mathcal{P}^d_k\coloneqq\{\mathfrak{u}\subset\{1,\dots,d\}\,:\,|\mathfrak{u}|=k, \mathfrak{u}_i<\mathfrak{u}_{i+1}\}\) and \(C_{\ref{lem: supersets}}^{(\boldsymbol{\nu}_{\mathfrak{u}},\boldsymbol{\alpha}_{\mathfrak{u}})}\coloneqq\prod_{i=1}^kC_{\ref{thm: Wendland norm}}^{(\nu_{\mathfrak{u}_i},\alpha_{\mathfrak{u}_i})}\).
\end{lemma}
\begin{proof}
    By Corollary \ref{cor: isomporphic}, we have 
    \(\boldsymbol{\rho}:\mathcal{K}^d_k\setminus\mathcal{I}_L^d\rightarrow\mathcal{P}^d_k\times\mathbb{N}^k\setminus\mathcal{I}^k_L\) is bijective, and hence
    \begin{align}\label{eq: transformed indices}
        \sum_{\boldsymbol{l}\in\mathcal{K}^d_k\setminus \mathcal{I}^d_L}\prod_{j=1}^d{\mathbf{Q}_{\boldsymbol{\nu},\boldsymbol{\alpha}}(\boldsymbol{l},\mathbf{p})}_j=\sum_{\mathfrak{u}\in\mathcal{P}^d_k}\sum_{\mathbf{a}\in\mathbb{N}^k\setminus\mathcal{I}_L^k}\prod_{j=1}^d{\mathbf{Q}_{\boldsymbol{\nu},\boldsymbol{\alpha}}(\boldsymbol{\rho}^{-1}(\mathfrak{u},\mathbf{a}),\mathbf{p})}_j.
    \end{align}
    We also have, by definition of \(\boldsymbol{\rho}\), and since \(a_i>0\) for all \(1\leq i\leq k\), that \(\boldsymbol{\rho}^{-1}(\mathfrak{u},\mathbf{a})_j=0\) if and only if \(j\notin \mathfrak{u}\). Hence we can split the product in \eqref{eq: transformed indices} as
    \begin{align}
        &=\sum_{\mathfrak{u}\in\mathcal{P}^d_k}\sum_{\mathbf{a}\in\mathbb{N}^k\setminus\mathcal{I}_L^k}\left(\prod_{j\notin \mathfrak{u}}\mathbf{Q}_{\boldsymbol{\nu},\boldsymbol{\alpha}}(\boldsymbol{\rho}^{-1}(\mathfrak{u},\mathbf{a}),\mathbf{p})_j\right)\left(\prod_{j\in \mathfrak{u}}\mathbf{Q}_{\boldsymbol{\nu},\boldsymbol{\alpha}}(\boldsymbol{\rho}^{-1}(\mathfrak{u},\mathbf{a}),\mathbf{p})_j\right),\nonumber\\
        &=\sum_{\mathfrak{u}\in\mathcal{P}^d_k}\sum_{\mathbf{a}\in\mathbb{N}^k\setminus\mathcal{I}_L^k}\prod_{j\in \mathfrak{u}}\mathbf{Q}_{\boldsymbol{\nu},\boldsymbol{\alpha}}(\boldsymbol{\rho}^{-1}(\mathfrak{u},\mathbf{a}),\mathbf{p})_j,\nonumber
    \end{align}
    since \(\mathbf{Q}_{\boldsymbol{\nu},\boldsymbol{\alpha}}(\boldsymbol{\rho}^{-1}(\mathfrak{u},\mathbf{a}),\mathbf{p})_j=1\) when \(\boldsymbol{\rho}^{-1}(\mathfrak{u},\mathbf{a})_j=0\),
    \begin{align}
        &=\sum_{\mathfrak{u}\in\mathcal{P}^d_k}\sum_{\mathbf{a}\in\mathbb{N}^k\setminus\mathcal{I}_L^k}\prod_{i=1}^k{\mathbf{Q}_{\boldsymbol{\nu},\boldsymbol{\alpha}}(\boldsymbol{\rho}^{-1}(\mathfrak{u},\mathbf{a}),\mathbf{p})_{\mathfrak{u}}}_i,\nonumber\\
        &=\sum_{\mathfrak{u}\in\mathcal{P}^d_k}\sum_{\mathbf{a}\in\mathbb{N}^k\setminus\mathcal{I}_L^k}\prod_{i=1}^k{\mathds{1}_{\{a_i> (\mathbf{p}_{\mathfrak{u}})_i\}}{\mathbf{Q}_{\boldsymbol{\nu},\boldsymbol{\alpha}}(\boldsymbol{\rho}^{-1}(\mathfrak{u},\mathbf{a}),\mathbf{p})_{\mathfrak{u}}}_i},\nonumber
    \end{align}
    since the contributions \({\mathbf{Q}_{\boldsymbol{\nu},\boldsymbol{\alpha}}(\boldsymbol{\rho}^{-1}(\mathfrak{u},\mathbf{a}),\mathbf{p})_{\mathfrak{u}}}_i\) are zero for all \((\boldsymbol{\rho}^{-1}(\mathfrak{u},\mathbf{a})_{\mathfrak{u}})_i\coloneqq a_i\leq(\mathbf{p}_{\mathfrak{u}})_i\), \(1\leq i\leq k\). By restricting the indices to only those with non-zero contributions by excluding the set \(\{\mathbf{a}\in\mathbb{N}^k:a_i\leq (\mathbf{p}_{\mathfrak{u}})_i\}\), we have
    \begin{align}
        &=\sum_{\mathfrak{u}\in\mathcal{P}^d_k}\sum_{\substack{\mathbf{a}\in\mathbb{N}^k\setminus\mathcal{I}_L^k\\a_i> p_{\mathfrak{u}_i}}}\prod_{i=1}^k{\mathbf{Q}_{\boldsymbol{\nu},\boldsymbol{\alpha}}(\boldsymbol{\rho}^{-1}(\mathfrak{u}),\mathbf{p})_{\mathfrak{u}}}_i,\nonumber\\
        &=\sum_{\mathfrak{u}\in\mathcal{P}^d_k}\sum_{\substack{\mathbf{a}\in\mathbb{N}^k\setminus\mathcal{I}_L^k\\a_i > p_{\mathfrak{u}_i}}}\prod_{i=1}^kC_{\ref{thm: Wendland norm}}^{(\nu_{\mathfrak{u}_i},\alpha_{\mathfrak{u}_i})}2^{-(\nu_{\mathfrak{u}_i}-\alpha_{\mathfrak{u}_i})
        a_i},\nonumber\\
        &=\sum_{\mathfrak{u}\in\mathcal{P}^d_k}C_{\ref{lem: supersets}}^{(\boldsymbol{\nu}_{\mathfrak{u}},\boldsymbol{\alpha}_{\mathfrak{u}})}\sum_{\substack{\mathbf{a}\in\mathbb{N}^k\setminus\mathcal{I}_L^k\\a_i > p_{\mathfrak{u}_i}}}2^{-\sum_{i=1}^k({\nu}_{\mathfrak{u}_i}-{\alpha}_{\mathfrak{u}_i})a_i},\nonumber
    \end{align}
    and, finally, considering a change of indices, \(l_i = a_i-p_{\mathfrak{u}_i}-1\), we have
    \begin{align}
        =\sum_{\mathfrak{u}\in\mathcal{P}^d_k}C_{\ref{lem: supersets}}^{(\boldsymbol{\nu}_{\mathfrak{u}},\boldsymbol{\alpha}_{\mathfrak{u}})}\sum_{\boldsymbol{l}\in\mathbb{N}_0^k\setminus\mathcal{I}_{\max\{0,L-|\mathbf{p}_{\mathfrak{u}}|-k\}}^k}2^{-(\boldsymbol{\nu}_{\mathfrak{u}}-\boldsymbol{\alpha}_{\mathfrak{u}})\cdot(\boldsymbol{l}+\mathbf{p}_{\mathfrak{u}}+\mathbf{1})}.\nonumber
    \end{align}
\end{proof}
\begin{proof}[Proof of Theorem \ref{thm: error in L}]
    First we derive the Sobolev approximation error for interpolating with Mat\'ern kernels on an isotropic sparse grid, as established in \cite{Nobile2018}. Let \(L>0\), \(\boldsymbol{\nu},\boldsymbol{\alpha}\in\mathbb{R}^d_{\geq1/2}\), and let \(\mathfrak{u}\in\mathcal{P}_k^d\). Then, since \(S_{L,\Phi_{\nu_{\mathfrak{u}}}}=P_{L,\mathbf{0},\Phi_{\boldsymbol{\nu}_{\mathfrak{u}}}}\), following the proof of Lemma \ref{lem: tensor hilbert} we have
    \begin{align}
        \begin{split}
        &\quad\,\,\|I-S_{L,\Phi_{\boldsymbol{\nu}_{\mathfrak{u}}}}\|_{H_{\textrm{mix}}^{\boldsymbol{\nu}_{\mathfrak{u}}+1/2}(\Gamma^k)\rightarrow H_{\textrm{mix}}^{\boldsymbol{\alpha}_{\mathfrak{u}}+1/2}(\Gamma^k)} \\    &\leq\sum_{\boldsymbol{l}\in\mathbb{N}_0^k\setminus\mathcal{I}^k_L}\prod_{i=1}^k\|R_\Gamma\left(s_{\mathcal{X}_l,\phi_{{\nu_{\mathfrak{u}}}_i,1}}-s_{\mathcal{X}_{l-1},\phi_{{\nu_{\mathfrak{u}}}_i,1}}\right)\|_{H^{{\nu_{\mathfrak{u}}}_i+1/2}(\Gamma)\rightarrow H^{{\alpha_{\mathfrak{u}}}_i+1/2}(\Gamma)},
        \end{split}\nonumber\\
        &\leq\sum_{\boldsymbol{l}\in\mathbb{N}_0^k\setminus\mathcal{I}^k_L}\prod_{i=1}^kC^{({\nu_{\mathfrak{u}}}_i,{\alpha_{\mathfrak{u}}}_i)}_{\ref{thm: Wendland norm}}h_{\mathcal{X}_{l-1},\Gamma}^{\nu_{{\mathfrak{u}}_i}-\alpha_{{\mathfrak{u}}_i}},\nonumber\\
        &=C_{\ref{lem: supersets}}^{(\boldsymbol{\nu}_{\mathfrak{u}},\boldsymbol{\alpha}_{\mathfrak{u}})}\sum_{\boldsymbol{l}\in\mathbb{N}_0^k\setminus\mathcal{I}^k_L}2^{-c|\boldsymbol{l}|_1}\eqqcolon\epsilon^{(k)}_{\boldsymbol{\nu}_{\mathfrak{u}},\boldsymbol{\alpha}_{\mathfrak{u}}}(L).\nonumber
    \end{align}
    Here we have used \(\nu_j-\alpha_j=c\) for all \(1\leq j \leq d\), and hence \(\nu_{\mathfrak{u}_i}-\alpha_{\mathfrak{u}_i}=c\) for all \(1\leq i\leq k\).
    Now, for general \(\mathbf{p}\in\mathbb{N}_0^d\), applying first Lemma \ref{lem: tensor hilbert} and then Lemma \ref{lem: supersets}, we have

    \begin{align}
    \begin{split}
    &\qquad\|I-P_{L,\mathbf{p},\Phi_{\boldsymbol{\nu},2^{\mathbf{p}}}}\|_{\mathcal{N}_{\Phi_{\boldsymbol{\nu},2^{\mathbf{p}}}}(\Gamma)\rightarrow\mathcal{N}_{\Phi_{\boldsymbol{\alpha},2^{\mathbf{p}}}}(\Gamma)}\\
    &\leq C^{(\boldsymbol{\nu},\boldsymbol{\alpha})}_{\ref{lem: tensor hilbert}}\sum_{k=1}^d\sum_{\boldsymbol{l}\in\mathcal{K}_k^d\setminus \mathcal{I}^d_L}\prod_{j=1}^d\mathbf{Q}_{\boldsymbol{\nu},\boldsymbol{\alpha}}(\boldsymbol{l},\mathbf{p})_j
    \end{split}\nonumber\\
    &\leq C^{(\boldsymbol{\nu},\boldsymbol{\alpha})}_{\ref{lem: tensor hilbert}}\sum_{k=1}^d\sum_{\mathfrak{u}\in\mathcal{P}^d_k}C_{\ref{lem: supersets}}^{(\boldsymbol{\nu}_{\mathfrak{u}},\boldsymbol{\alpha}_{\mathfrak{u}})}\sum_{\boldsymbol{l}\in\mathbb{N}_0^k\setminus\mathcal{I}_{\max\{0,L-|\mathbf{p}_{\mathfrak{u}}|-k\}}^k}2^{-(\boldsymbol{\nu}_{\mathfrak{u}}-\boldsymbol{\alpha}_{\mathfrak{u}})\cdot(\boldsymbol{l}+\mathbf{p}_{\mathfrak{u}}+\mathbf{1})},\nonumber\\
    &=C^{(\boldsymbol{\nu},\boldsymbol{\alpha})}_{\ref{lem: tensor hilbert}}\sum_{k=1}^d2^{-ck}\sum_{\mathfrak{u}\in\mathcal{P}^d_k}2^{-c|\mathbf{p}_{\mathfrak{u}}|_1}C_{\ref{lem: supersets}}^{(\boldsymbol{\nu}_{\mathfrak{u}},\boldsymbol{\alpha}_{\mathfrak{u}})}\sum_{\boldsymbol{l}\in\mathbb{N}_0^k\setminus\mathcal{I}_{\max\{0,L-|\mathbf{p}_{\mathfrak{u}}|-k\}}^k}2^{-c|\boldsymbol{l}|_1}.\nonumber
    \end{align}
    Substituting in \(\epsilon^{(k)}_{\boldsymbol{\nu}_{\mathfrak{u}},\boldsymbol{\alpha}_{\mathfrak{u}}}(L-|\mathbf{p}_{\mathfrak{u}}|_1 -k)\) gives us the statement.
\end{proof}

\subsection{Counting abscissae}\label{subsec: counting}
In order to gauge the approximation error in terms of the number required evaluations of \(f\), we must find expressions for the size of a lengthscale-informed sparse grid for any given level \(L\in\mathbb{N}_0\) and penalty \(\mathbf{p}\in\mathbb{N}_0^d\),
\begin{align}
    N_{d,\mathbf{p}}(L) \coloneqq |\mathcal{X}^{\otimes}_{\mathbf{p},L}|.\nonumber
\end{align}
In general, counting abscissae in sparse grids can be challenging, and often only upper bounds are achievable. Fortunately, for our choice of one-dimensional point sets, \(\mathcal{X}_l\subset(-1/2,1/2)\), closed-form expressions for the number of points have been derived for isotropic sparse grids (see Seciton 4, \cite{burkhardt2014counting}), which we are able to extend to our lengthscale-informed construction. The proof of the following theorem uses similar arguments as in the proof of Lemma \ref{lem: supersets}.

\begin{proof}[Proof of Theorem \ref{thm: counting abscissae}]
    The number of point evaluations in a lengthscale-adapted sparse grid is given by
    \begin{align}
        N_{d,\mathbf{p}}(L) =\sum_{\boldsymbol{l}\in\mathcal{I}^d_{L}}\prod_{j=1}^d\mathbf{B}(\boldsymbol{l},\mathbf{p})_j,\nonumber
    \end{align}
    where we define \(\mathbf{B}(\boldsymbol{l},\mathbf{p})_j\coloneqq|\mathcal{X}_{l_j}^{p_j}|-|\mathcal{X}_{l_j-1}^{p_j}|\) and \(\mathcal{I}_L^k\) is as defined in Definition \ref{def: penalised sparse grid operator}, (see \cite{Plumlee2014}). From the definition of \(\mathcal{X}_l^p\) in Definition \ref{def:chi_l^p}, we have
    \begin{align}
        |\mathcal{X}_{l_j}^{p_j}|=\begin{cases}
            |\mathcal{X}_{l_j-p_j}|=2^{l_j-p_j+1}-1&\textrm{if }l_j\geq p_j,\\
            1&\textrm{if }0\leq l_j< p_j, \textrm{ and}\\
            0&\textrm{otherwise},
        \end{cases}\nonumber
    \end{align}
    and hence,
    \begin{align}
        \mathbf{B}(\boldsymbol{l},\mathbf{p})_j=\begin{cases}
            2^{l_j-p_j}&\textrm{if }l_j> p_j,\\
            0&\textrm{if }0<l_j< p_j, \textrm{ and}\\
            1&\textrm{if }l_j=0.
        \end{cases}\nonumber
    \end{align}
    By Corollary \ref{cor: isomporphic}, \(\boldsymbol{\rho}:\mathcal{K}_k^d\cap\mathcal{I}_L^d\rightarrow\{\mathbf{0}\}\cup\mathcal{P}^d_k\times\mathbb{N}^k\cap\mathcal{I}^k_{L}\) is bijective, and since, by definition, \(\bigcup_{k=0}^d\mathcal{K}_k^d =\mathbb{N}_0^d\), and \(\mathcal{K}_{k_1}^d\cap\mathcal{K}_{k_2}^d=\emptyset\) if and only if \(k_1\neq k_2,\)
    we have
    \begin{align}
        \sum_{\boldsymbol{l}\in\mathcal{I}^d_{L}}\prod_{j=1}^d\mathbf{B}(\boldsymbol{l},\mathbf{p})_j &= \sum_{k=0}^d\sum_{\boldsymbol{l}\in\mathcal{K}_k^d\cap\mathcal{I}^d_{L}}\prod_{j=1}^d\mathbf{B}(\boldsymbol{l},\mathbf{p})_j,\nonumber\\
        &= 1+ \sum_{k=1}^d\sum_{\mathfrak{u}\in\mathcal{P}^d_k}\sum_{\mathbf{a}\in\mathbb{N}^k\cap\mathcal{I}^k_L}\prod_{j=1}^d\mathbf{B}(\boldsymbol{\rho}^{-1}(\mathfrak{u},\mathbf{a}),\mathbf{p})_j,\nonumber\\
        &= 1 +\sum_{k=1}^d\sum_{\mathfrak{u}\in\mathcal{P}^d_k}\sum_{\mathbf{a}\in\mathbb{N}^k\cap\mathcal{I}^k_L}\prod_{i=1}^k\mathbf{B}(\mathbf{a},\mathbf{p}_\mathfrak{u})_i,\nonumber
    \end{align}
    since \(\boldsymbol{\rho}(\mathfrak{u},\mathbf{a})_j=1\) for all \(j\notin \mathfrak{u}\). We note that \(\mathbf{B}(\mathbf{a},\mathbf{p}_{\mathfrak{u}})_i=0\) for all \(0<a_i<p_{\mathfrak{u}_i}\), and hence we can further restrict the multi-index set, giving
    \begin{align}
        N_{d,\mathbf{p}}(L)= 1 + \sum_{k=1}^d\sum_{\mathfrak{u}\in\mathcal{P}^d_k}\sum_{\substack{\mathbf{a}\in\mathcal{I}^k_L\\a_i>0\\a_i>p_{\mathfrak{u}_i}}}2^{|\mathbf{a}|_1-|\mathbf{p}_{\mathfrak{u}}|_1}.\nonumber
    \end{align}
    Finally, changing the multi-indices by \(l_i=a_i-p_{\mathfrak{u}_i}-1\) gives
    \begin{align}
        N_{d,\mathbf{p}}(L) &= 1 + \sum_{k=1}^d\sum_{\mathfrak{u}\in\mathcal{P}^d_k}\sum_{\boldsymbol{l}\in\mathcal{I}^k_{L-|\mathbf{p}_{\mathfrak{u}}|-k}}2^{|\boldsymbol{l}|_1+k},\nonumber\\
        &= 1 + \sum_{k=1}^d2^k\sum_{\mathfrak{u}\in\mathcal{P}^d_k}\sum_{\boldsymbol{l}\in\mathcal{I}^k_{L-|\mathbf{p}_{\mathfrak{u}}|_1-k}}2^{|\boldsymbol{l}|_1},\nonumber
    \end{align}
    where the final sum is exactly the number of points in an isotropic sparse grid, \(N_{k,\mathbf{0}}(L-|\mathbf{p}_{\mathfrak{u}}|_1-k)\), with the closed form expression given in Section 4, \cite{burkhardt2014counting}. Finally we note that when \(L-|\mathbf{p}_{\mathfrak{u}}|_1-k<0\), the set \(\mathcal{I}^k_{L-|\mathbf{p}_{\mathfrak{u}}|_1-k}\) is empty and hence the contribution is zero. 
\end{proof}
When considering growing penalties as discussed in Section \ref{subsec: lengthscale anisotropy}, the following can be shown. 
\begin{corollary}\label{cor: dimension independent points}
   Let \(\{p_j\}_{j\in\mathbb{N}}\) be a non-decreasing sequence of integers, \(p_j\in\mathbb{N}_0\), such that, for all \(j\in\mathbb{N}\), there exists an \(m\in\mathbb{N}\) such that \(p_{j+m}>p_j\). Then, for fixed \(L\in\mathbb{N}_0\), \(N_{d,\mathbf{p}}(L)\) is bounded independently of the dimension, \(d\).
\end{corollary}
\begin{proof}
Let \(k\in\{1,2,\dots,d\}\) and let \(r(k,L,\mathbf{p})=\min\{j\in\mathbb{N}:p_j>L-k\}\). Let \(d>r(k,L,\mathbf{p})\) and \(\mathfrak{u}\in\mathcal{P}_k^d\setminus\mathcal{P}_k^{r(k,L,\mathbf{p})}\), then \(|\mathbf{p_{\mathfrak{u}}}|\geq p_{r(k,L,\mathbf{p})}\), since there must exist at least one \(j\in\mathfrak{u}\) such that \(r(k,L,\mathbf{p})<j\leq d\). Thus \(L-|\mathbf{p}_{\mathfrak{u}}|-k\leq L-p_{{r(k,L,\mathbf{p})}}-k<0\) and hence \(N_{k,\mathbf{0}}(L-|\mathbf{p}_{\mathfrak{u}}|-k)=0\). Therefore, for large enough \(d\), we have
\begin{align}
    \sum_{\mathfrak{u}\in\mathcal{P}^d_k}N_{k,\mathbf{0}}(L-|\mathbf{p}_{\mathfrak{u}}|-k)=\sum_{\mathfrak{u}\in\mathcal{P}^{r(k,L,\mathbf{p})}_k}N_{k,\mathbf{0}}(L-|\mathbf{p}_{\mathfrak{u}}|-k),\nonumber
\end{align}
for all \(1\leq k\leq d\) and hence,
\begin{align}
    N_{d,\mathbf{p}}(L)&=1+\sum_{k=1}^d2^k\sum_{\mathfrak{u}\in\mathcal{P}^{r(k,L,\mathbf{p})}_k}N_{k,\mathbf{0}}(L-|\mathbf{p}_{\mathfrak{u}}|-k).\label{eq: kmax}
\end{align}

by Theorem \ref{thm: counting abscissae}. For a given \(L\in\mathbb{N}_0\), is clear that there must exist some \(k_{\max}\in\mathbb{N}\) such that \(|\mathbf{p}_{\mathfrak{u}}|+k>L\) for all \(k>k_{\max}\) and all \(\mathfrak{u}\in\mathcal{P}^{r(k,L,\mathbf{p})}_k\). Accordingly, we complete the proof by replacing \(d\) with \(k_{\max}\) in \eqref{eq: kmax}.

\end{proof}

\section{Fast implementation}\label{sec:fast implementation}
In Algorithm 1, Plumlee 2014 \cite{Plumlee2014}, a fast inference algorithm was outlined to perform quick inversions of the induced covariance matrix when employing product kernels on sparse grid designs. In this section, we adapt this algorithm to general lengthscale-informed sparse grids, enhancing its efficiency for high dimensional problems, especially when the penalty grows with dimension. We start by introducing a key result where, for ease of presentation, we hereafter denote \(P_{L,\mathbf{p},\Phi_{\boldsymbol{\nu},2^{\mathbf{p}}}}\) by \(P_{d,\mathbf{p},L}\) and \(s_{\mathcal{X}_{l_j}^{p_j}, \phi_{\nu_j,2^{p_j}}}\) by \(s_{\mathcal{X}_{l_j}^{p_j}}\).

\begin{proposition}[]\label{prop: alt LISG def}
Define the multi-index set \(\mathcal{J}^d_L=\{\boldsymbol{l}\in\mathbb{N}_0^d\,:\,\max\{0,L-d+1\}\leq|\boldsymbol{l}|_1\leq L\}\). The lengthscale-informed sparse grid operator in Definition \ref{def: penalised sparse grid operator} can be rewritten
    \begin{align}
        P_{d,\mathbf{p},L} &= \sum_{\boldsymbol{l}\in \mathcal{J}_L^d}b_{L}(\boldsymbol{l})\bigotimes_{j=1}^ds_{\mathcal{X}^{p_j}_{l_j}},\nonumber
    \end{align}
    where
    \begin{align}
        b_{L}(\boldsymbol{l})=(-1)^{L-|\boldsymbol{l}|_1}\binom{d-1}{L-|\boldsymbol{l}|_1}.\nonumber
    \end{align}
\end{proposition}
\begin{proof}
    This is Theorem 2 in \cite{Plumlee2014}, an instance of the standard sparse grid combination technique, see \cite{pflaum1999,DELVOS198299}.
\end{proof}
As an immediate consequence of Proposition \ref{prop: alt LISG def}, the inverse covariance matrix can be written entirely in terms of Kronecker products of much smaller Toeplitz matrices, each corresponding to one-dimensional designs, 
\begin{align}
    \Phi_{\boldsymbol{\nu},2^{\mathbf{p}}}(\mathcal{X}_{\mathbf{p},L}^{\otimes},\mathcal{X}_{\mathbf{p},L}^{\otimes})^{-1}=\sum_{\boldsymbol{l}\in\mathcal{J}_L^d}b_L(\boldsymbol{l})A(\boldsymbol{l})\left[\bigotimes_{j=1}^d\phi_{\nu_j,2^{p_j}}(\mathcal{X}_{l_j}^{p_j},\mathcal{X}_{l_j}^{p_j})^{-1}\right]A(\boldsymbol{l})^T,\nonumber
\end{align}
where each \(A(\boldsymbol{l})\) is an \(N_{d,\mathbf{p}}(L) \times (|\mathcal{X}_{l_1}|\times\cdots\times|\mathcal{X}_{l_d}|)\) matrix with exactly one entry equal to \(1\) in each column and zeros elsewhere, determined by the indexing of the covariance matrix. Accordingly, to evaluate \(P_{d,\mathbf{p},L}\), we only need compute the inverses of these component matrices and the coefficients \(b_L(\boldsymbol{l})\). We can immediately implement this algorithm for lengthscale-informed sparse grids and, in general, we observe a large computational speed up; the the slowest single procedure now being the inversion of the largest Toeplitz matrix of order \(\mathcal{O}(L2^L)\). Many of these matrices are also identical, enabling even faster computations when storing individual components. However, this method requires summation over the set \(\mathcal{J}_L^d\), which grows extremely quickly for \(d\gg10\), eventually losing its computational advantage over a simple Cholesky factorisation. Fortunately, with the inclusion of a non-trivial penalty \(\mathbf{p}\), significant redundancy is introduced into these calculations. To exploit this, we express \(P_{d,\mathbf{p},L}\) instead as a sum over a reduced multi-index set, \(\mathcal{A}_{\mathbf{p},L}^d\), now depending on the penalty. We define this set via the following map.
\begin{definition}\label{def: q}
    We define the mapping between multi-indices, \(\mathbf{q}_{d,\mathbf{p}}:\mathcal{J}_L^d\rightarrow\mathbb{N}_0^d\), by \(\mathbf{q}_{d,\mathbf{p}}(\boldsymbol{l})=\mathbf{a}\), such that \(a_j=\max\{l_j-p_j,0\}\), for all \(1\leq j\leq d\). We denote the image of \(\mathbf{q}_{d,\mathbf{p}}\) by \(\mathcal{A}_{\mathbf{p},L}^d\). For a given \(\mathbf{a}\in\mathcal{A}_{\mathbf{p},L}^d\), we also define the sets \(\mathfrak{u}(\mathbf{a})\coloneqq\{j\in\{1,\dots,d\}\,:\,a_j=0\}\) and \(\mathfrak{v}(\mathbf{a})\coloneqq\{1,\dots,d\}\setminus\mathfrak{u}(\mathbf{a})\).
\end{definition}
\begin{lemma}\label{lem: disjoint preimages}
    For \(\mathbf{a}\neq\mathbf{a'}\in\mathcal{A}_{\mathbf{p},L}^d\), the pre-images of \(\mathbf{q}_{d,\mathbf{p}}\) are disjoint; \(\mathbf{q}_{d,\mathbf{p}}^{-1}(\mathbf{a})\cap\mathbf{q}^{-1}_{d,\mathbf{p}}(\mathbf{a}')=\emptyset\).
\end{lemma}
\begin{proof}
    Consider \(d=1\), i.e. \(\mathbf{p}=p\in\mathbb{N}_0\). Then, for \(a\in\mathcal{A}^1_p\), we have
    \begin{align}
        q^{-1}_{1,p}(a)=\begin{cases}
            a+p &\textrm{ if }a>0, and\\
            \{0,\dots,p\} &\textrm{ if }a=0.
        \end{cases}\nonumber
    \end{align}
    For general dimension \(d>1\), the pre-image is then simply the product
    \begin{align}
        \mathbf{q}_{d,\mathbf{p}}^{-1}(\mathbf{a})=\bigtimes_{j=1}^d\mathbf{q}_{1,p_j}^{-1}(a_j).\nonumber
    \end{align}
    Let \(\mathbf{a}\neq\mathbf{a}'\in\mathcal{A}_{\mathbf{a},L}^d\), then there exists at least one index \(j\), \(1\leq j\leq d\), such that \(a_j\neq a_j'\). Firstly, assume \(a_j,a_j'\neq 0\). Then, in both cases, the pre-image is a single multi-index, and
    \(\mathbf{q}_{d,\mathbf{p}}^{-1}(\mathbf{a})_j = a_j +p_j\neq a_j' +p_j = \mathbf{q}_{d,\mathbf{p}}^{-1}(\mathbf{a}')_j\). Next we assume \(a_j=0\), and hence \(a_j'>0\). Then, for all \(\boldsymbol{l}\in \mathbf{q}_{d,\mathbf{p}}^{-1}(\mathbf{a})\), we have \(l_j\in\{0,\dots,p_j\}\). Therefore, \(l_j\neq a_j'+p_j=\mathbf{q}_{d,\mathbf{p}}^{-1}(\mathbf{a}')_j\). By symmetry of the last argument, all cases are considered, an hence \(\mathbf{q}_{d,\mathbf{p}}^{-1}(\mathbf{a})\cap \mathbf{q}_{d,\mathbf{p}}^{-1}(\mathbf{a}')=\emptyset\), as required.
\end{proof}
\begin{proposition}\label{prop: alt alt LISG def}
    We have that
    \begin{align}
        P_{d,\mathbf{p},L}=\sum_{\mathbf{a}\in\mathcal{A}_{\mathbf{p},L}^d}u_{\mathbf{p},L}(\mathbf{a})\bigotimes_{j=1}^ds_{\mathcal{X}_{l_j}},\nonumber
    \end{align}
    where
    \begin{align}
        u_{\mathbf{p},L(\mathbf{a})}=\sum_{m=0}^{|\mathbf{p}_{\mathfrak{u}(\mathbf{a})}|}\mathscr{S}_{\mathbf{a},\mathbf{p}}(m)(-1)^{L-m-|(\mathbf{a}+\mathbf{p})_{\mathfrak{v}(\mathbf{a})}|_1}\binom{d-1}{L-m-|(\mathbf{a}+\mathbf{p})_{\mathfrak{v}(\mathbf{a})}|_1},\nonumber
    \end{align}
    and
    \begin{align}
    \mathscr{S}_{\mathbf{a},\mathbf{p}}(m)=\sum_{\mathfrak{w}\subset\mathfrak{u}(\mathbf{a})}(-1)^{|\mathfrak{w}|}\binom{m+|\mathfrak{u}(\mathbf{a})|-|\mathfrak{w}|-|\mathbf{p}_{\mathfrak{w}}|_1-1}{|\mathfrak{u(\mathbf{a})}|-1}.\nonumber
    \end{align}
\end{proposition}
\begin{proof}
    For all \(\boldsymbol{l}\in\mathcal{J}_L^d\), we have
    \(s_{\mathcal{X}_{l_1}^{p_1}}\otimes\cdots\otimes s_{\mathcal{X}_{l_d}^{p_d}} = s_{\mathcal{X}_{\mathbf{q}_{d,\mathbf{p}}(\boldsymbol{l})_1}}\otimes\cdots\otimes s_{\mathcal{X}_{\mathbf{q}_{d,\mathbf{p}}(\boldsymbol{l})_d}},
    \)
    since,
    \begin{align}
        \mathcal{X}_{\mathbf{q}_{d,\mathbf{p}}(\boldsymbol{l})_j}&=\begin{cases}
            \mathcal{X}_{l_j-p_j}&\textrm{ for }l_j\geq p_j+1,\textrm{ and}\\
            \mathcal{X}_0&\textrm{ for }0\leq l_j\leq p_j,
            \end{cases}\nonumber
    \end{align}
    for all \(1\leq j\leq d\), which is exactly the definition of \(\mathcal{X}_{l_j}^{p_j}\) in Definition \ref{def:chi_l^p}. 
    By Lemma \ref{lem: disjoint preimages}, we can rearrange the expression for \(P_{d,\mathbf{p},L}\) in Proposition \ref{prop: alt LISG def} into a sum over multi-indices in \(\mathcal{A}_{\mathbf{p},L}^d\) as,
    \begin{align}
        P_{d,\mathbf{p},L}=\sum_{\mathbf{a}\in\mathcal{A}_{\mathbf{p},L}^d}\left(\sum_{\boldsymbol{l}\in \mathbf{q}_{d,\mathbf{p}}^{-1}(\mathbf{a})}b_{L}(\boldsymbol{l})\right)\bigotimes_{j=1}^d s_{\mathcal{X}_{a_j}}.\nonumber
    \end{align}
    We now wish to simplify the sum over the pre-image \(\mathbf{q}_{d,\mathbf{p}}^{-1}(\mathbf{a})\) for a given \(\mathbf{a}\in\mathcal{A}_{\mathbf{p},L}^d\) by considering the zero and non-zero components separately. Assuming \(|\mathfrak{u}(\mathbf{a})|=k\), we have
    \begin{align}
        &\sum_{\boldsymbol{l}\in \mathbf{q}_{d,\mathbf{p}}^{-1}(\mathbf{a})}b_{L}(\boldsymbol{l})=\sum_{l_1\in \mathbf{q}_{d,p_1}^{-1}(a_1)}\cdots\sum_{l_d\in \mathbf{q}_{d,p_d}^{-1}(a_d)}(-1)^{L-|\boldsymbol{l}|_1}\binom{d-1}{L-|\boldsymbol{l}|_1},\nonumber\\
        \begin{split}
        &=\sum_{l_{\mathfrak{u}(\mathbf{a})_1}\in \mathbf{q}_{d,p_{{\mathfrak{u}(\mathbf{a})}_1}}^{-1}(a_{\mathfrak{u}(\mathbf{a})_1})}\cdots\sum_{l_{\mathfrak{u}(\mathbf{a})_k}\in \mathbf{q}_{d,p_{{\mathfrak{u}(\mathbf{a})}_k}}^{-1}(a_{\mathfrak{u}(\mathbf{a})_k})}\\
        &\hspace{2.2cm}(-1)^{L-|\boldsymbol{l}_{\mathfrak{u}(\mathbf{a})}|_1-|(\mathbf{a}+\mathbf{p})_{\mathfrak{v}(\mathbf{a})}|_1}\binom{d-1}{L-|\boldsymbol{l}_{\mathfrak{u}(\mathbf{a})}|_1-|(\mathbf{a}+\mathbf{p})_{\mathfrak{v}(\mathbf{a})}|_1}.
        \end{split}\label{eq: sums over preimages}
    \end{align}
    For all \(1\leq i \leq k\), we have \(\mathbf{q}_{d,p_{{\mathfrak{u}(\mathbf{a})}_i}}^{-1}(a_{\mathfrak{u}(\mathbf{a})_i})=\{0,\dots,p_{{\mathfrak{u}(\mathbf{a})}_i}\}\). Therefore, we can rewrite \eqref{eq: sums over preimages} as
    \begin{align}
        \sum_{\boldsymbol{l}\in \mathbf{q}_{d,\mathbf{p}}^{-1}(\mathbf{a})}b_{L}(\boldsymbol{l})=\sum_{m=0}^{|\mathbf{p}_{\mathfrak{u}(\mathbf{a})}|}\mathscr{S}_{\mathbf{a},\mathbf{p}}(m)(-1)^{L-m-|(\mathbf{a}+\mathbf{p})_{\mathfrak{v}(\mathbf{a})}|_1}\binom{d-1}{L-m-|(\mathbf{a}+\mathbf{p})_{\mathfrak{v}(\mathbf{a})}|_1},\nonumber
    \end{align}
    where \(\mathscr{S}_{\mathbf{a},\mathbf{p}}(m)\coloneqq\#\{\boldsymbol{l}\in\mathbb{N}_0^k\,:\,|\boldsymbol{l}|_1=m,l_i\leq p_{\mathfrak{u}(\mathbf{a})_i}\}\), with an explicit expression given by inclusion-exclusion (see Chapter 2, \cite{StanleyCombinatorics}).
\end{proof}
Given \(L\in\mathbb{N}_0\), \(\mathbf{p}\in\mathbb{N}_0^d\) and a target function \(f:\Gamma^d\rightarrow\mathbb{R}\), we outline the procedural construction of the weight vector \(\mathbf{w}=\Phi_{\boldsymbol{\nu},2^{\mathbf{p}}}(\mathcal{X}_{\mathbf{p},L}^{\otimes},\mathcal{X}_{\mathbf{p},L}^{\otimes})^{-1}f(\mathcal{X}_{\mathbf{p},L}^{\otimes})\) in Algorithm \ref{alg: 1}, adapted for lengthscale-informed sparse grids from Algorithm 1, \cite{Plumlee2014}. Here, for a given \(\boldsymbol{a}\in\mathcal{A}_{\mathbf{p},L}^d\), we define \(u_{\mathbf{p},L}(\mathbf{a})\) as in Proposition \ref{prop: alt alt LISG def}, and let \(\mathbf{w}_{\mathbf{a}}\) denote the sub-vector of \(\mathbf{w}\) corresponding to the design points in the component grid \(\mathcal{X}_{a_1}\times\cdots\times\mathcal{X}_{a_d}\subset\mathcal{X}^{\otimes}_{\mathbf{p},L}\).

\begin{algorithm}[h]
\caption{Fast inference algorithm for lengthscale-informed sparse grids.}
    \begin{algorithmic}
    \State Initialise $\mathbf{w}=\mathbf{0}\in\mathbb{R}^{N_{d,\mathbf{p}}(L)}$
    \For{$\mathbf{a}\in\mathcal{A}_{\mathbf{p},L}^d$}
        \State $\mathbf{w}_{\mathbf{a}}=\mathbf{w}_{\mathbf{a}} + u_{\mathbf{p},L}(\mathbf{a})\left[\bigotimes_{j=1}^d\phi_{\nu_j,2^{p_j}}(\mathcal{X}_{a_j},\mathcal{X}_{a_j})^{-1}\right]f(\mathcal{X}_{a_1}\times\cdots\times\mathcal{X}_{a_d})$ 
    \EndFor
    \end{algorithmic}
\label{alg: 1}
\end{algorithm}
The explicit computational cost of constructing \(\mathcal{A}_{\mathbf{p},L}^d\), and consequently the algorithm, is difficult to assess as it is heavily dependent on the penalty; the greater the anisotropy, the smaller the set \(\mathcal{A}_{\mathbf{p},L}^d\). In the worst case, for a trivial penalty \(\mathbf{p}=(0,\dots,0)\), we recover isotropic sparse grids, and therefore the set \(|\mathcal{A}_{\mathbf{p},L}^d|=|\mathcal{J}_L^d|\leq\binom{L+d-1}{d-1}\). In the case of growing penalties as discussed in Section \ref{subsec: lengthscale anisotropy}, we have the following result.

\begin{proposition}
    Let \(\{p_j\}_{j\in\mathbb{N}}\) be a non-decreasing sequence of integers, \(p_j\in\mathbb{N}_0\), such that, for all \(j\in\mathbb{N}\), there exists an \(m\in\mathbb{N}\) such that \(p_{j+m}>p_j\). Then, for fixed \(L\in\mathbb{N}_0\), \(|\mathcal{A}_{\mathbf{p},L}^d|\) is bounded independently of \(d\).
\end{proposition}

\begin{proof}

    Let \(L\in\mathbb{N}_0\) be given and let \(\boldsymbol{l}\in\mathcal{J}_L^d\). Clearly for all \(d\in\mathbb{N}\) we have \(|\mathcal{A}_{\mathbf{p},L}^d|\leq|\mathcal{A}_{\mathbf{p},L}^{d+1}|\). By definition of \(\{p_j\}_{j\in\mathbb{N}}\), for \(d\) large enough, there must exist a \(k_{\max}<d\) such that \(l_k<p_k\) for all \(k>k_{\max}\), and therefore \(\mathbf{q}_{k,\mathbf{p}}(\boldsymbol{l})_k=0\) for all \(k>k_{\max}\).  It suffices to show that, for any \(k_2\geq k_1\geq \max\{k_{\max},L+1\}\), \(|\mathcal{A}_{\mathbf{p},L}^{k_1}|=|\mathcal{A}_{\mathbf{p},L}^{k_2}|\). We achieve this by defining a mapping between these sets, \(t:\mathcal{A}_{\mathbf{p},L}^{k_1}\rightarrow \mathcal{A}_{\mathbf{p},L}^{k_2}\), \(t(\mathbf{a}^{(1)}) = (a^{(1)}_1,a^{(1)}_2,\dots,a^{(1)}_{k_1},0,\dots,0)\in\mathbb{N}_0^{k_2}\), and show that it is bijective. First we show that the image is indeed contained in \(\mathcal{A}_{\mathbf{p},L}^{k_2}\). Let \(\mathbf{a}^{(1)}\in\mathcal{A}_{\mathbf{p},L}^{k_1}\), then the pre-image \(\mathbf{q}_{k_1,L}^{-1}(\mathbf{a}^{(1)})\) is non-empty in \(\mathcal{J}_L^{k_1}\) by Definition \ref{def: q}. Define an analogous map \(t^*:\mathcal{J}_L^{k_1}\rightarrow\mathcal{J}_{L}^{k_2}\) by \(t^*(\boldsymbol{l}^{(1)})=(l^{(1)}_1,l^{(1)}_2,\dots,l^{(1)}_{k_1},0,\dots,0)\in\mathbb{N}_0^{k_2}\). We have both \(|t^*(\boldsymbol{l}^{(1)})|_1=|\boldsymbol{l}^{(1)}|_1\leq L\) and \(|t^*(\boldsymbol{l}^{(1)})|_1\geq\max\{0,L-k_2+1\}=0\), and so \(t^*(\boldsymbol{l}^{(1)})\in\mathcal{J}_{L}^{k_2}\) for all \(\boldsymbol{l}^{(1)}\in\mathbf{q}_{k_1,L}^{-1}(\mathbf{a}^{(1)})\). We then have 
    \begin{align}    
        &\mathbf{q}_{k_2,\mathbf{p}}(t^*(\boldsymbol{l}^{(1)}))=(\max\{t^*(\boldsymbol{l}^{(1)})_1-p_1,0\},\dots,\max\{t^*(\boldsymbol{l}^{(1)})_{k_2}-p_{k_2},0\}),\nonumber\\
        &=(\max\{t^*(\boldsymbol{l}^{(1)})_1-p_1,0\},\dots,\max\{t^*(\boldsymbol{l}^{(1)})_{k_1}-p_{k_1},0\},0,\dots,0),\nonumber\\
        &=(\max\{\boldsymbol{l}^{(1)}_1-p_1,0\},\dots,\max\{\boldsymbol{l}^{(1)}_{k_1}-p_{k_1},0\},0,\dots,0),\nonumber\\
        &=(\mathbf{q}_{k_1,\mathbf{p}}(\boldsymbol{l}^{(1)}),0,\dots,0)=t(\mathbf{a}^{(1)}),\nonumber
    \end{align}
    and therefore \(t(\mathbf{a}^{(1)})\) is indeed contained in \(\mathcal{A}_{\mathbf{p},L}^{k_2}\). Clearly \(t\) is injective, what remains to show is surjectivity. Let \(\mathbf{a}^{(2)}\in\mathcal{A}_{\mathbf{p},L}^{k_2}\). Since \(k_1>k_{\max}\), we have established that \(\mathbf{a}^{(2)}=(a_1^{(2)},\cdots,a_{k_1}^{(2)},0,\dots,0)\in\mathbb{N}_0^{k_2}\). Furthermore, there must exist an \(\boldsymbol{l}^{(2)}\in\mathcal{J}_{L}^{k_2}\) such that
    \(\mathbf{q}_{k_2,\mathbf{p}}(\boldsymbol{l}^{(2)})=\mathbf{a}^{(2)}\). Therefore, \(\max\{l_j^{(2)}-p_j^{(2)},0\}=a_j^{(2)}\) for all \(1\leq j\leq k_2\), and in particular, \(\max\{l_k^{(2)}-p_k^{(2)},0\}=0\) for all \(k_1\leq k\leq k_2\). Consider the multi-index \(\boldsymbol{l}'^{(2)}\coloneqq(l_{1}^{(2)},\dots,l_{k_1}^{(2)},0,\dots,0)\in\mathbb{N}_0^{k_2}\). Clearly \(|\boldsymbol{l}'^{(2)}|_1\leq|\boldsymbol{l}^{(2)}|_1\leq L\), and since \(k_2>L+1\), we have \(|\boldsymbol{l}'^{(2)}|_1>\max\{0,L-k_2-1\}=0\). Therefore \(\boldsymbol{l}'^{(2)}\in\mathcal{J}_L^{k_2}\) and hence \(\boldsymbol{l}'^{(2)}\in\mathbf{q}_{k_2,\mathbf{p}}^{-1}(\mathbf{a}^{(2)})\). Now define \(\boldsymbol{l}'^{(1)}\coloneqq(l'^{(2)}_1,\dots,l_{k_1}^{(2)})\), then \(|\boldsymbol{l}'^{(1)}|_1=|\boldsymbol{l}'^{(2)}|_1\), and since \(k_1>L+1\), we have \(\boldsymbol{l}'^{(1)}\in\mathcal{J}_{L}^{k_1}\). Then, by definition,
    \begin{align}
        \mathcal{A}_{\mathbf{p},L}^{k_1}\ni\mathbf{q}_{k_1,\mathbf{p}}(\boldsymbol{l}'^{(1)})&=(\max\{l'^{(1)}_1-p_1,0\},\dots,\max\{l'^{(1)}_{k_1}-p_{k_1},0\}),\nonumber\\
        &=(\max\{l^{(1)}_1-p_1,0\},\dots,\max\{l^{(1)}_{k_1}-p_{k_1},0\}),\nonumber\\
        &=(a_1^{(2)},\dots,a_1^{(2)}),\nonumber
    \end{align}
    and \(t((a_1^{(2)},\dots,a_1^{(2)}))=\mathbf{a}^{(2)}\), as required.
\end{proof}

\section{Numerical experiments}\label{sec: numerics}
In this section, we present numerical results for approximating known target functions, \(f:\Gamma^d\rightarrow\mathbb{R}\), \(\Gamma=(-1/2,1/2)\). We compare the performance of separable Mat\'ern kernel interpolants across three different point sets: lengthscale-informed sparse grids (LISG), standard isotropic sparse grids (SG), and Monte Carlo points (MC) (uniformly distributed points \(\mathbf{x}_i\sim U(\Gamma^d)\)). For our target functions we use random linear combinations of separable Mat\'ern kernels with unit standard deviation, \(\boldsymbol{\sigma}=\mathbf{1}\in\mathbb{R}^d\),
\begin{align}\label{eq: target function}
    f(\mathbf{x})=\sum_{i=1}^M\xi_i\Phi_{\boldsymbol{\nu},\boldsymbol{\lambda}}(\mathbf{y}_i,\mathbf{x})\nonumber
\end{align}
where the random parameters \(\xi_i\) and \(\mathbf{y}_i\) are i.i.d. according to \(\xi_i\sim\mathcal{N}(0,5)\) and \(\mathbf{y}_i\sim U\left(\Gamma^d\right)\), respectively. This choice of target functions is particularly useful for validating the theory as their native space norms (with respect to the same kernel) are known exactly (see Proposition \ref{prop: native space norm}) and remain bounded independently of the dimension \(d\). Consequently, the lengthscale-anisotropy condition given in \eqref{eq: anisotropic condition} is automatically satisfied when \(\Phi_{\boldsymbol{\nu},\boldsymbol{\lambda}}\) is used as the interpolating kernel.

In each experiment, see Figures \ref{fig: iso vs LISG}, \ref{fig: LISG time}, \ref{fig: LISG high dimensions}, \ref{fig: LISG misspecified} and \ref{fig: gp post var}, we consider two distinct lengthscale anisotropy behaviours; exponential and linear growth, corresponding to linearly, \(\mathbf{p}_{\textrm{lin}}\), and logarithmically, \(\mathbf{p}_{\textrm{log}}\), growing penalties, respectively. Specifically, we take \(p_{\textrm{lin},j}=j-1\) and \(p_{\textrm{log},j} = \lceil\log_2(j)\rceil\), where \(\lceil\cdot\rceil\) corresponds to the standard ceiling function, enforcing integer valued penalties. In Figures \ref{fig: iso vs LISG}-\ref{fig: LISG misspecified}, the relative \(L^2\)-errors presented are averaged over 10 runs on independent realisations of \(f\), each being approximated with 100 Monte Carlo samples. For large designs in high dimensions, for example \(d\geq25\) and \(L\geq8\), due to the efficiency of Algorithm~\ref{alg: 1}, the major bottleneck becomes sampling the emulator with each sample requiring \(d\times N\)-many kernel evaluations. Hence, to compute the $L^2$-error in  reasonable time, we have chosen the number of Monte Carlo samples to be as small as possible while ensuring sufficiently small Monte Carlo error. In practice, evaluating emulators is trivially parallelisable, and so we do not view this as a particular drawback of GP emulation. We use Algorithm~\ref{alg: 1} to evaluate the interpolants for both sparse grid methods, and for Monte Carlo points we use a standard Cholesky decomposition. The Python code used to generate the data for each plot can be found at \url{https://github.com/elliot-addy/LISG}.

In the first experiment, see Figure \ref{fig: iso vs LISG}, we compare the error of lengthscale-informed sparse grids with standard isotropic sparse grids and Monte Carlo points in approximating realisations of \(f\). We employ both lengthscale-informed kernels, where we matched the lengthscale hyperparameters in each axial direction with the lengthscale-anisotropy in \(f\), and isotropic kernels, where we set \(\boldsymbol{\lambda}=\mathbf{1}\). The results shown assume \(\nu_j=1.5\) in all directions, \(1\leq j\leq d\), for both \(f\) and the interpolating kernels. We note that similar relationships were observed for the values \(\nu_j=0.5\) and \(\nu_j=2.5\). 

We immediately see a clear separation between the methods employing isotropic kernels and those with lengthscales adapted to the anisotropy in \(f\), with the latter converging much faster. The difference in the errors between the SG and LISG methods is more pronounced in the higher dimensional case, \(d=8\), and for linearly growing penalties, \(\mathbf{p}=\mathbf{p}_{\textrm{lin}}\), as supported by our results and discussion in Section \ref{subsec: results}. For both isotropic and lengthscale-informed interpolating kernels, we see that the error consistently converges fastest with Monte Carlo designs. We believe this discrepancy may be partially explained by the extrapolation behaviour of the one-dimensional interpolants \(s_{\mathcal{X}_l,\phi_{\nu,\lambda}}\) near the endpoints of the intervals, and thus inherited by sparse grid designs in higher dimensions. Consequently, we have also presented the corresponding results for lengthscale-informed sparse grids based on point-sets weighted towards the endpoints; namely Clenshaw-Curtis abscissae (excluding endpoints), \(\mathcal{X}_{\textrm{CC},l}\coloneqq\{\sin(\theta)/2\,:\,-\pi<\theta<\pi, \theta=n\pi/2^{l} ,n\in\mathbb{Z}\}\), and uniformly-spaced abscissae (including endpoints),
\begin{align}
    \mathcal{X}_{\textrm{B},l}\coloneqq\begin{cases}
        \mathcal{X}_{l-1}\cup\{-1/2,1/2\}&\textrm{ if } l\geq 2,\\
        \{-1/2,0,1/2\}&\textrm{ if } l=1, \textrm{ and}\\
        \{0\}&\textrm{ if } l=0.
    \end{cases}\nonumber
\end{align}
In both cases, the error is improved over the uniformly-spaced grids defined in Section \ref{sec: results}. Otherwise, to the authors' knowledge, no such approximation bounds have been derived that are able to explain the convergence rate observed in the Monte Carlo case, however we point the reader to~\cite{adcock2023montecarlogoodsampling}, which indicates that the choice of point sets may be secondary to the the choice of basis functions, which mirrors our observations. We also note that it is difficult to compare these methods for \(d\geq10\), due to how quickly isotropic sparse grid designs grow, and for \(N>10^4\), due to the inefficiency of evaluating emulators based on unstructured designs, such as with Monte Carlo points. 

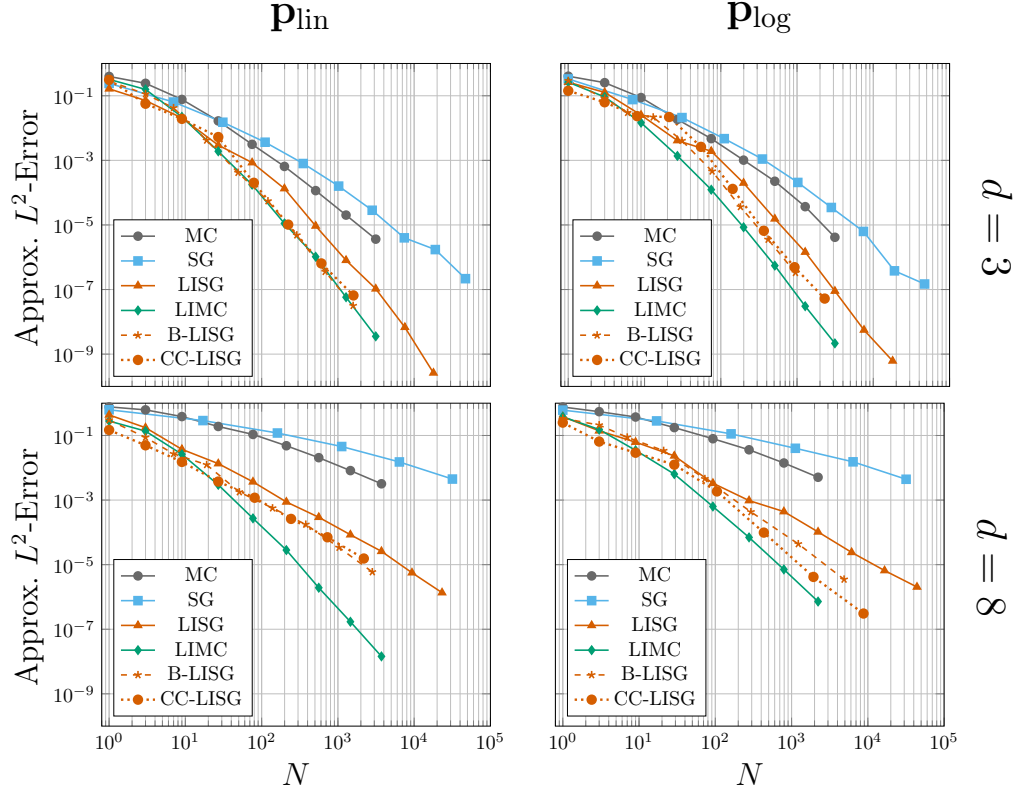
\begin{figure}[h]
    \begin{subfigure}{.45\textwidth}
    \begin{tikzpicture}[scale=0.75]
        \begin{axis}[
        cycle list name = myCycleList,
        xmode = log, ymode = log,
        xmin = 0.8, xmax=1e5,
        ymin = 1e+-10, ymax= 1,
        legend pos=south west,
        xticklabels={},
        grid=both
        ]
        \addplot table [x index=0,y index=1, col sep=space] {lin_mc_nu_1.5_dim_3.dat};
        \addplot table [x index=0,y index=1, col sep=space] {lin_sg_nu_1.5_dim_3.dat};
        \addplot table [x index=0,y index=1, col sep=space] {lin_psg_nu_1.5_dim_3.dat};
        \addplot table [x index=0,y index=1, col sep=space] {lin_stretched_mc_nu_1.5_dim_3.dat};
        \addplot[color=vermillion, thick, dashed, mark=star, mark options={draw=vermillion, solid}] table [x index=0,y index=1, col sep=space] {boundary_lin_psg_nu_1.5_dim_3.dat};
        \addplot[color=vermillion, very thick, dotted, mark=*, mark options={draw=vermillion, solid}] table [x index=0,y index=1, col sep=space] {cc_lin_psg_nu_1.5_dim_3.dat};
        \legend{MC, SG, LISG, LIMC, B-LISG, CC-LISG};
        \end{axis}
        \node at (3.5,6.5) {\Large$\mathbf{p}_{\textrm{lin}}$};
        \node[rotate=90] at (-1.3,2.8) {\large Approx. \(L^2\)-Error};
    \end{tikzpicture}
    \end{subfigure}
    \hspace{10mm}
    \begin{subfigure}{.45\textwidth}
    \begin{tikzpicture}[scale=0.75]
        \begin{axis}[
        cycle list name = myCycleList,
        xmode = log, ymode = log,
        xmin = 0.8, xmax=1e5,
        ymin = 1e+-10, ymax=1,
        legend pos=south west,
        xticklabels={}, yticklabels={},
        grid=both
        ]
        \addplot table [x index=0,y index=1, col sep=space] {log_mc_nu_1.5_dim_3.dat};
        \addplot table [x index=0,y index=1, col sep=space] {log_sg_nu_1.5_dim_3.dat};
        \addplot table [x index=0,y index=1, col sep=space] {log_psg_nu_1.5_dim_3.dat};
        \addplot table [x index=0,y index=1, col sep=space] {log_stretched_mc_nu_1.5_dim_3.dat};
        \addplot[color=vermillion, thick, dashed, mark=star, mark options={draw=vermillion, solid}] table [x index=0,y index=1, col sep=space] {boundary_log_psg_nu_1.5_dim_3.dat};
        \addplot[color=vermillion, very thick, dotted, mark=*, mark options={draw=vermillion, solid}] table [x index=0,y index=1, col sep=space] {cc_log_psg_nu_1.5_dim_3.dat};
        \legend{MC, SG, LISG, LIMC, B-LISG, CC-LISG};
        \end{axis}
        \node at (3.5,6.5) {\Large$\mathbf{p}_{\textrm{log}}$};
        \node[rotate=-90] at (7.7,2.8) {\Large$d=3$};
    \end{tikzpicture}
    \end{subfigure}
    \begin{subfigure}{.45\textwidth}
    \begin{tikzpicture}[scale=0.75]
        \begin{axis}[
        cycle list name = myCycleList,
        xmode = log, ymode = log,
        xmin = 0.8, xmax=1e5,
        ymin = 1e+-10, ymax= 1,
        legend pos=south west,
        xlabel={\Large$N$},
        grid=both
        ]
        \addplot table [x index=0,y index=1, col sep=space] {lin_mc_nu_1.5_dim_8.dat};
        \addplot table [x index=0,y index=1, col sep=space] {lin_sg_nu_1.5_dim_8.dat};
        \addplot table [x index=0,y index=1, col sep=space] {lin_psg_nu_1.5_dim_8.dat};
        \addplot table [x index=0,y index=1, col sep=space] {lin_stretched_mc_nu_1.5_dim_8.dat};
        \addplot[color=vermillion, thick, dashed, mark=star, mark options={draw=vermillion, solid}] table [x index=0,y index=1, col sep=space] {boundary_lin_psg_nu_1.5_dim_8.dat};
        \addplot[color=vermillion, very thick, dotted, mark=*, mark options={draw=vermillion, solid}] table [x index=0,y index=1, col sep=space] {cc_lin_psg_nu_1.5_dim_8.dat};
        \legend{MC, SG, LISG, LIMC, B-LISG, CC-LISG};
        \end{axis}
        \node[rotate=90] at (-1.3,2.8) {\large Approx. \(L^2\)-Error};
    \end{tikzpicture}
    \end{subfigure}
    \hspace{10mm}
    \begin{subfigure}{.45\textwidth}
    \begin{tikzpicture}[scale=0.75]
        \begin{axis}[
        cycle list name = myCycleList,
        xmode = log, ymode = log,
        xmin = 0.8, xmax=1e5,
        ymin = 1e+-10, ymax= 1,
        legend pos=south west,
        yticklabels={},
        xlabel={\Large$N$},
        grid=both
        ]
        \addplot table [x index=0,y index=1, col sep=space] {log_mc_nu_1.5_dim_8.dat};
        \addplot table [x index=0,y index=1, col sep=space] {log_sg_nu_1.5_dim_8.dat};
        \addplot table [x index=0,y index=1, col sep=space] {log_psg_nu_1.5_dim_8.dat};
        \addplot table [x index=0,y index=1, col sep=space] {log_stretched_mc_nu_1.5_dim_8.dat};
        \addplot[color=vermillion, thick, dashed, mark=star, mark options={draw=vermillion, solid}] table [x index=0,y index=1, col sep=space] {boundary_log_psg_nu_1.5_dim_8.dat};
        \addplot[color=vermillion, very thick, dotted, mark=*, mark options={draw=vermillion, solid}] table [x index=0,y index=1, col sep=space] {cc_log_psg_nu_1.5_dim_8.dat};
        \legend{MC, SG, LISG, LIMC, B-LISG, CC-LISG};
        \end{axis}
        \node[rotate=-90] at (7.7,2.8) {\Large$d=8$};
    \end{tikzpicture}
    \end{subfigure}
    \caption{Relative \(L^2\)-error in approximating realisations of \(f\) when (a), interpolating with isotropic kernels (\(\boldsymbol{\lambda}=\mathbf{1}\)) on Monte Carlo (MC) and standard sparse grid (SG) designs, and (b), interpolating with anisotropic kernels (\(\boldsymbol{\lambda}=2^\mathbf{p}\)) on Monte Carlo (LIMC) and lengthscale-informed sparse grid designs, based on uniformly spaced-points (LISG), uniformly-spaced including boundary points (B-LISG), and Clenshaw-Curtis points (CC-LISG). We consider input dimensions 3 and 8, both linearly and logarithmically growing penalties, and employ separable Mat\'ern kernels with \(\nu_j=1.5\) for all \(1\leq j \leq d\).}
    \label{fig: iso vs LISG}
\end{figure}

In Figure~\ref{fig: LISG time}, we similarly examine the approximation error for various design-kernel combinations, this time with respect to the the compute time of the interpolants.\footnote{All computations were performed on a desktop pc with an Intel i7-10710U (12 cores, 4.7 GHz), 31 GB RAM, running Ubuntu 24.04.3 LTS and Python 3.12.3.} Specifically, we record the wall-clock time required to solve the linear system
\begin{align}
f(\mathcal{X})=\Phi_{\boldsymbol{\nu},\boldsymbol{\lambda}}(\mathcal{X},\mathcal{X})\mathbf{w},\nonumber
\end{align}
where \(\mathcal{X}\) denotes the different point-sets considered. The experimental setup is identical to that of Figure~\ref{fig: iso vs LISG}, but applied to a reduced number of designs. As before, we use \(\nu_j=1.5\) in all directions for both \(f\) and the interpolating kernels. We note that our implementation of Algorithm~\ref{alg: 1} is not optimised, and therefore the compute times reported for the sparse grid methods could likely be improved by a constant factor.

Across all six studies, the LISG method achieves the best error-cost scaling of the considered methods with respect to the compute time. Comparing these results with Figure~\ref{fig: iso vs LISG} in dimensions \(d=3\) and \(d=8\) highlights the computational savings provided by the fast implementation described in Algorithm~\ref{alg: 1}. For linearly growing penalties in high dimensions (\(d=8,25\)), the observed computational cost of the LISG construction for a given accuracy closely matches that of Monte Carlo designs with anisotropic kernels. In contrast, in the more realistic case of logarithmically growing penalties, the LISG method exhibits a modest but consistent improvement in empirical error-cost scaling relative to Monte Carlo designs. In both penalty regimes, the computational complexity of the anisotropic methods appears to be largely unchanged between dimensions 8 and 25, suggesting an insensitivity to the input dimension beyond \(d\) large enough. This observation aligns with Corollary~\ref{cor: dimension independent points}, which states that the size of lengthscale-informed sparse grids---and thus associated covariance matrices---are bounded independently of the dimension for growing penalties. In all cases, isotropic sparse grids yield substantially poorer accuracy for the same computational effort, despite also benefitting from the fast implementation.

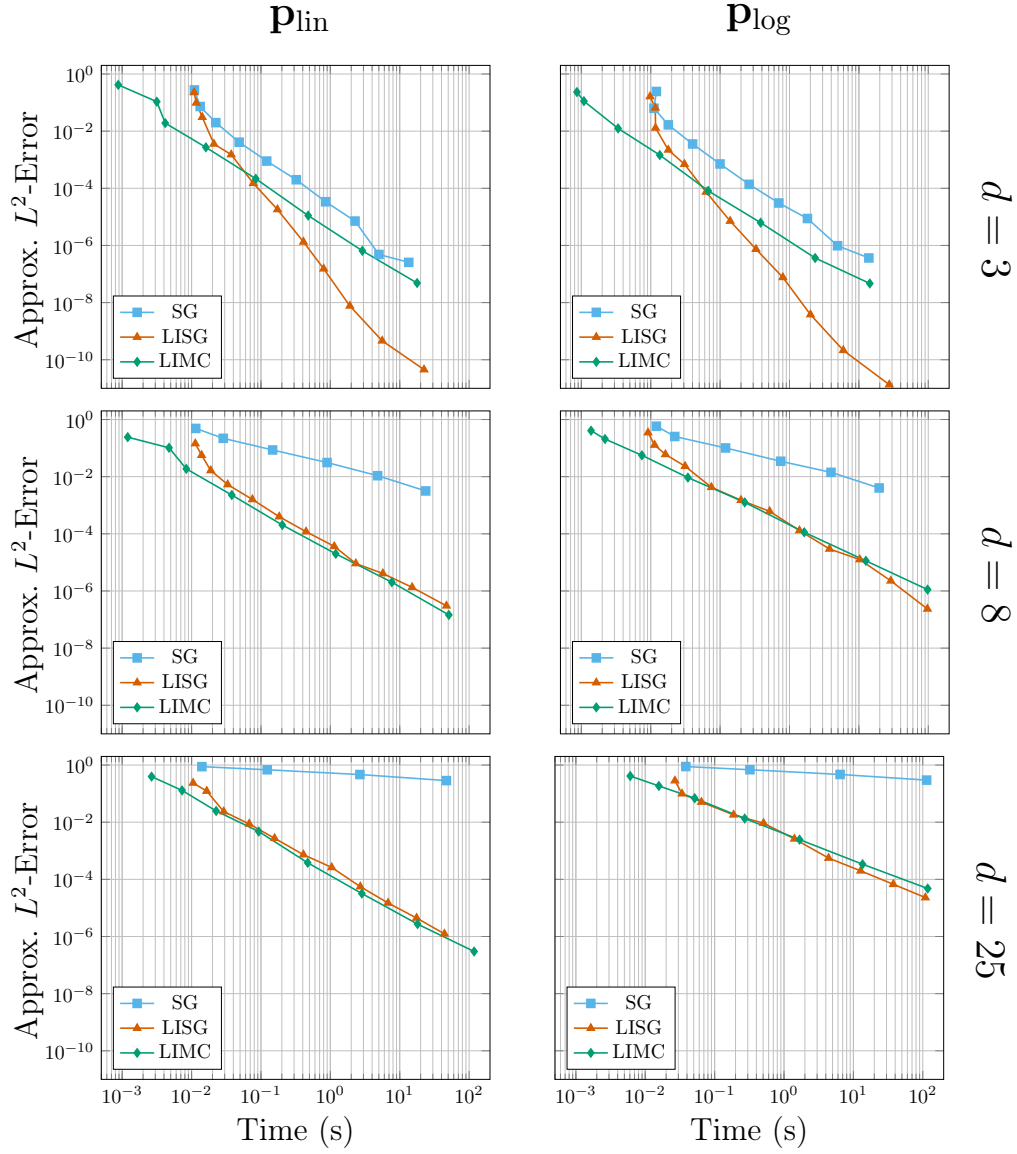
\begin{figure}[h]
    \begin{subfigure}{.45\textwidth}
    \begin{tikzpicture}[scale=0.75]
        \begin{axis}[
        cycle list name = myCycleList2,
        grid = both,
        xmode = log, ymode = log,
        xmin = 5e-4, xmax=2e2,
        ymin = 1e+-11, ymax= 2,
        legend pos=south west,
        xticklabels={}
        ]
        \addplot table [x index=0,y index=1, col sep=space] {time_sg_nu_1.5_dim_3_linp.dat};
        \addplot table [x index=0,y index=1, col sep=space] {time_lisg_nu_1.5_dim_3_linp.dat};
        \addplot table [x index=0,y index=1, col sep=space] {time_mc_nu_1.5_dim_3_linp.dat};
        \legend{SG,LISG,LIMC};
        \end{axis}
        \node at (3.5,6.5) {\Large$\mathbf{p}_{\textrm{lin}}$};
        \node[rotate=90] at (-1.3,2.8) {\large Approx. \(L^2\)-Error};
    \end{tikzpicture}
    \end{subfigure}
    \hspace{10mm}
    \begin{subfigure}{.45\textwidth}
    \begin{tikzpicture}[scale=0.75]
        \begin{axis}[
        cycle list name = myCycleList2,
        grid = both,
        xmode = log, ymode = log,
        xmin = 5e-4, xmax=2e2,
        ymin = 1e+-11, ymax=2,
        legend pos=south west,
        xticklabels={}, yticklabels={}
        ]
        \addplot table [x index=0,y index=1, col sep=space] {time_sg_nu_1.5_dim_3_logp.dat};
        \addplot table [x index=0,y index=1, col sep=space] {time_lisg_nu_1.5_dim_3_logp.dat};
        \addplot table [x index=0,y index=1, col sep=space] {time_mc_nu_1.5_dim_3_logp.dat};
        \legend{SG,LISG,LIMC};
        \end{axis}
        \node at (3.5,6.5) {\Large$\mathbf{p}_{\textrm{log}}$};
        \node[rotate=-90] at (7.7,2.8) {\Large$d=3$};
    \end{tikzpicture}
    \end{subfigure}
    \begin{subfigure}{.45\textwidth}
    \begin{tikzpicture}[scale=0.75]
        \begin{axis}[
        cycle list name = myCycleList2,
        grid = both,
        xmode = log, ymode = log,
        xmin = 5e-4, xmax=2e2,
        ymin = 1e+-11, ymax= 2,
        legend pos=south west,
        xticklabels = {}
        ]
        \addplot table [x index=0,y index=1, col sep=space] {time_sg_nu_1.5_dim_8_linp.dat};
        \addplot table [x index=0,y index=1, col sep=space] {time_lisg_nu_1.5_dim_8_linp.dat};
        \addplot table [x index=0,y index=1, col sep=space] {time_mc_nu_1.5_dim_8_linp.dat};
        \legend{SG,LISG,LIMC};
        \end{axis}
        \node[rotate=90] at (-1.3,2.8) {\large Approx. \(L^2\)-Error};
    \end{tikzpicture}
    \end{subfigure}
    \hspace{10mm}
    \begin{subfigure}{.45\textwidth}
    \begin{tikzpicture}[scale=0.75]
        \begin{axis}[
        cycle list name = myCycleList2,
        grid = both,
        xmode = log, ymode = log,
        xmin = 5e-4, xmax=2e2,
        ymin = 1e+-11, ymax= 2,
        legend pos=south west,
        yticklabels = {}, xticklabels = {}
        ]
        \addplot table [x index=0,y index=1, col sep=space] {time_sg_nu_1.5_dim_8_logp.dat};
        \addplot table [x index=0,y index=1, col sep=space] {time_lisg_nu_1.5_dim_8_logp.dat};
        \addplot table [x index=0,y index=1, col sep=space] {time_mc_nu_1.5_dim_8_logp.dat};
        \legend{SG,LISG,LIMC};
        \end{axis}
        \node[rotate=-90] at (7.7,2.8) {\Large$d=8$};
    \end{tikzpicture}
    \end{subfigure}
    \begin{subfigure}{.45\textwidth}
    \begin{tikzpicture}[scale=0.75]
        \begin{axis}[
        cycle list name = myCycleList2,
        grid = both,
        xmode = log, ymode = log,
        xmin = 5e-4, xmax=2e2,
        ymin = 1e+-11, ymax= 2,
        xlabel={\Large Time (s)},
        legend pos=south west
        ]
        \addplot table [x index=0,y index=1, col sep=space] {time_sg_nu_1.5_dim_25_linp.dat};
        \addplot table [x index=0,y index=1, col sep=space] {time_lisg_nu_1.5_dim_25_linp.dat};
        \addplot table [x index=0,y index=1, col sep=space] {time_mc_nu_1.5_dim_25_linp.dat};
        \legend{SG,LISG,LIMC};
        \end{axis}
        \node[rotate=90] at (-1.3,2.8) {\large Approx. \(L^2\)-Error};
    \end{tikzpicture}
    \end{subfigure}
    \hspace{10mm}
    \begin{subfigure}{.45\textwidth}
    \begin{tikzpicture}[scale=0.75]
        \begin{axis}[
        cycle list name = myCycleList2,
        grid = both,
        xmode = log, ymode = log,
        xmin = 5e-4, xmax=2e2,
        ymin = 1e+-11, ymax= 2,
        xlabel={\Large Time (s)},
        legend pos=south west,
        yticklabels = {}
        ]
        \addplot table [x index=0,y index=1, col sep=space] {time_sg_nu_1.5_dim_25_logp.dat};
        \addplot table [x index=0,y index=1, col sep=space] {time_lisg_nu_1.5_dim_25_logp.dat};
        \addplot table [x index=0,y index=1, col sep=space] {time_mc_nu_1.5_dim_25_logp.dat};
        \legend{SG,LISG,LIMC};
        \end{axis}
        \node[rotate=-90] at (7.7,2.8) {\Large$d=25$};
    \end{tikzpicture}
    \end{subfigure}
    \caption{Relative \(L^2\)-approximation error measured against the time taken to compute individual interpolants of \(f\). We employ the same setup as in Figure \ref{fig: iso vs LISG}, with the addition of experiments undertaken for input dimension 25.}
    \label{fig: LISG time}
\end{figure}

In Figure \ref{fig: LISG high dimensions}, we assess the effectiveness of Lengthscale-Informed Sparse Grids in significantly higher dimensions. Although theoretically applicable in arbitrarily high dimensions, for ease of implementation we consider up to a maximum dimension of \(d=100\), by which point clear trends are established. For the case \(\nu_j=2.5\), the number of points for which we can test is limited due to ill conditioning of the covariance matrices. 

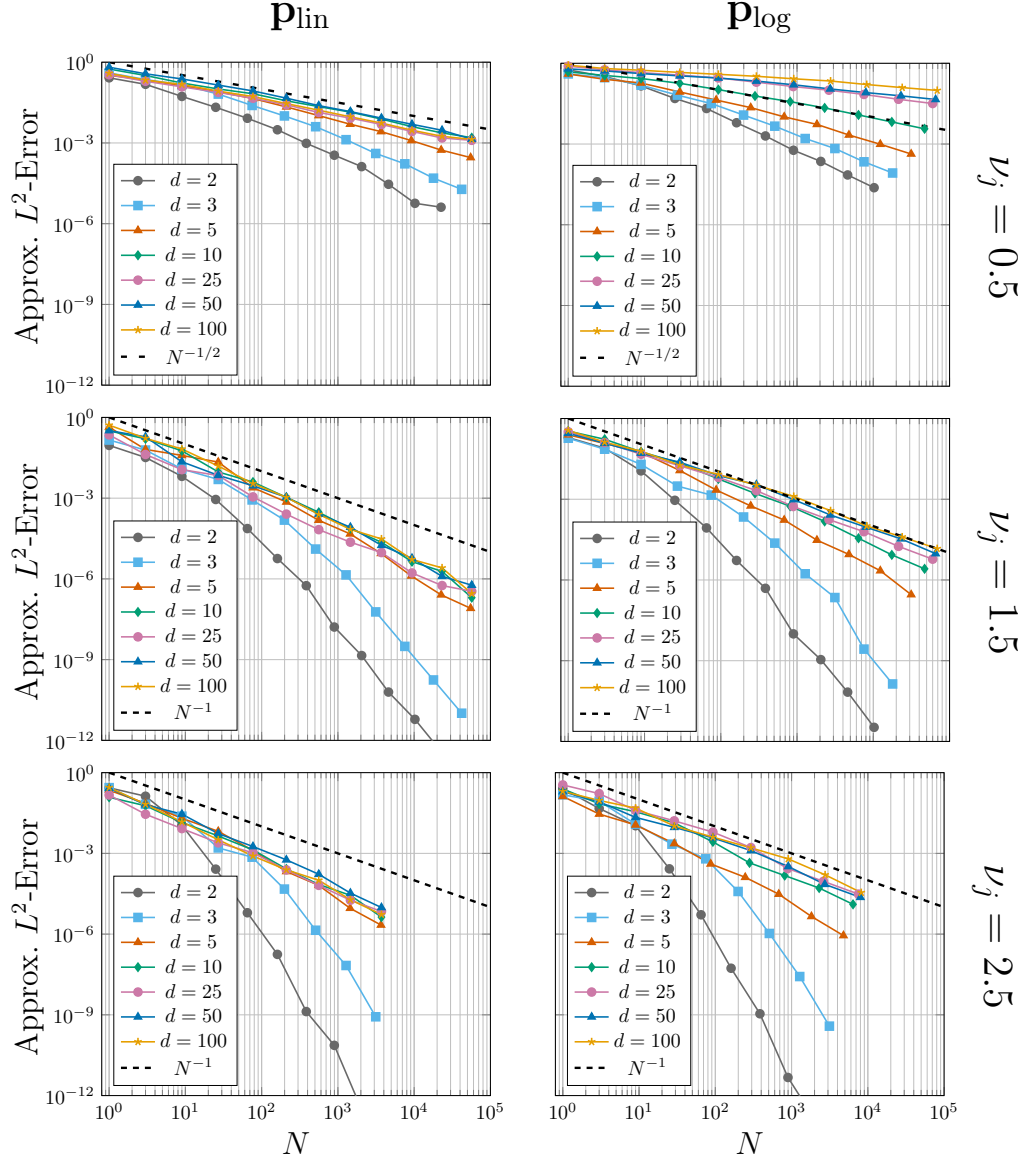
\begin{figure}[h]
    \begin{subfigure}{.45\textwidth}
    \begin{tikzpicture}[scale=0.75]
        \begin{axis}[
        cycle list name = myCycleList,
        grid = both,
        xmode = log, ymode = log,
        xmin = 0.8, xmax=1e5,
        ymin = 1e+-12, ymax= 1,
        legend pos=south west,
        xticklabels={}
        ]
        \addplot table [x index=0,y index=1, col sep=space] {nu_0.5_dim_2_pen_lin_avg_10.dat};
        \addplot table [x index=0,y index=1, col sep=space] {nu_0.5_dim_3_pen_lin_avg_10.dat};
        \addplot table [x index=0,y index=1, col sep=space] {nu_0.5_dim_5_pen_lin_avg_10.dat};
        \addplot table [x index=0,y index=1, col sep=space] {nu_0.5_dim_10_pen_lin_avg_10.dat};
        \addplot table [x index=0,y index=1, col sep=space] {nu_0.5_dim_25_pen_lin_avg_10.dat};
        \addplot table [x index=0,y index=1, col sep=space] {nu_0.5_dim_50_pen_lin_avg_10.dat};
        \addplot table [x index=0,y index=1, col sep=space] {nu_0.5_dim_100_pen_lin_avg_10.dat};
        \addplot[loosely dashed, very thick, black, domain=1:1e5, samples=100] {1/(sqrt(x))};
        \legend{\(d=2\),\(d=3\),\(d=5\),\(d=10\),\(d=25\),\(d=50\),\(d=100\),\(N^{-1/2}\)};
        \end{axis}
        \node at (3.5,6.5) {\Large$\mathbf{p}_{\textrm{lin}}$};
        \node[rotate=90] at (-1.3,2.8) {\large Approx. \(L^2\)-Error};
    \end{tikzpicture}
    \end{subfigure}
    \hspace{10mm}
    \begin{subfigure}{.45\textwidth}
    \begin{tikzpicture}[scale=0.75]
        \begin{axis}[
        cycle list name = myCycleList,
        grid = both,
        xmode = log, ymode = log,
        xmin = 0.8, xmax = 1e5,
        ymin = 1e+-12, ymax=1,
        legend pos=south west,
        xticklabels={}, yticklabels={}
        ]
        \addplot table [x index=0,y index=1, col sep=space] {nu_0.5_dim_2_pen_log_avg_10.dat};
        \addplot table [x index=0,y index=1, col sep=space] {nu_0.5_dim_3_pen_log_avg_10.dat};
        \addplot table [x index=0,y index=1, col sep=space] {nu_0.5_dim_5_pen_log_avg_10.dat};
        \addplot table [x index=0,y index=1, col sep=space] {nu_0.5_dim_10_pen_log_avg_10.dat};
        \addplot table [x index=0,y index=1, col sep=space] {nu_0.5_dim_25_pen_log_avg_10.dat};
        \addplot table [x index=0,y index=1, col sep=space] {nu_0.5_dim_50_pen_log_avg_10.dat};
        \addplot table [x index=0,y index=1, col sep=space] {nu_0.5_dim_100_pen_log_avg_10.dat};
        \addplot[loosely dashed, very thick, black, domain=1:1e5, samples=100] {1/(sqrt(x))};
        \legend{\(d=2\),\(d=3\),\(d=5\),\(d=10\),\(d=25\),\(d=50\),\(d=100\),\(N^{-1/2}\)};
        \end{axis}
        \node at (3.5,6.5) {\Large$\mathbf{p}_{\textrm{log}}$};
        \node[rotate=-90] at (7.7,2.8) {\Large$\nu_j=0.5$};
    \end{tikzpicture}
    \end{subfigure}
    \begin{subfigure}{.45\textwidth}
    \begin{tikzpicture}[scale=0.75]
        \begin{axis}[
        cycle list name = myCycleList,
        grid = both,
        xmode = log, ymode = log,
        xmin = 0.8, xmax=1e5,
        ymin = 1e+-12, ymax= 1,
        legend pos=south west,
        xticklabels = {}
        ]
        \addplot table [x index=0,y index=1, col sep=space] {nu_1.5_dim_2_pen_lin_avg_10.dat};
        \addplot table [x index=0,y index=1, col sep=space] {nu_1.5_dim_3_pen_lin_avg_10.dat};
        \addplot table [x index=0,y index=1, col sep=space] {nu_1.5_dim_5_pen_lin_avg_10.dat};
        \addplot table [x index=0,y index=1, col sep=space] {nu_1.5_dim_10_pen_lin_avg_10.dat};
        \addplot table [x index=0,y index=1, col sep=space] {nu_1.5_dim_25_pen_lin_avg_10.dat};
        \addplot table [x index=0,y index=1, col sep=space] {nu_1.5_dim_50_pen_lin_avg_10.dat};
        \addplot table [x index=0,y index=1, col sep=space] {nu_1.5_dim_100_pen_lin_avg_10.dat};
        \addplot[dashed, very thick, black, domain=1:1e5, samples=100] {1/x};
        \legend{\(d=2\),\(d=3\),\(d=5\),\(d=10\),\(d=25\),\(d=50\),\(d=100\),\(N^{-1}\)};
        \end{axis}
        \node[rotate=90] at (-1.3,2.8) {\large Approx. \(L^2\)-Error};
    \end{tikzpicture}
    \end{subfigure}
    \hspace{10mm}
    \begin{subfigure}{.45\textwidth}
    \begin{tikzpicture}[scale=0.75]
        \begin{axis}[
        cycle list name = myCycleList,
        grid = both,
        xmode = log, ymode = log,
        xmin = 0.8, xmax=1e5,
        ymin = 1e+-12, ymax= 1,
        legend pos=south west,
        yticklabels = {}, xticklabels = {}
        ]
        \addplot table [x index=0,y index=1, col sep=space] {nu_1.5_dim_2_pen_log_avg_10.dat};
        \addplot table [x index=0,y index=1, col sep=space] {nu_1.5_dim_3_pen_log_avg_10.dat};
        \addplot table [x index=0,y index=1, col sep=space] {nu_1.5_dim_5_pen_log_avg_10.dat};
        \addplot table [x index=0,y index=1, col sep=space] {nu_1.5_dim_10_pen_log_avg_10.dat};
        \addplot table [x index=0,y index=1, col sep=space] {nu_1.5_dim_25_pen_log_avg_10.dat};
        \addplot table [x index=0,y index=1, col sep=space] {nu_1.5_dim_50_pen_log_avg_10.dat};
        \addplot table [x index=0,y index=1, col sep=space] {nu_1.5_dim_100_pen_log_avg_10.dat};
        \addplot[dashed, very thick, black, domain=1:1e5, samples=100] {1/x};
        \legend{\(d=2\),\(d=3\),\(d=5\),\(d=10\),\(d=25\),\(d=50\),\(d=100\),\(N^{-1}\)};
        \end{axis}
        \node[rotate=-90] at (7.7,2.8) {\Large$\nu_j=1.5$};
    \end{tikzpicture}
    \end{subfigure}
    \begin{subfigure}{.45\textwidth}
    \begin{tikzpicture}[scale=0.75]
        \begin{axis}[
        cycle list name = myCycleList,
        grid = both,
        xmode = log, ymode = log,
        xmin = 0.8, xmax=1e5,
        ymin = 1e+-12, ymax= 1,
        xlabel={\Large$N$},
        legend pos=south west
        ]
        \addplot table [x index=0,y index=1, col sep=space] {nu_2.5_dim_2_pen_lin_avg_10.dat};
        \addplot table [x index=0,y index=1, col sep=space] {nu_2.5_dim_3_pen_lin_avg_10.dat};
        \addplot table [x index=0,y index=1, col sep=space] {nu_2.5_dim_5_pen_lin_avg_10.dat};
        \addplot table [x index=0,y index=1, col sep=space] {nu_2.5_dim_10_pen_lin_avg_10.dat};
        \addplot table [x index=0,y index=1, col sep=space] {nu_2.5_dim_25_pen_lin_avg_10.dat};
        \addplot table [x index=0,y index=1, col sep=space] {nu_2.5_dim_50_pen_lin_avg_10.dat};
        \addplot table [x index=0,y index=1, col sep=space] {nu_2.5_dim_100_pen_lin_avg_10.dat};
        \addplot[dashed, very thick, black, domain=1:1e5, samples=100] {1/x};
        \legend{\(d=2\),\(d=3\),\(d=5\),\(d=10\),\(d=25\),\(d=50\),\(d=100\),\(N^{-1}\)};
        \end{axis}
        \node[rotate=90] at (-1.3,2.8) {\large Approx. \(L^2\)-Error};
    \end{tikzpicture}
    \end{subfigure}
    \hspace{10mm}
    \begin{subfigure}{.45\textwidth}
    \begin{tikzpicture}[scale=0.75]
        \begin{axis}[
        cycle list name = myCycleList,
        grid = both,
        xmode = log, ymode = log,
        xmin = 0.8, xmax=1e5,
        ymin = 1e+-12, ymax= 1,
        xlabel={\Large$N$},
        legend pos=south west,
        yticklabels = {}
        ]
        \addplot table [x index=0,y index=1, col sep=space] {nu_2.5_dim_2_pen_log_avg_10.dat};
        \addplot table [x index=0,y index=1, col sep=space] {nu_2.5_dim_3_pen_log_avg_10.dat};
        \addplot table [x index=0,y index=1, col sep=space] {nu_2.5_dim_5_pen_log_avg_10.dat};
        \addplot table [x index=0,y index=1, col sep=space] {nu_2.5_dim_10_pen_log_avg_10.dat};
        \addplot table [x index=0,y index=1, col sep=space] {nu_2.5_dim_25_pen_log_avg_10.dat};
        \addplot table [x index=0,y index=1, col sep=space] {nu_2.5_dim_50_pen_log_avg_10.dat};
        \addplot table [x index=0,y index=1, col sep=space] {nu_2.5_dim_100_pen_log_avg_10.dat};
        \addplot[dashed, very thick, black, domain=1:1e5, samples=100] {1/x};
        \legend{\(d=2\),\(d=3\),\(d=5\),\(d=10\),\(d=25\),\(d=50\),\(d=100\),\(N^{-1}\)};
        \end{axis}
        \node[rotate=-90] at (7.7,2.8) {\Large$\nu_j=2.5$};
    \end{tikzpicture}
    \end{subfigure}
    \caption{Relative \(L^2\)-approximation error when interpolating target functions, \(f\), of increasing dimension, \(d\), using lengthscale-informed sparse grids.
    }
    \label{fig: LISG high dimensions}
\end{figure}

As discussed in Section~\ref{subsec: results}, in all-but-one of the six penalty and smoothness combinations, we observe the presence of two regimes: The pre-asymptotic regime driven by the penalty, \(\mathbf{p}\), and hence the lengthscale anisotropy, and the asymptotic regime governed by the smoothness, \(\nu_j\), and the dimensionality, \(d\). For smaller dimensions, there is a quick onset of the asymptotic behaviour, seen clearly for dimensions \(d=2,3\) and, in some instances, \(d=5\). For these dimensions, the asymptotic convergence rate is faster than the presymptomatic rate, and accordingly we see an increased convergence rate once \(L\) is large enough. This effect is diminished with increasing \(d\), and eventually we would expect a significant slow down. The exception to this pattern occurs for \(\nu_j=0.5\) and \(\mathbf{p}=\mathbf{p}_{\log}\), however we note that this behaviour is present for slightly faster, quadratic growth in the lengthscale, \(p_j=\lceil j^2\rceil\). As we increase the smoothness, the pre-asymptotic convergence rate increases, as expected. Notably, this increase is much larger from \(\nu_j =0.5\) to \(1.5\), than with \(\nu_j=1.5\) to \(2.5\). We also observe an increased convergence rate when employing linearly, instead of logarithmically, growing penalties, however the difference is small relative to the additional anisotropy present. 
\begin{figure}[h]
    \begin{subfigure}{.45\textwidth}
    \begin{tikzpicture}[scale=0.75]
        \begin{axis}[
        cycle list name = myCycleList,
        xmode = log, ymode = log,
        xmin = 0.8, xmax=1e4,
        ymin = 1e+-5, ymax= 1,
        legend pos=south west,
        xticklabels={},
        grid=both
        ]
        \addplot[color=vermillion, very thick, dotted, mark=diamond*, 
            mark options={draw=vermillion, solid}]
            table [x index=0, y index=1, col sep=space] 
            {nu_0.5_dim_10_pert_-1_pen_lin_avg_10.dat};
        
        \addplot[color=vermillion, thick, dashed, mark=triangle*, 
            mark options={draw=vermillion, solid}]
            table [x index=0, y index=1, col sep=space] 
            {nu_0.5_dim_10_pert_-0.5_pen_lin_avg_10.dat};
    
        \addplot[color=vermillion, thick, solid, mark=*]
            table [x index=0, y index=1, col sep=space] 
            {nu_0.5_dim_10_pert_-0.2_pen_lin_avg_10.dat};
    
        \addplot[color=darkgray, thick, solid, mark=square*]
            table [x index=0, y index=1, col sep=space] 
            {nu_0.5_dim_10_pert_0_pen_lin_avg_10.dat};
    
        \addplot[color=skyblue, thick, solid, mark=*]
            table [x index=0, y index=1, col sep=space] 
            {nu_0.5_dim_10_pert_0.2_pen_lin_avg_10.dat};
    
        \addplot[color=skyblue, thick, dashed, mark=triangle*, 
            mark options={draw=skyblue, solid}]
            table [x index=0, y index=1, col sep=space] 
            {nu_0.5_dim_10_pert_0.5_pen_lin_avg_10.dat};
    
        \addplot[color=skyblue, very thick, dotted, mark=diamond*,
            mark options={draw=skyblue, solid}]
            table [x index=0, y index=1, col sep=space] 
            {nu_0.5_dim_10_pert_1_pen_lin_avg_10.dat};
    
        \legend{
            $\eta = -1$,
            $\eta = -0.5$,
            $\eta = -0.2$,
            $\eta = 0$,
            $\eta = 0.2$,
            $\eta = 0.5$,
            $\eta = 1$
        };
        \end{axis}
        \node at (3.5,6.5) {\Large$\nu_j=0.5$};
        \node[rotate=90] at (-1.3,2.8) {\large Approx. \(L^2\)-Error};
    \end{tikzpicture}
    \end{subfigure}
    \hspace{10mm}
    \begin{subfigure}{.45\textwidth}
    \begin{tikzpicture}[scale=0.75]
        \begin{axis}[
        cycle list name = myCycleList,
        xmode = log, ymode = log,
        xmin = 0.8, xmax=1e4,
        ymin = 1e+-5, ymax=1,
        legend pos=south west,
        xticklabels={}, yticklabels={},
        grid=both
        ]
        \addplot[color=vermillion, very thick, dotted, mark=diamond*, 
            mark options={draw=vermillion, solid}]
            table [x index=0, y index=1, col sep=space] 
            {nu_1.5_dim_10_pert_-1_pen_lin_avg_10.dat};
            
        \addplot[color=vermillion, thick, dashed, mark=triangle*, 
            mark options={draw=vermillion, solid}]
            table [x index=0, y index=1, col sep=space] 
            {nu_1.5_dim_10_pert_-0.5_pen_lin_avg_10.dat};
    
        \addplot[color=vermillion, thick, solid, mark=*]
            table [x index=0, y index=1, col sep=space] 
            {nu_1.5_dim_10_pert_-0.2_pen_lin_avg_10.dat};
    
        \addplot[color=darkgray, thick, solid, mark=square*]
            table [x index=0, y index=1, col sep=space] 
            {nu_1.5_dim_10_pert_0_pen_lin_avg_10.dat};
    
        \addplot[color=skyblue, thick, solid, mark=*]
            table [x index=0, y index=1, col sep=space] 
            {nu_1.5_dim_10_pert_0.2_pen_lin_avg_10.dat};
    
        \addplot[color=skyblue, thick, dashed, mark=triangle*, 
            mark options={draw=skyblue, solid}]
            table [x index=0, y index=1, col sep=space] 
            {nu_1.5_dim_10_pert_0.5_pen_lin_avg_10.dat};
    
        \addplot[color=skyblue, very thick, dotted, mark=diamond*,
            mark options={draw=skyblue, solid}]
            table [x index=0, y index=1, col sep=space] 
            {nu_1.5_dim_10_pert_1_pen_lin_avg_10.dat};
    
        \legend{
            $\eta = -1$,
            $\eta = -0.5$,
            $\eta = -0.2$,
            $\eta = 0$,
            $\eta = 0.2$,
            $\eta = 0.5$,
            $\eta = 1$
        };
        \end{axis}
        \node at (3.5,6.5) {\Large$\nu_j=1.5$};
        \node[rotate=-90] at (7.7,2.8) {\Large$d=10$};
    \end{tikzpicture}
    \end{subfigure}
    \begin{subfigure}{.45\textwidth}
    \begin{tikzpicture}[scale=0.75]
        \begin{axis}[
        xmode = log, ymode = log,
        xmin = 0.8, xmax=1e4,
        ymin = 1e+-5, ymax= 1,
        legend pos=south west,
        xlabel={\Large$N$},
        grid=both
        ]
        \addplot[color=vermillion, thick, dashed, mark=triangle*, 
            mark options={draw=vermillion, solid}]
            table [x index=0, y index=1, col sep=space] 
            {nu_0.5_dim_50_pert_-0.5_pen_lin_avg_10.dat};
    
        \addplot[color=vermillion, thick, solid, mark=*]
            table [x index=0, y index=1, col sep=space] 
            {nu_0.5_dim_50_pert_-0.2_pen_lin_avg_10.dat};
    
        \addplot[color=darkgray, thick, solid, mark=square*]
            table [x index=0, y index=1, col sep=space] 
            {nu_0.5_dim_50_pert_0_pen_lin_avg_10.dat};
    
        \addplot[color=skyblue, thick, solid, mark=*]
            table [x index=0, y index=1, col sep=space] 
            {nu_0.5_dim_50_pert_0.2_pen_lin_avg_10.dat};
    
        \addplot[color=skyblue, thick, dashed, mark=triangle*, 
            mark options={draw=skyblue, solid}]
            table [x index=0, y index=1, col sep=space] 
            {nu_0.5_dim_50_pert_0.5_pen_lin_avg_10.dat};
    
        \addplot[color=skyblue, very thick, dotted, mark=diamond*,
            mark options={draw=skyblue, solid}]
            table [x index=0, y index=1, col sep=space] 
            {nu_0.5_dim_50_pert_1_pen_lin_avg_10.dat};
    
        \legend{
            $\eta = -0.5$,
            $\eta = -0.2$,
            $\eta = 0$,
            $\eta = 0.2$,
            $\eta = 0.5$,
            $\eta = 1$
        };
        \end{axis}
        \node[rotate=90] at (-1.3,2.8) {\large Approx. \(L^2\)-Error};
    \end{tikzpicture}
    \end{subfigure}
    \hspace{10mm}
    \begin{subfigure}{.45\textwidth}
    \begin{tikzpicture}[scale=0.75]
        \begin{axis}[
        xmode = log, ymode = log,
        xmin = 0.8, xmax=1e4,
        ymin = 1e+-5, ymax= 1,
        legend pos=south west,
        yticklabels={},
        xlabel={\Large$N$},
        grid=both
        ]
        \addplot[color=vermillion, thick, dashed, mark=triangle*, 
            mark options={draw=vermillion, solid}]
            table [x index=0, y index=1, col sep=space] 
            {nu_1.5_dim_50_pert_-0.5_pen_lin_avg_10.dat};
    
        \addplot[color=vermillion, thick, solid, mark=*]
            table [x index=0, y index=1, col sep=space] 
            {nu_1.5_dim_50_pert_-0.2_pen_lin_avg_10.dat};
    
        \addplot[color=darkgray, thick, solid, mark=square*]
            table [x index=0, y index=1, col sep=space] 
            {nu_1.5_dim_50_pert_0_pen_lin_avg_10.dat};
    
        \addplot[color=skyblue, thick, solid, mark=*]
            table [x index=0, y index=1, col sep=space] 
            {nu_1.5_dim_50_pert_0.2_pen_lin_avg_10.dat};
    
        \addplot[color=skyblue, thick, dashed, mark=triangle*, 
            mark options={draw=skyblue, solid}]
            table [x index=0, y index=1, col sep=space] 
            {nu_1.5_dim_50_pert_0.5_pen_lin_avg_10.dat};
    
        \addplot[color=skyblue, very thick, dotted, mark=diamond*,
            mark options={draw=skyblue, solid}]
            table [x index=0, y index=1, col sep=space] 
            {nu_1.5_dim_50_pert_1_pen_lin_avg_10.dat};
    
        \legend{
            $\eta = -0.5$,
            $\eta = -0.2$,
            $\eta = 0$,
            $\eta = 0.2$,
            $\eta = 0.5$,
            $\eta = 1$
        };
        \end{axis}
        \node[rotate=-90] at (7.7,2.8) {\Large$d=50$};
    \end{tikzpicture}
    \end{subfigure}
    \caption{Effect of misspecifying lengthscale anisotropy on the convergence of the relative \(L^2\)-approximation error in \(N\), the number of function evaluations.}
    \label{fig: LISG misspecified}
\end{figure}
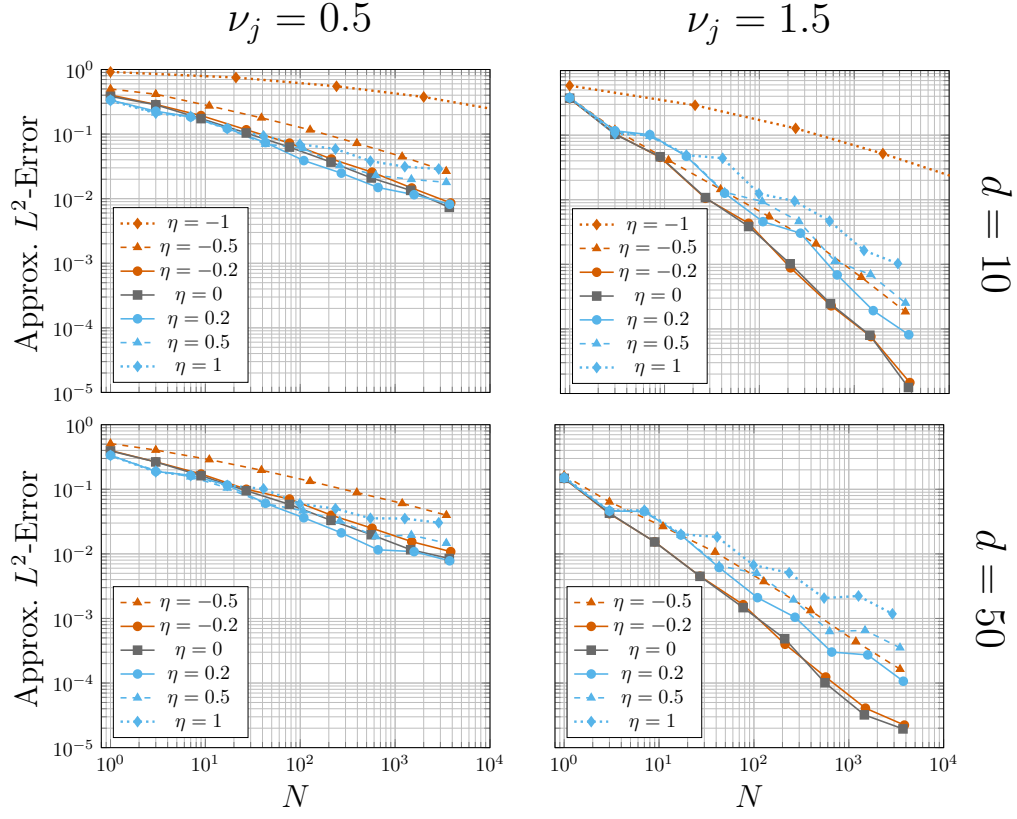
A limitation of the error bound in Theorem \ref{thm: error in L} is that we require lengthscales of the form \(\lambda_j=2^p\), \(p\in\mathbb{N}_0\). Furthermore, as discussed in Section \ref{subsec: lengthscale anisotropy}, lengthscales are rarely known \emph{a priori} in practice, and instead are more likely to be estimated in many applications. Hence, to test the robustness of the method, in Figure \ref{fig: LISG misspecified} we explore the effect of misspecifying the penalty. We take realisations of \(f\) with linearly growing penalties, \(\mathbf{p}=\mathbf{p}_{\textrm{lin}}\), and approximate using LISG designs, with separable Mat\'ern kernels now stretched according to a perturbed penalty, \(\mathbf{p}^{\eta}_{\textrm{lin}}\), defined in terms of a perturbation parameter, \(\eta\), by \(p_{\textrm{lin},j}^{\eta}=\lceil(1+\eta)p_j\rceil\). 

Predictably, we observe in all cases that misspecifying the penalty by a small degree only has a marginal effect on the error, whereas a large discrepancy between the anisotropy in \(f\) and the kernels results in a significantly worse error. We note that all of these cases perform better than isotropic sparse grids (equivalent to \(\eta=-1\)), which are infeasible in \(d=50\) due to their size. Interestingly, for \(\nu_j=0.5\), we see a small improvement when slightly overestimating the lengthscale according to \(\eta=0.2\). This is not true for \(\nu_j=1.5\), where instead underestimating the lengthscale according to \(\eta=-0.2\) results in a similar error as when using the exact penalties. 

Finally, in Figure \ref{fig: gp post var}, we present heatmaps representing the posterior marginal variance of two-dimensional Gaussian processes, whose mean functions correspond to kernel interpolants on LISG, SG and MC designs. We note that only the LISG method employs anisotropic kernels, which accounts for most of the discrepancy to the other two methods.

\begin{figure}[hp]
    \begin{subfigure}{.45\textwidth}
    \begin{tikzpicture}[scale=0.75]
        \begin{axis}[
            ylabel={\LARGE$x_2$},
            xticklabels={},
            xmin=-0.5, xmax=0.5,
            ymin=-0.5, ymax=0.5,
            axis on top,
            colorbar,
            colormap name=cividis,
            point meta min=0,
            point meta max=0.3,
            axis equal image
          ]
            \addplot graphics [
              includegraphics={width=\axiswidth,height=\axisheight},
              xmin=-0.5, xmax=0.5,
              ymin=-0.5, ymax=0.5,
            ] {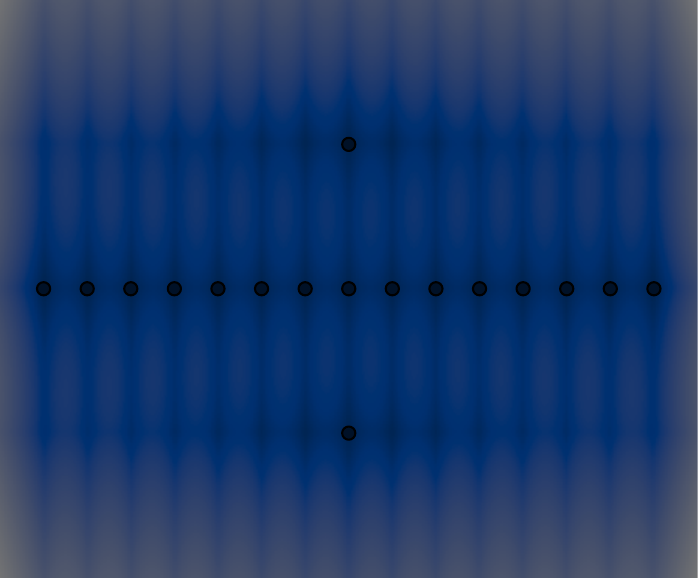};
        \end{axis}
        \node at (3.2,6.5) {\Large$\nu_j=0.5$};
    \end{tikzpicture}
    \end{subfigure}
    \hspace{10mm}
    \begin{subfigure}{.45\textwidth}
    \begin{tikzpicture}[scale=0.75]
        \begin{axis}[
            xticklabels={}, yticklabels={},
            xmin=-0.5, xmax=0.5,
            ymin=-0.5, ymax=0.5,
            axis on top, 
            colorbar,
            colorbar style={
                /pgf/number format/fixed,
                /pgf/number format/precision=3,
                scaled ticks=false
            },
            colormap name=cividis,
            point meta min=0,
            point meta max=0.01,
            axis equal image
          ]
            \addplot graphics [
              includegraphics={width=\axiswidth,height=\axisheight},
              xmin=-0.5, xmax=0.5,
              ymin=-0.5, ymax=0.5,
            ] {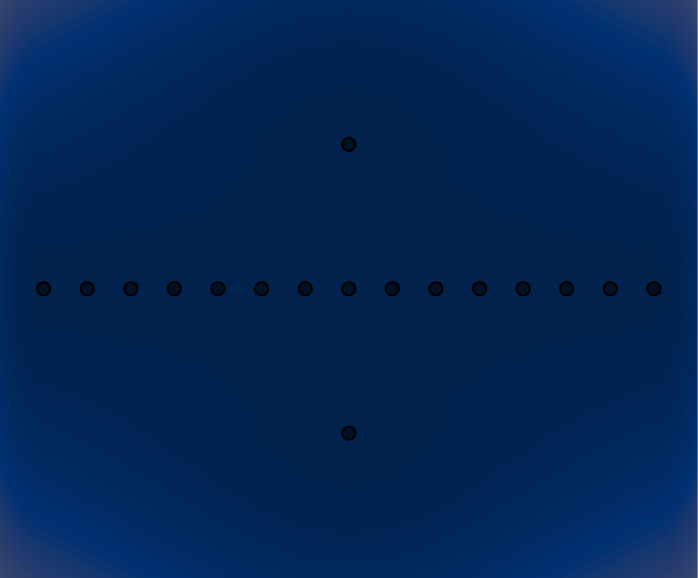};
        \end{axis}
        \node at (3,6.5) {\Large$\nu_j=1.5$};
        \node[rotate=-90] at (8,2.8) {\Large LISG};
    \end{tikzpicture}
    \end{subfigure}
    \begin{subfigure}{.45\textwidth}
    \begin{tikzpicture}[scale=0.75]
        \begin{axis}[
            ylabel={\LARGE$x_2$},
            xticklabels={},
            xmin=-0.5, xmax=0.5,
            ymin=-0.5, ymax=0.5,
            axis on top,
            colorbar,
            colormap name=cividis,
            point meta min=0,
            point meta max=0.3,
            axis equal image
          ]
            \addplot graphics [
              includegraphics={width=\axiswidth,height=\axisheight},
              xmin=-0.5, xmax=0.5,
              ymin=-0.5, ymax=0.5,
            ] {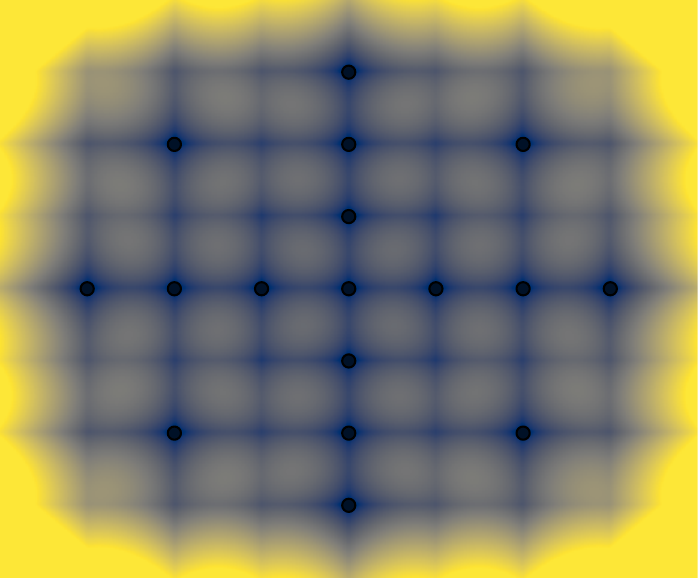};
        \end{axis}
    \end{tikzpicture}
    \end{subfigure}
    \hspace{10mm}
    \begin{subfigure}{.45\textwidth}
    \begin{tikzpicture}[scale=0.75]
        \begin{axis}[
            xticklabels={}, yticklabels={},
            xmin=-0.5, xmax=0.5,
            ymin=-0.5, ymax=0.5,
            axis on top, 
            colorbar,
            colorbar style={
                /pgf/number format/fixed,
                /pgf/number format/precision=3,
                scaled ticks=false
            },
            colormap name=cividis,
            point meta min=0,
            point meta max=0.01,
            axis equal image
          ]
            \addplot graphics [
              includegraphics={width=\axiswidth,height=\axisheight},
              xmin=-0.5, xmax=0.5,
              ymin=-0.5, ymax=0.5,
            ] {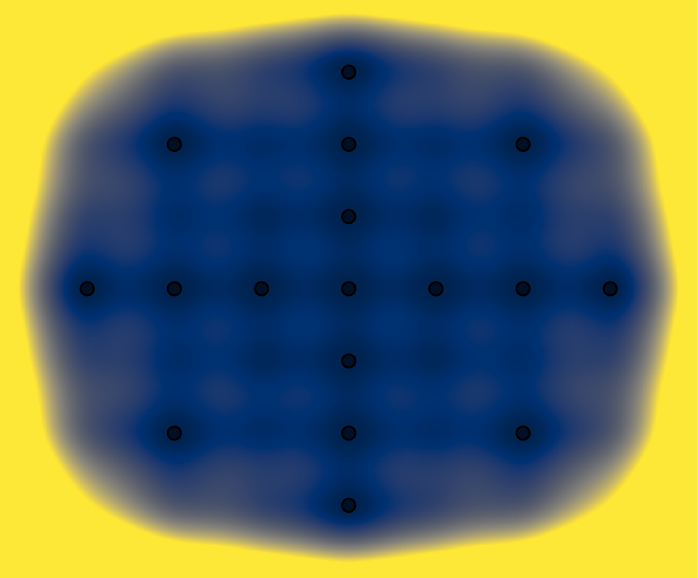};
        \end{axis}
        \node[rotate=-90] at (8,2.8) {\Large SG};
    \end{tikzpicture}
    \end{subfigure}
    \begin{subfigure}{.45\textwidth}
    \begin{tikzpicture}[scale=0.75]
        \begin{axis}[
            xlabel={\LARGE$x_1$},
            ylabel={\LARGE$x_2$},
            xmin=-0.5, xmax=0.5,
            ymin=-0.5, ymax=0.5,
            axis on top,
            colorbar,
            colormap name=cividis,
            point meta min=0,
            point meta max=0.3,
            axis equal image
          ]
            \addplot graphics [
              includegraphics={width=\axiswidth,height=\axisheight},
              xmin=-0.5, xmax=0.5,
              ymin=-0.5, ymax=0.5,
            ] {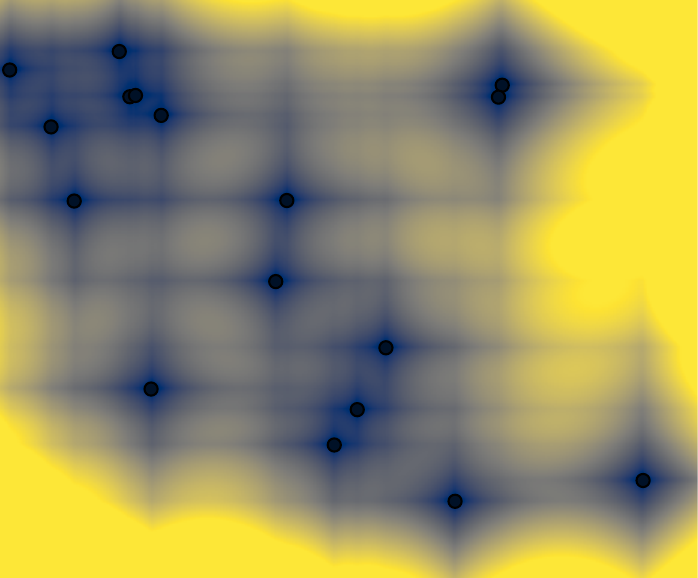};
        \end{axis}
    \end{tikzpicture}
    \end{subfigure}
    \hspace{10mm}
    \begin{subfigure}{.45\textwidth}
    \begin{tikzpicture}[scale=0.75]
        \begin{axis}[
            yticklabels={},
            xlabel={\LARGE$x_1$},
            xmin=-0.5, xmax=0.5,
            ymin=-0.5, ymax=0.5,
            axis on top, 
            colorbar,
            colorbar style={
                /pgf/number format/fixed,
                /pgf/number format/precision=3,
                scaled ticks=false
            },
            colormap name=cividis,
            point meta min=0,
            point meta max=0.01,
            axis equal image
          ]
            \addplot graphics [
              includegraphics={width=\axiswidth,height=\axisheight},
              xmin=-0.5, xmax=0.5,
              ymin=-0.5, ymax=0.5,
            ] {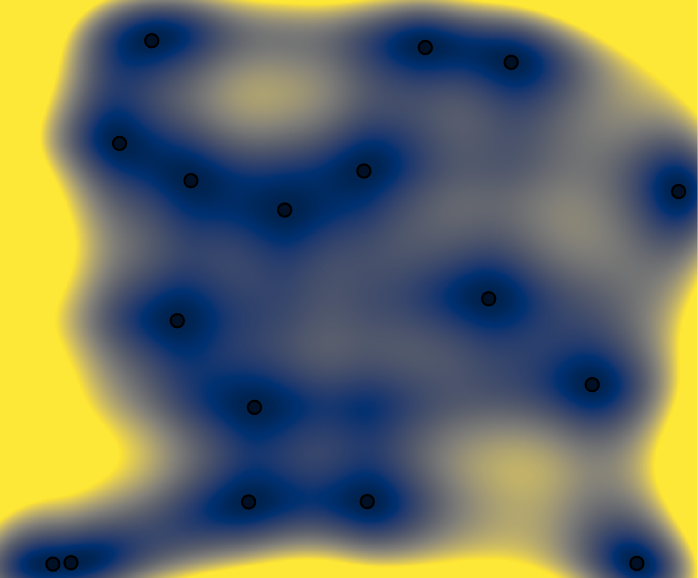};
        \end{axis}
        \node[rotate=-90] at (8,2.8) {\Large MC};
    \end{tikzpicture}
    \end{subfigure}
    \caption{Heat map displaying the posterior variance of a two-dimensional Gaussian process trained on approximately 17 function evaluations arranged in different designs. For the LISG method, a penalty of \(\mathbf{p}=(1,3)\) is assumed. Values exceeding the maximum are displayed as the maximum.}
    \label{fig: gp post var}
\end{figure}
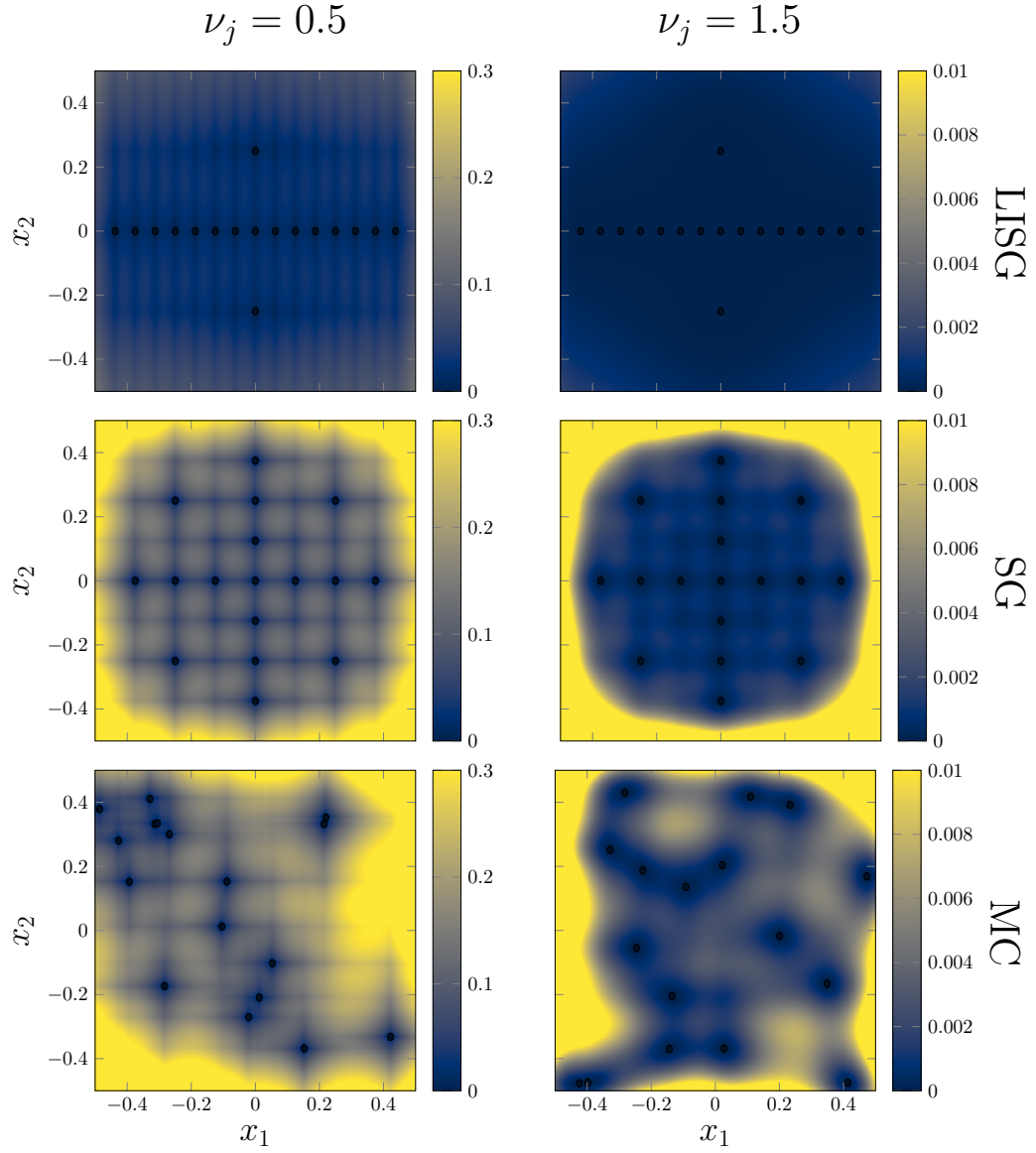

\FloatBarrier
\section{Conclusion}
We have presented lengthscale-informed sparse grids (LISGs); a novel sparse grid construction for interpolation with separable Mat\'ern kernels, 
designed to exploit lengthscale-anisotropy in the target function \(f\). For sparse grid approximation methods to apply, we assume \(f\) lies in a Sobolev space with bounded mixed smoothness, and to be considered anisotropic, \(f\) is assumed to be bounded independently of the dimension in the native space norm of a separable Mat\'ern kernel with the prescribed lengthscales. In Theorem \ref{thm: error in L}, we derived a general error bound indicating the superior performance of LISG approximation over isotropic methods given the presence of anisotropy, especially in a high-dimensional setting. Making the connection to Gaussian process regression, we additionally derive an analogous error bound for the posterior marginal variance in Proposition \ref{prop: vairance convergence}. We included an adaptation to the fast inference algorithm from \cite{Plumlee2014} to LISG, resulting in an algorithm that scales well with both input dimension and the number of function evaluations. We tested the method on a series of numerical experiments on target functions with known anisotropy in up to 100 dimensions.

To the best of our knowledge, this is the first result proving error bounds in the high-dimensional setting assuming
only anisotropy in the lengthscales, not requiring anisotropic (or infinite) regularity. Although we in this work have assumed that the anisotropy structure, i.e. the lengthscales $\{\lambda_j\}_{j=1}^d$, is known \emph{a priori}, we expect that our method will perform equally well in practice when lengthscales have to be estimated, provided that that target function exhibits the required anisotropy. Our numerical experiments with misspecified lengthscales indicate a level of robustness to errors in the estimation of the lengthscales. Furthermore, since a typical workflow already involves the estimation of lengthscales, the application of our method is not much more cumbersome than that of isotropic sparse grids.

\section*{Acknowledgements}
The authors would like to thank the Isaac Newton Institute for Mathematical Sciences, Cambridge, for support and hospitality during the programme \emph{Mathematical and statistical foundation of future data-driven engineering}, where work on this paper was undertaken, supported by EPSRC (EP/R014604/1). The authors further thank Abdul-Lateef Haji-Ali and Ken Newman for helpful discussions and support. EA was supported by Biomathematics and Statistics Scotland (BioSS) and the EPSRC Centre for Doctoral Training in Mathematical Modelling, Analysis and Computation (MAC-MIGS) funded by the EPSRC (EP/S023291/1), Heriot-Watt University, and The University of Edinburgh. ALT was partially supported by
EPSRC grants EP/X01259X/1 and EP/Y028783/1.

\begin{appendices}

\section{}\label{appendix: a}

This section contains results related to the isomorphism of multi-index sets used in the proofs in Section \ref{sec: results}.

\begin{lemma}\label{lem: bijection}
    Let \(d\in\mathbb{N}\) and \(k\in\mathbb{N}_0\). Define subsets of \(\mathbb{N}_0^d\) by \(\mathcal{K}^d_k\coloneqq\{\boldsymbol{l}\in\mathbb{N}_{0}^d:|\{1\leq j\leq d:l_j\neq0\}|=k\}\) and \(\mathcal{P}^d_k\coloneqq\{\mathfrak{u}\subset\{1,\dots,d\}\,:\,|\mathfrak{u}|=k, \mathfrak{u}_i<\mathfrak{u}_{i+1}\}\). We have that \(\mathcal{K}_k^d\cong\{\mathbf{0}\}\cup\mathcal{P}^d_k\times\mathbb{N}^k\) with bijection, for a given \(\boldsymbol{l}\in\mathbb{N}_0^d\),
    \begin{align}
        \boldsymbol{\rho}(\boldsymbol{l}) = \begin{cases}
            \mathbf{0}\in\mathbb{N}_0^k&\textrm{if }\boldsymbol{l}=\mathbf{0}, \textrm{ and}\\
            \left(\mathfrak{u}(\boldsymbol{l}),\boldsymbol{l}_{\mathfrak{u}(\boldsymbol{l})}\right)&\textrm{otherwise},
            \end{cases}\nonumber
    \end{align}
    where \(\mathfrak{u}(\boldsymbol{l})\coloneqq\{j\in\{1,\dots,d\}:l_j=0\}\subset\mathcal{P}^d_k\) and, if \(|\mathfrak{u}(\boldsymbol{l})|=k\), we define, for a given \(\mathbf{b}\in\mathbb{N}_0^d\), \(\mathbf{b}_{\mathfrak{u}(\boldsymbol{l})}\coloneqq(b_{\mathfrak{u}(\boldsymbol{l})_1},\dots,b_{\mathfrak{u}(\boldsymbol{l})_k})\).
\end{lemma}

\begin{proof} Define the candidate inverse of \(\boldsymbol{\rho}\) by \(\boldsymbol{\rho}^{-1}(\mathbf{0})=\mathbf{0}\in\mathbb{N}_0^d\), and for \((\mathfrak{u},\mathbf{a})\in\mathcal{P}^d_k\times\mathbb{N}^k\),
\begin{align}
    \boldsymbol{\rho}^{-1}(\mathfrak{u},\mathbf{a})_{\mathfrak{u}_i}=
        a_i \textrm{ for }1\leq i\leq k\quad\textrm{and}\quad\boldsymbol{\rho}^{-1}(\mathfrak{u},\mathbf{a})_j=0\textrm{ for }j\notin\mathfrak{u}.\nonumber
\end{align}
Let \(\boldsymbol{l}\in\mathcal{K}^d_k\). If \(\boldsymbol{l}=\mathbf{0}\), then trivially \(\boldsymbol{\rho}^{-1}(\boldsymbol{\rho}(\mathbf{0}))=\mathbf{0}\). Otherwise, for \(1\leq i\leq k\), we have
\begin{align}
    \boldsymbol{\rho}^{-1}(\boldsymbol{\rho}(\boldsymbol{l}))_{\mathfrak{u}(\boldsymbol{l})_i}&=\boldsymbol{\rho}^{-1}(\mathfrak{u}(\boldsymbol{l}),\boldsymbol{l}_{\mathfrak{u}(\boldsymbol{l})})_{\mathfrak{u}(\boldsymbol{l})_i}=l_{\mathfrak{u}(\boldsymbol{l})_i},\nonumber
\end{align}
and for all \(j\notin\mathfrak{u}(\boldsymbol{l})\), by definition \(\boldsymbol{l}_j=0\), and
\begin{align}
     \boldsymbol{\rho}^{-1}(\boldsymbol{\rho}(\boldsymbol{l}))_j&=\boldsymbol{\rho}^{-1}(\mathfrak{u}(\boldsymbol{l}),\boldsymbol{l}_{\mathfrak{u}(\boldsymbol{l})})_j=0.\nonumber
\end{align}
Therefore, for all \(1\leq j\leq d\), \(\boldsymbol{\rho}^{-1}(\boldsymbol{\rho}(\boldsymbol{l}))_j=l_j\), and so \(\boldsymbol{\rho}^{-1}(\boldsymbol{\rho}(\boldsymbol{l}))=\boldsymbol{l}\).

In the opposite direction, trivially we have \(\boldsymbol{\rho}(\boldsymbol{\rho}^{-1}(\mathbf{0}))=\mathbf{0}\). Otherwise, let \((\mathfrak{u},\mathbf{a})\in\mathcal{P}^d_k\times\mathbb{N}^k\). For \(1\leq i\leq k\), we have \(\boldsymbol{\rho}^{-1}(\mathfrak{u},\mathbf{a})_{\mathfrak{u}_i}= a_i\in\mathbb{N}\). Specifically, \(\boldsymbol{\rho}^{-1}(\mathfrak{u},\mathbf{a})_{\mathfrak{u}_i}\neq0\), and so by definition, \(\mathfrak{u}_i\in\mathfrak{u}(\boldsymbol{\rho}^{-1}(\mathfrak{u},\mathbf{a}))\). Furthermore, if \(j\notin\mathfrak{u}\), we have \(\boldsymbol{\rho}^{-1}(\mathfrak{u},\mathbf{a})_j=0\), and so \(j\notin\mathfrak{u}(\boldsymbol{\rho}^{-1}(\mathfrak{u},\mathbf{a}))\). Hence,
\begin{align}
    \mathfrak{u}(\boldsymbol{\rho}^{-1}(\mathfrak{u},\mathbf{a}))=\{\mathfrak{u}_i\}_{1\leq i \leq k}=\mathfrak{u}.\nonumber
\end{align}
Finally, we then have
\begin{align}
    \begin{split}
    \boldsymbol{\rho}^{-1}(\mathfrak{u},\mathbf{a})_{\mathfrak{u}(\boldsymbol{\rho}^{-1}(\mathfrak{u},\mathbf{a}))}&=\boldsymbol{\rho}^{-1}(\mathfrak{u},\mathbf{a})_{\mathfrak{u}},\\
    &\coloneqq\left(\boldsymbol{\rho}^{-1}(\mathfrak{u},\mathbf{a})_{\mathfrak{u}_1},\dots,\boldsymbol{\rho}^{-1}(\mathfrak{u},\mathbf{a})_{\mathfrak{u}_k}\right),\\
    &=(a_{1},\dots,a_{k}), \\
    &= \mathbf{a}.
    \end{split}\nonumber
\end{align}
Therefore,
\begin{align}
    \begin{split}
    \boldsymbol{\rho}(\boldsymbol{\rho}^{-1}(\mathfrak{u},\mathbf{a}))&\coloneqq\left(\mathfrak{u}(\boldsymbol{\rho}^{-1}(\mathfrak{u},\mathbf{a})),\boldsymbol{\rho}^{-1}(\mathfrak{u},\mathbf{a})_{\mathfrak{u}(\boldsymbol{\rho}^{-1}(\mathfrak{u},\mathbf{a}))}\right),\\
    &=(\mathfrak{u},\mathbf{a}),\nonumber
    \end{split}
\end{align}
as required.
\end{proof}

\begin{corollary}\label{cor: level sets}
    Let \(d,k\in\mathbb{N}\). Let \(\boldsymbol{\rho}:\mathcal{K}_k^d \leftrightarrow\{\mathbf{0}\}\cup\mathcal{P}^d_k\times\mathbb{N}^k\) be defined as in Lemma \ref{lem: bijection} and let \(\mathbf{0}\neq\boldsymbol{l}\in\mathcal{K}_k^d\), \(\boldsymbol{\rho}(\boldsymbol{l})=(\mathfrak{u},\mathbf{a})\). For \(L\in\mathbb{N}_0\), define the multi-index set \(\mathcal{I}_L^d\coloneqq\{\boldsymbol{l}\in\mathbb{N}_0^d:|\boldsymbol{l}|\leq L\}\). Then \(\boldsymbol{l}\in\mathcal{I}_L^d\) if and only if \(\mathbf{a}\in\mathcal{I}_L^k\).
\end{corollary}

\begin{proof}
Let \(\boldsymbol{l}\in \mathcal{I}^d_L\), then \(\sum_{j=1}^dl_j\leq L\).
By definition, \(\mathbf{a}=\boldsymbol{l}_{\mathfrak{u}(\boldsymbol{l})}\), and hence \(\sum_{i=1}^ka_i=\sum_{i=1}^kl_{\mathfrak{u}(\boldsymbol{l})_i}\leq\sum_{j=1}^dl_j\leq L\), since \(l_j\geq0\). Thus \(\mathbf{a}\in\mathcal{I}^k_L\).

Let \(\mathbf{a}\in\mathcal{I}^k_L\). By definition, \(l_j=0\) if \(j\notin\mathfrak{u}(\boldsymbol{l})\). Since \(\boldsymbol{l}\in\mathcal{K}^d_k\), we have \(|\mathfrak{u}(\boldsymbol{l})|=k\), and hence \(\sum_{j=1}^dl_j=\sum_{i=1}^kl_{\mathfrak{u}(\boldsymbol{l})_i}=\sum_{i=1}^ka_i\leq L\). Thus \(\boldsymbol{l}\in\mathcal{I}^d_L\) also.
\end{proof}

\begin{corollary}\label{cor: isomporphic}
    Let \(d,k\in\mathbb{N}\). Let \(\mathcal{K}_k^d\), \(\mathcal{P}_k^d\) and \(\mathcal{I}_L^d\) be defined as in Lemma~\ref{lem: bijection} and Corollary~\ref{cor: level sets}. Then \(\mathcal{K}_k^d\cap\mathcal{I}_L^d\cong\{\mathbf{0}\}\cup\mathcal{P}^d_k\times\mathbb{N}^k\cap\mathcal{I}_L^k\) and \(\mathcal{K}_k^d\setminus\mathcal{I}_L^d\cong\mathcal{P}^d_k\times\mathbb{N}^k\setminus\mathcal{I}_L^k\).
\end{corollary}

\begin{proof}
    This follows directly from the two previous results as \(\boldsymbol{\rho}\) acts as a bijection in both cases, noting that \(\mathbf{0}\in\mathcal{I}^d_{L}\) for all \(L\geq0\).
\end{proof}




\end{appendices}


\bibliography{sn-bibliography}

\end{document}